\documentclass[10pt]{article}
\usepackage
    {inputenc}
\usepackage[english]{babel}

\usepackage{amsmath,amsfonts,amssymb,amsthm,mathtools}

\usepackage[a4paper, left = 3.3cm, right= 3.3cm, top= 3cm, bottom = 3.3cm, heightrounded]{geometry}
\usepackage[colorlinks=true, allcolors=blue]{hyperref}

\usepackage{graphicx}
\usepackage{caption}
\usepackage{subcaption}
\usepackage{enumerate}
\usepackage{comment}

\usepackage[style=alphabetic, backend=bibtex, doi=false,isbn=false,url=false]{biblatex} 
\usepackage{csquotes}
\newtheorem{theorem}{Theorem}
\newtheorem{proposition}{Proposition}
\newtheorem{lemma}{Lemma}

\newtheorem{remark}{Remark}

\newcommand{\Cbb}{{\mathbb{C}}}
\newcommand{\Ebb}{{\mathbb{E}}}

\newcommand{\Nbb}{{\mathbb{N}}}
\newcommand{\Pbb}{{\mathbb{P}}}
\newcommand{\Qbb}{{\mathbb{Q}}}
\newcommand{\Rbb}{{\mathbb{R}}}
\newcommand{\Zbb}{{\mathbb{Z}}}
\newcommand{\Fcal}{{\mathcal{F}}}
\newcommand{\Hcal}{{\mathcal{H}}}
\newcommand{\Kcal}{{\mathcal{K}}}
\newcommand{\Lcal}{{\mathcal{L}}}
\newcommand{\Mcal}{{\mathcal{M}}}
\newcommand{\Ocal}{{\mathcal{O}}}
\newcommand{\nuest}{{\nu_\text{est}}}

\newcommand{\Acal}{{\mathcal{A}}}
\newcommand{\Bcal}{{\mathcal{B}}}

\newcommand{\Pcal}{{\mathcal{P}}}
\newcommand{\Qcal}{{\mathcal{Q}}}
\newcommand{\Zcal}{{\mathcal{Z}}}

\DeclareMathOperator{\argmax}{{\text{arg max}}}
\DeclareMathOperator{\argmin}{{\text{arg min}}}
\DeclareMathOperator{\Diam}{{\text{Diam}}}
\DeclareMathOperator{\Int}{{\text{relint}}}
\DeclareMathOperator{\Jac}{{\text{Jac}}}

\renewcommand{\geq}{\geqslant}
\renewcommand{\leq}{\leqslant}

\title{Support and distribution inference from noisy data}

\author{J\'er\'emie Capitao-Miniconi$^*$, \'Elisabeth Gassiat$^*$ and Luc Leh\'ericy$^{**}$\\
{\small $^*$ Universit\'e Paris-Saclay, CNRS, Laboratoire de math\'ematiques d'Orsay, 91405, Orsay, France,}\\
{\small $^{**}$ Universit\'e Côte d’Azur, CNRS, LJAD, France.}}

\date{}

\begin{document}
\maketitle

\begin{abstract}
Understanding what kind of noise in the observations allow to recover low dimensional structures of a signal is of interest in statistical learning, as a first step to build efficient dimension reduction procedures. In this work we give a new contribution on the type of noise which can affect the data without preventing to build consistent estimators of the support and distribution of the signal. We focus on the situation where the observations are corrupted with additive and independent noise. We prove that for general classes of supports, it is possible to recover both the support and the distribution of the signal without   knowing the noise distribution and with no sample of the noise. We exhibit classes of distributions over which we prove adaptive minimax rates (up to a $\log \log$ factor) for the estimation of the support in Hausdorff distance. Moreover, for the class of distributions with compact support, we provide estimators of the unknown (in general singular) distribution and prove maximum rates in Wasserstein distance. We also prove an almost matching lower bound on the associated minimax risk.
\end{abstract}

\section{Introduction}
\label{sec:intro}

\subsection{Context and aim}

It is a common observation that high dimensional data has a low intrinsic dimension. 
The computational geometry point of view gave rise to a number of interesting algorithms (see~\cite{MR3837127} and references therein) for the reconstruction of a non linear shape from a point cloud, 
and in the statistical community, past years have seen increasing interest for manifold estimation. The case of non noisy data, that is when  the observations are sampled on the unknown manifold,  is by now relatively well understood. When  the loss is measured using the Hausdorff distance, minimax rates for manifold estimation are known and have been proved  recently.
The rates depend on the intrinsic dimension of the manifold and  differ when the manifold has a boundary or does not have a boundary, due to the particular way  points accumulate near boundaries (see~\cite{AaronBoundary} for the most recent results, together with an overview of the subject and references).
When considering the estimation of a distribution with unknown non linear low dimensional support, one has to choose a loss function. 
The Wasserstein distance allows to compare distributions that can be mutually singular, and is thus useful to compare distributions having possibly different supports. 
Moreover, approximating an unknown probability distribution $\mu$ by a good estimator $\hat{\mu}$ with respect to the Wasserstein metric allows to infer the topology of the support of $\mu$, see~\cite{MR2859954}. When using non noisy data, one can look at~\cite{MR4421180} and~\cite{MR4441130} for the most recent results and for an overview of the references.

However, despite these fruitful developments, geometric inference from noisy data remains a theoretical  and practical widely open problem.
In this paper, we are interested in the estimation of possibly low dimensional supports, and of distributions supported on such supports, when the observations are corrupted with {\it unknown} noise. We aim at giving a new contribution on the type of noise which can affect the data without preventing to build consistent estimators of the support and of the law of the noisy signal.

\subsection{Previous works: estimation of the support with noisy data}

Some of the geometric ideas that have been developed to handle non noisy data can be applied, or adapted, to handle noisy data and build estimators with controlled risk. These works generally consider a noise that is normal to the unknown manifold, in which case the amplitude of the noise has to be bounded by the reach of the manifold (the reach is some regularity parameter of a manifold, see~\cite{Federer59} for a precise definition). 
The upper bound on the risk contains a term depending on the amplitude of noise. Thus, the upper bound on the estimation risk is meaningful only when the bound on the noise is small, and the estimator is consistent when the noise tends to $0$ with the amount of data tending to infinity. See~\cite{MR3909931}, \cite{AaronBoundary}, \cite{AFischer}, \cite{MR4002890}, see also~\cite{MR4420765} in which the noise can be non orthogonal to the manifold. In~\cite{AizenbudSober21}, the noise is not normal to the manifold but the data is 
uniformly sampled on a tubular neighborhood of the unknown manifold, which allows to take advantage of the fact that the manifold lies in the middle of the observations. The magnitude of the noise also has to be upper bounded by the reach.
When the noise is not assumed very small, results are known in the  specific setting 
of clutter noise, see~\cite{GPVW12}, that is the situation where a proportion of data is uniformly sampled from a known compact set, and the remaining data is noiseless. The authors propose a
clever idea to remove noise by comparing the way the empirical data concentrate near any regular shape, and they find a consistent estimator with upper bounded risk.

When we accept to consider noise with known distribution, a popular model for noisy data is the deconvolution model, in which the low dimensional data are corrupted with independent additive noise.
In such models, 
all estimation procedures are roughly based on the fact that it is possible to get an estimator of the characteristic function of the 
non noisy data by dividing an estimator of the characteristic function of the noisy data by that (known) of the noise. 
In the deconvolution setting, the authors of
\cite{GPVW12} consider data corrupted with Gaussian noise, and propose as estimator of the manifold an upper level set of an estimator of a  kernel smoothing density of the unknown distribution. With the truncated Hausdorff loss, the authors prove that their estimator achieves a maximum risk (over some class of distributions) upper bounded by $(\sqrt{\log n})^{-1+\delta}$ for any positive $\delta$, and prove a lower bound of order $(\log n)^{-1+\delta}$ for the minimax risk.
Taking an upper level set of an estimated density had been earlier proposed to estimate a support based on non noisy data in~\cite{MR1604449}.
In the context of full dimensional convex support and with additive Gaussian noise, \cite{MR4255215} proposes an estimation procedure using convexity ideas. The authors prove an upper bound of order $\log \log n / \sqrt{\log n}$ and a lower bound of order $(\log n)^{-2/\tau}$
for the minimax Hausdorff risk, for any $\tau\in(0,1)$. Earlier work with known noise and with full dimensional support is
\cite{MR2298884}, where the author first builds an estimator of the unknown density using deconvolution ideas, then samples from this estimated density and 
takes a union of balls centered on the sampled points, such as in~\cite{MR579432}.


\subsection{Previous works: estimation of the distribution with noisy data}
 
The case of unknown but small (and orthogonal to the unknown manifold) noise is handled
in~\cite{MR4421180}, the author proposes a kernel estimator and proves that it is minimax. The rate depends on the upper bound of the noise. Non parametric Bayesian methods have been explored  in~\cite{berenfeld2022estimating} for observations on a tubular neighborhood of the unknown manifold, that is again for bounded noise.

In the deconvolution problem, with known Gaussian noise, the authors of 
\cite{MR3189324} prove matching upper and lower bounds  for the minimax risk of the estimation of the unknown distribution using the Wasserstein distance. Results for other known noises, but limited to one dimensional observations, can be found in~\cite{MR3314482}.



\subsection{Contribution and main results}

In this work, we focus on the situation where the observations are corrupted with additive and independent noise. 
It has been proved recently~\cite{LCGL2021} that, under very mild assumptions, it is possible to solve the deconvolution problem  without knowing the noise distribution and with no sample of the noise. 
In that work, the authors consider the density estimation problem. Here, we are faced with the more general situation where the underlying non noisy data may have a distribution with a possibly lower dimensional support than the ambient space, thus having no density with respect to Lebesgue measure. The intuition behind our work is that certain geometrical properties of the signal support induce sufficient structure for the deconvolution problem to be solvable without knowledge of the noise distribution, and that in such situations the identifiability theory of~\cite{LCGL2021} applies. 

Our main contributions are as follows.
\begin{itemize}
    
    \item We are able to propose estimators (based on observations corrupted with totally unknown noise) achieving adaptive minimax rates (up to a $\log \log$ factor) for the estimation of the support in Hausdorff distance. The minimax rates are investigated over well chosen classes of distributions, see Theorem~\ref{theorem:rateH}, Theorem~\ref{theo:adapt} and Theorem~\ref{thm:lower}. 
    Specifically, the minimax risk for the Hausdorff distance is upper bounded by $(\log \log n)^{L}/(\log n)^{\kappa}$ for some $L$, where $\kappa\in(1/2,1]$ is a parameter depending on the tail of the distribution of the signal ($\kappa=1$ corresponds to compactly supported distributions, and $\kappa = 1/2$ to sub-Gaussian distributions), while the minimax risk is lower bounded by $1/(\log n)^{\kappa}$ if $\kappa\in(1/2,1)$ and $1/(\log n)^{1-\delta}$ if $\kappa=1$, $\delta$ being any (small) positive number. Adaptation is with respect to $\kappa$. In some sense, exhibiting these classes of distributions allows to fill the gap between the upper and lower bounds in \cite{GPVW12}, together with the extension to totally unknown noise.
    
    \item We consider the estimation of the unknown (in general singular) distribution of the hidden non noisy data itself when it has a compact support. We prove almost matching upper and lower bounds of order $1/(\log n)$ for the estimation risk of the distribution in Wasserstein distance, see Theorem~\ref{theorem:upperboundwasserstein} and Theorem~\ref{theorem:lowerboundwasserstein}.
    
     \item 
    We give a precise outline to the intuitions of structure which make it possible to get our estimation results.
     We 
     exhibit simple geometric properties of a support so that, whatever the distribution on such an (unknown) support (provided it does not have too heavy tails), the deconvolution problem can be solved without any knowledge regarding the noise, see Theorem~\ref{prop:idH} and Theorem~\ref{prop:idgeneral}. 
     
\end{itemize}
Although we exhibit estimators, let us insist on the fact that our goal is mainly theoretical. We do not pretend to propose easy to compute estimation procedures, but to give precise answers about minimax adaptive rates for support and distribution estimation with noisy data in a very general deconvolution setting, where the noise is unknown and can have any distribution.

\subsection{Organisation of the paper}


In Section~\ref{sec:ident} we precise the setting and provide a general overview of the estimation procedure.
We focus on support estimation in Section~\ref{sec:support}. 
%
Section~\ref{sec:Wasserstein} is devoted to the estimation of the distribution when it is compactly supported. 
We then exhibit in Section~\ref{subsec:geocondi} geometric conditions under which the identifiability theory of~\cite{LCGL2021} applies. Genericity of such conditions is discussed in the supplementary material Section \ref{sec:genericity}. We discuss possible improvements and open questions in Section \ref{sec:discussion}. 
Detailed proofs are given in Section \ref{sec:proofs} of the supplementary material. We gather in Section \ref{sec:notations} a summary of notations defined along the paper.

\subsection{Notations}
\label{sec_notations_intro}

The Euclidean norm (in any dimension) will be denoted $\|\cdot \|_{2}$. 
If $A$ is a subset of $\Rbb^D$, we write 
$d(x,A)= \inf\{ \|x-y\|_{2} \ | \ y \in A \}$. For any $r > 0$, we write $B_r = (-r,r)$.
For any measurable function $f$ on $B_r^D$, we write $\|f\|_{\infty,r}$ the essential supremum of $f$ over $B_r^D$ and $\|f\|_{2,r} $ the norm of $f$ in $L^2(B_r^D)$.

\section{The identifiability Theorem and estimation procedures}
\label{sec:ident}

In this section, we 
recall the general identifiability Theorem proved in~\cite{LCGL2021}
and we provide an overview of the ideas underlying our estimation procedures.

\subsection{Setting}

We consider independent and identically distributed observations $Y_{i}$, $i=1,\dots,n$   coming from the model
\begin{equation}
\label{eq:model0}
Y=X+\varepsilon,
\end{equation}
in which the signal $X$ and the noise $\varepsilon$ are independent random variables. 
We assume that the observation has dimension at least two, and that its coordinates can be partitioned in such a way that the corresponding blocks of noise variables are independently distributed, that is
\begin{equation}
\label{eq:model}
Y = \begin{pmatrix} Y^{(1)}\\Y^{(2)}\end{pmatrix} = \begin{pmatrix} X^{(1)}\\X^{(2)}\end{pmatrix} + \begin{pmatrix} \varepsilon^{(1)}\\ \varepsilon^{(2)}\end{pmatrix}=X+\varepsilon
\end{equation}
in which $Y^{(1)}, X^{(1)}, \varepsilon^{(1)} \in \Rbb^{d_1}$ and $Y^{(2)}, X^{(2)}, \varepsilon^{(2)} \in \Rbb^{d_2}$, for $d_1, d_2 \geq 1$ with $d_1+d_2 = D$, and we assume that the noise components $\varepsilon^{(1)}$ and $\varepsilon^{(2)}$ are independent random variables.
We write $G$ the distribution of $X$ and $\Mcal_{G}$ its support. For $i \in \{1,2\}$, we write ${\Qbb}^{(i)}$ the distribution of $\varepsilon^{(i)}$, so that $\Qbb = \Qbb^{(1)} \otimes \Qbb^{(2)}$ is the distribution of $\varepsilon$.

We shall not make any more assumption on the distribution of the noise $\varepsilon$, and we shall not assume that its distribution is known. Indeed in~\cite{LCGL2021}, it is proved that under very mild conditions on the distribution of the signal $X$, model \eqref{eq:model} is fully identifiable, that is one can recover $G$, and thus its support, and $\Qbb$ from the convolution $G* \Qbb$.

\subsection{Identifiability Theorem}
\label{subsec:ident}

Let us introduce the assumptions on the distribution of the signal we shall use.
To state the first assumption, we let $\rho$ be a positive real number.
\begin{description}
\item[A($\rho$)]
There exist $a,b > 0$ such that for all $\lambda \in \Rbb^D$, $\Ebb \left[\exp \left(\lambda^\top X\right)\right]
\leq a \exp \left( b \|\lambda\|_{2}^\rho\right)$.
\end{description}

Assumption A($\rho$) is about the tail of $G$ as the following proposition shows. 

\begin{proposition}
\label{prop:A(rho)}
\begin{itemize}
    \item A random variable $X$ satisfies A(1) if and only if its support is compact.
    \item A random variable $X$ satisfies A($\rho$) for $\rho > 1$ if and only if there exist constants $c,d > 0$ such that for any $t \geq 0$,
    \begin{equation*}
        \Pbb(\|X\|_2 \geq t) \leq c \exp(-d t^{\rho / (\rho-1)}).
    \end{equation*}
\end{itemize}
\end{proposition}
The proof of Proposition~\ref{prop:A(rho)} is detailed in Section~\ref{proof:A(rho)}.

Under A($\rho$), the characteristic function of the signal can be extended into the multivariate analytic function
\begin{eqnarray*}
\Phi_X: \Cbb^{d_1}\times \Cbb^{d_2} &\longrightarrow& \Cbb \\
(z_1,z_2)&\longmapsto& \Ebb \left[ \exp \left(iz_1^\top X^{(1)} + i z_2^\top X^{(2)}\right)\right].
\end{eqnarray*}
The second assumption is the following. 

\begin{description}
\item[(Adep)] For any $z_{0}\in \Cbb^{d_1}$, 
$z \mapsto \Phi_X (z_{0},z)$
is not the null function and for any $z_{0}\in \Cbb^{d_2}$, 
$z \mapsto \Phi_X (z,z_{0})$
is not the null function.
\end{description}
It has been shown in \cite{LCGL2021} that several models satisfy this assumption, such as the repeated measurements submodel (see Corollary 2.3 in \cite{LCGL2021}) or the noisy independent component analysis submodel (see Corollary 2.2 in \cite{LCGL2021}). In particular, it follows directly from Corollary 2.4 of \cite{LCGL2021} and the inverse function Theorem that the errors in variable regression model (in which $X^{(2)} = g(X^{(1)})$ for some function $g$) satisfies assumption (Adep) when the regression function is a non-constant differentiable function.
In this paper, we are interested in identifying supports of distributions and we will provide geometrical conditions so that Assumption (Adep) is verified, see Section~\ref{subsec:geocondi}.

Obviously, if no centering constraint is put on the signal or on the noise, it is possible to translate the signal by a fixed vector $m \in \Rbb^D$ and the noise by $-m$ without changing the observation. The model can thus be identifiable only up to translation.

\begin{theorem}[from~\cite{LCGL2021}]
\label{thm:idG}
If the distribution of the signal satisfies A($\rho$) and (Adep), then the distribution of the signal and the distribution of the noise can be recovered up to translation from the distribution of the observations.
\end{theorem}
The proof of this theorem is based on recovering $\Phi_X$ for which the structure provided by (Adep) is crucial. The arguments show that knowing the characteristic function of the observations in a neighborhood of the origin allows to recover $\Phi_{X}$ in a neighborhood of the origin, and then over the whole multidimensional complex plane.

\subsection{Overview of the estimation procedures}
\label{overview}


The identifiability result above is the base upon which our estimators are built. A key part of its proof is the following result: if a multivariate analytic function $\phi$ satisfies $\phi(0) = 1$, $\phi(t) = \overline{\phi(-t)}$ for all $t$, as well as assumptions A($\rho$) for some $\rho < 2$, (Adep) and
\begin{equation*}
\phi(t_1,t_2) \Phi_X(t_1,0) \Phi_X(0,t_2) = \Phi_X(t_1,t_2) \phi(t_1,0) \phi(0,t_2)
\end{equation*}
for all $(t_1, t_2)$ in a neighborhood of zero in $\Rbb^{d_1} \times \Rbb^{d_2}$, then $\phi = \Phi_X$.
In particular, fixing some $\nuest > 0$, the only function $\phi$ such that
\begin{equation*}
\int_{B_{\nuest}^{d_1}\times B_{\nuest}^{d_2}} | \phi(t_1,t_2) \Phi_X(t_1,0) \Phi_X(0,t_2) - \Phi_X(t_1,t_2) \phi(t_1,0) \phi(0,t_2) |^2 
|\Phi_{\varepsilon^{(1)}}(t_1)\Phi_{\varepsilon^{(2)}}(t_2) |^2 d t_1 d t_2 = 0
\end{equation*}
is $\Phi_X$. Indeed 
the characteristic functions of the noises $\Phi_{\varepsilon^{(i)}}$, $i \in \{1,2\}$, are continuous and equal to 1 in zero. Moreover, the introduction of these characteristic functions allow to combine them with the characteristic functions $\Phi_X$ of the signal to get the characteristic functions of the observations, so that we are able to construct an empirical criterion $M_n$ that estimates the above integral (see Equation~\eqref{eq_def_Mn}). 

Following the usual so-called $M$-estimation procedure, we then construct an estimator $\widehat{\phi}$ of $\Phi_X$
by minimizing $M_n$ over multivariate analytic functions satisfying A($\rho$) and such that $\phi(0) = 1$ and $\phi(t) = \overline{\phi(-t)}$ for all $t$, satisfying (Adep), and having a finite polynomial expansion of degree $m$.
The degree $m$ must be well chosen and 
we use $m=\lceil 4\frac{\log n}{\log \log n} \rceil$.
Proposition~\ref{prop:phihat} shows that with high probability, for some $\nu \in (0,\nuest]$ such that the Fourier transform of the noise $\varepsilon$ does not vanish on $[-\nu,\nu]^D$,
\begin{equation*}
\| \widehat{\phi} - \Phi_X \|^2_{2,\nu} = O\left(\frac{1}{n^{1-\delta}}\right)
\end{equation*}
for some $\delta > 0$ arbitrarily small.
This control is obtained uniformly over a range of values of $\rho$, to be used further for adaptation in the parameter $\rho$. 

While it may not be possible to directly recover the distribution $G$ of $X$  from $\Phi_X$ by inverse Fourier transform since $G$ could be singular, it is possible to recover the convolution function $G * \Psi$ of $G$ by $\Psi$ for any properly chosen kernel $\Psi$. Moreover, when $\Psi$ is an approximation of the unity, 
the support of $G$ can be approximated by an upper level set of $G * \Psi$. Intuitively, if $\Psi$ were the density of a random variable $Z$, $G * \Psi$ would be the density of $X+Z$. When $\Psi$ is an approximation of the unity, 
the distribution of $Z$ will be close to the Dirac in zero, thus $X+Z$ will concentrate close to the support of $X$.
Figure~\ref{fig:manifold} illustrates this idea. In Section~\ref{subsec:Mupper}, we propose an estimator of the support as an upper-level set of an estimated density following ideas of~\cite{GPVW12}, the main difference being with the smoothing kernel we choose.
Indeed, with this kernel, no prior knowledge on the intrinsic dimension is needed to build the estimator. The precise choice of $\Psi$ is non trivial and is given in Section~\ref{subsec:Mupper}.

\begin{figure}[t]
\centering
    \includegraphics[scale=0.45]{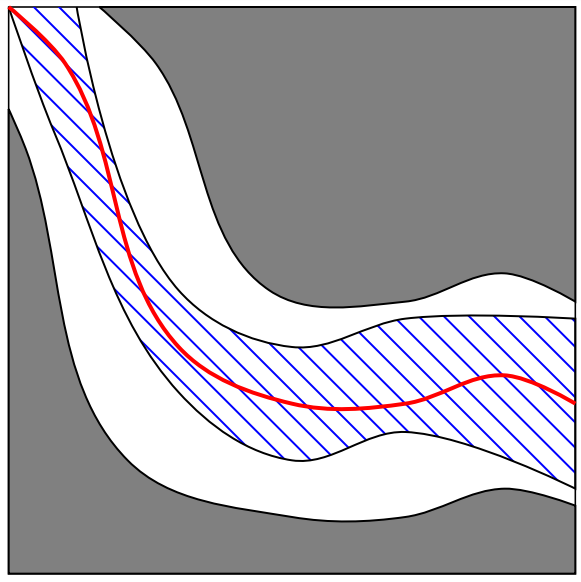}








    \caption{In red, the support of the signal distribution $\Mcal_G$, the blue hatched 
    area represents the set $\{ \bar{g} \geq \lambda_n + \| \bar{g} - \widehat{g}_n\|_{\infty} \}$ and the gray area represents the set $\{ \bar{g} \leq \lambda_n - \| \bar{g} - \widehat{g}_n\|_{\infty} \}$, so that the estimator of the support lies in between the gray and the blue areas.}
    \label{fig:manifold}
    
\end{figure}

Proposition~\ref{prop:phihat} directly ensures that the estimator
\begin{equation*}
\widehat{g}:= \Fcal^{-1}\left[ \widehat{\phi} \cdot \Fcal[\Psi] \right]
\end{equation*}
is close to $\bar{g}:= \Fcal^{-1}\left[ \Phi_X \cdot \Fcal[\Psi] \right] = G * \Psi$, whose upper level sets are close to $\Mcal_G$. Here, for any integrable function $f$ from $\Rbb^D$ to $\Rbb$, we denote by $\Fcal[f]$ (resp. $\Fcal^{-1}[f]$) the (resp. inverse) Fourier transform of $f$ defined, for all $y \in \Rbb^d$, by
\begin{equation*}
\Fcal[f](y) = \int e^{it^{\top}y} f(t) dt \ \ \text{and} \ \ \Fcal^{-1}[f](y) = (\frac{1}{2 \pi})^D \int e^{-it^{\top}y} f(t) dt.
\end{equation*}

Our estimator is thus taken as
\begin{equation*}
\widehat{\Mcal} = \{ y: \widehat{g}(y) > \lambda \},
\end{equation*}
where $\lambda$ has to be well chosen.
To obtain the best possible rates 
requires a further polynomial truncation of $\widehat{\phi}$ before Fourier inversion, depending on the parameter $\rho$. We then provide a model choice procedure to adapt in $\rho$ and get adaptive rates. Estimation of the distribution starts from the estimation of the support and follows similar ideas.

\section{Estimation of the support}
\label{sec:support}

In Section~\ref{subsec:chara}, we describe the estimator of the characteristic function used in all our procedures, and we give its properties. In Section~\ref{subsec:Mupper}, we provide an estimator of the support of the signal when $\rho$ is known, and prove an upper bound for the maximum risk in Hausdorff distance. Section~\ref{subsec:adapt} is devoted to the construction of an adaptive estimator of the support for unknown $\rho$. In Section~\ref{subsec:Mlower}, we prove a lower bound which shows that our estimator is minimax up to some power of $\log \log n$ for all $\rho \in (1,2)$ and up to any small power of $\log n$ for $\rho=1$. 


\subsection{Estimation of the characteristic function}
\label{subsec:chara}

We shall need sets of multivariate analytic functions for which A($\rho$) and (Adep) hold. 
For any $S>0$, let $\Upsilon_{\rho,S}$ be the subset of multivariate analytic functions from $\Cbb^{D}$ to $\Cbb$ defined as follows.
\begin{equation}
\label{eq_def_Upsilon}
\Upsilon_{\rho,S} = \left\{
\phi \text{ analytic } \text{s.t. } \forall z\in \Rbb^{D},\overline{\phi(z)}=\phi(-z), \phi(0) = 1 \text{ and } \forall i \in \Nbb^D \setminus \{0\}, \left| \frac{\partial^{i} \phi(0)}{\prod_{a=1}^d i_a!}\right| \leq \frac{S^{\|i\|_1}}{\|i\|_1^{ \|i\|_1 /\rho}} \right\}
\end{equation}
where $\|i\|_1 =\sum_{a=1}^{D} i_{a}$.
If the distribution of $X$ satisfies A($\rho$), then there exists $S$ such that $\Phi_X \in \Upsilon_{\rho,S}$, and the converse also holds, see Lemma 3.1 in~\cite{LCGL2021}. 

Let $\Phi_{\varepsilon^{(i)}}$ be the characteristic function of $\varepsilon^{(i)}$, $i=1,2$, and define for all $\phi\in \Upsilon_{\rho,S}$ and any $\nu >0$,
\begin{equation*}
M(\phi; \nu | \Phi_X) = \int_{B_{\nu}^{d_1}\times B_{\nu}^{d_2}} | \phi(t_1,t_2) \Phi_X(t_1,0) \Phi_X(0,t_2) - \Phi_X(t_1,t_2) \phi(t_1,0) \phi(0,t_2) |^2 
|\Phi_{\varepsilon^{(1)}}(t_1)\Phi_{\varepsilon^{(2)}}(t_2)  |^2 d t_1 d t_2.
\end{equation*}
Fix some $\nuest>0$, and define $M_{n}$ for any $\phi$ as follows
\begin{equation}
\label{eq_def_Mn}
M_{n}(\phi) = \int_{B_{\nuest}^{d_1}\times B_{\nuest}^{d_2}} | \phi(t_1,t_2) \tilde\phi_n(t_1,0) \tilde\phi_n(0,t_2) - \tilde \phi_n(t_1,t_2) \phi(t_1,0) \phi(0,t_2) |^2 d t_1 d t_2 ,
\end{equation}
where for all $(t_1,t_2) \in \Rbb^{d_1}\times \Rbb^{d_2}$,
\begin{equation*}
\tilde \phi_n(t_1,t_2) = \frac{1}{n}\sum_{\ell=1}^{n} \exp\left\{it_1^\top Y_{\ell}^{(1)} + it_2^\top Y_{\ell}^{(2)}\right\}.
\end{equation*}
By the law of large numbers, $M_{n}(\phi)$ is a consistent estimator of $M(\phi; \nuest | \Phi_X)$. As usual in M-estimation to get consistency of the estimator, we need to minimize $M_n$ over a closed set of functions over which $\Phi_X$ is the only minimizer of
$M(\cdot; \nuest | \Phi_X)$.
We thus introduce $\Hcal$ a subset of functions $\Cbb^D \rightarrow \Cbb^D$ such that all elements of $\Hcal$ satisfy (Adep) and such that the set of the restrictions to $B_{\nuest}^D$ of functions in $\Hcal$ is closed in $L^{2}(B_{\nuest}^D)$. Indeed we shall consider consistency of the estimator in $L^{2}(B_{\nuest}^D)$, see Proposition \ref{prop:phihat} below.  For instance, $\Hcal$ can be defined using submodels such as errors in variable regression with regression functions in a closed set of smooth functions when it is believed that the signal $X$ belongs to this submodel. $\Hcal$ can also be defined using geometric constraints such as those described in Section \ref{subsec:geocondi}, as soon as it is believed that the support of $X$ satisfies those constraints since we need $\Phi_X \in \Hcal$.
Let also $\Cbb_m[X]$ be the set of polynomial functions of degree $m$ in $D$ indeterminates.
%
For any integer $m$ and any $\rho > 1$, we define $\widehat \Phi_{n,m,\rho}$ be a (up to $1/n$) measurable minimizer of the functional $\phi \mapsto M_n(\phi)$ over $\Upsilon_{\rho,S} \cap \Hcal \cap \Cbb_m[X]$, that is
\begin{equation*}
M_n( \widehat \Phi_{n,m,\rho}) \leq \inf_{\phi \in \Upsilon_{\rho,S} \cap \Hcal \cap \Cbb_m[X]} M_n(\phi) + \frac{1}{n}.
\end{equation*}
For good choices of $m$, $\widehat \Phi_{n,m,\rho}$ is a consistent estimator of $\Phi_X$ in
$L^{2}(B_{\nu}^D)$ at almost parametric rate. The constants will depend on the signal through $\rho$ and $S$, and on the noise through its second moment and the following quantity:
\begin{equation}
\label{eq:cnu}
c_{\nu}=\inf\{|\Phi_{\varepsilon^{(1)}}(t)|,\;t\in B_{\nu}^{d_1}\} \wedge 
\inf\{|\Phi_{\varepsilon^{(2)}}(t)|,\;t\in B_{\nu}^{d_2}\}.
\end{equation}
Note that for any noise distribution, for small enough $\nu$, $c_{\nu}$ is a positive real number.
For any $\nu>0$, $c(\nu ) >0$, $E>0$, define
$\Qcal^{(D)} (\nu,c({\nu}),E)$ the set of distributions $\mathbb{Q}=\otimes_{j=1}^2 \mathbb{Q}_{j}$ on $\Rbb^D$ such that $c_{\nu}\geq c(\nu)$ and $\int_{\Rbb^D} \|x\|^2 d\mathbb{Q}(x)\leq E$. 

\begin{proposition}[Variant of Proposition 1 in~\cite{InferSphere}]
\label{prop:phihat}
For all $\rho_0 < 2$, $\nu\in(0,\nuest]$,
$S,c(\nu),E,C>0$ and $\delta, \delta',\delta'' \in (0,1)$ with $\delta' > \delta$,
there exist positive constants $c'$ and $n_0$ such that the following holds: let $\rho \in [1, \rho_0]$, for all $\Phi_X \in \Upsilon_{\rho,S} \cap \Hcal$ and $\Qbb \in \Qcal^{(D)} (\nu,c({\nu}),E)$, for all $n\geq n_0$ and $s \in [1, n^{1-\delta'}]$, with probability at least $1 - 2e^{-s}$,
\begin{equation*}
\sup_{\rho' \in [\rho,\rho_0], \  m \in [2\rho' \frac{\log n}{\log \log n}, C \frac{\log n}{\log \log n}]}
    \int_{B_{\nu}^{d_1}\times B_{\nu}^{d_2}} |\widehat \Phi_{n,m,\rho'}(t) - \Phi_X(t)|^2 dt
        \leq c' \left(\frac{s}{n^{1-\delta}}\right)^{1-\delta''}.
\end{equation*}
\end{proposition}

Note that the constants $c$ and $n_0$ do not depend on the distribution of $X$ or $\varepsilon$.
The proof of Proposition~\ref{prop:phihat} is based on results in~\cite{InferSphere} and~\cite{LCGL2021} and is detailed in Section~\ref{subsec:prop:phihat}.
For sake of simplicity, we denote $\widehat \Phi_{n,\rho}$ the estimator $\widehat \Phi_{n,m,\rho}$ in which $m=\lceil 4\frac{\log n}{\log \log n} \rceil$. Note that 
with this value of $m$, we may apply the inequality in Proposition \ref{prop:phihat} uniformly for $\rho'\in[\rho,\rho_0]$ whatever $\rho_0\in[1,2)$ and $\rho\in[1,\rho_0)$.

\subsection{Estimation of the support: upper bound}
\label{subsec:Mupper}

We shall consider the minimax risk in Hausdorff distance, which is defined, for $A_1$ and $A_2$ subsets of $\Rbb^D$, as
\begin{equation*}
d_H(A_1,A_2) = \sup_{x \in A_1 \cup A_2}|d(x,A_1)-d(x,A_2)|.
\end{equation*}
When $\rho > 1$, the support of $G$ is not compact. Since we allow the support to be a non-compact set, 
we define a truncated loss function as in~\cite{GPVW12}. We fix $\Kcal$ a compact subset of $\Rbb^D$ (which can be arbitrarily large) and for any $S_1$, $S_2$ subsets of $\Rbb^D$, the truncated loss function is
\begin{equation*}
H_{\Kcal} (S_1, S_2) = d_{H}(S_1 \cap \Kcal, S_2 \cap \Kcal).
\end{equation*}
We now define the class over which we will prove an upper bound for the maximum risk in Hausdorff distance. As in~\cite{LCGL2021}, it will be convenient to use $\kappa=1/\rho$. Since consistency of $\widehat{\Phi}_{n,m,\rho}$ is obtained for $X$ such that $\Phi_{X}\in \Hcal \cap \Upsilon_{1/\kappa,S}$ we summarize $\Lcal(\kappa, S, \Hcal)$ as the set of distributions $G$ such that, if $X$ is a random variable with distribution $G$, then $\Phi_X \in \Hcal \cap \Upsilon_{1/\kappa,S}$.
%
Moreover, for any positive constants $a$, $d$ and $r_0$, we define $St_{\Kcal}(a,d,r_0)$ as the set of positive measures $G$ such that for all $x \in \Mcal_G \cap \Kcal$, for all $r \leq r_0$, $G(B(x,r)) \geq a r^d$.
The distributions in $St_{\Kcal}(a,d,r_0)$ are called $(a,d)$-standard. Such a class of distributions is commonly used for inferring topological information, see for instance~\cite{MR3837127}.
\begin{remark}
\begin{itemize}
\item If a measure $\mu$ (for instance the $d$-dimensional Hausdorff measure on a manifold) is $(a,d')$-standard for some positive constants $a$ and $d'$, and if $G$ admits a density $g$ with respect to $\mu$ such that $g$ is lower bounded by $c>0$, then $G$ is $(ac,d')$-standard.
\item We do not make any assumption on the reach  of the support of $G$ (see~\cite{Federer59}) since it is not necessary here, although it provides a convenient way to check the $(a,d)$-standard assumption: if $\Mcal_G$ is a Riemannian manifold of dimension $d$ with $\text{reach}(\Mcal_G) \geq \tau_{\min} > 0$, then the $d$-dimensional Hausdorff measure restricted to $\Mcal_G$ is ($a,d$)-standard for some $a > 0$ (see Lemma 32 of~\cite{GPVW12}).
\end{itemize} 
\end{remark}



We now introduce the kernel we shall use for our construction. For any  $A > 0$, define, for all $y \in \Rbb$, 
\begin{equation*}
u_{A}(y) = \exp\left\{ -\frac{1}{(1-2y)^{A}} -\frac{1}{(1+2y)^{A}}\right\} 1|_{[-\frac{1}{2},\frac{1}{2}]}(y)
\end{equation*}
and
for all $y \in \Rbb^D$, 
\begin{equation*}
    \psi_A(y) = I(A) \ \Fcal^{-1}[u_{A} * u_{A}](\|y\|_2)  \ \text{with} \ I(A) = \frac{1}{\int \Fcal^{-1} [ u_{A} * u_{A} ](\|x\|_2) dx} .
\end{equation*}
For $h>0$ and $x \in \Rbb^D$, we write $\psi_{A,h}(x) = h^{-D} \psi_A(x/h)$, hence $\Fcal[\psi_{A,h}](t) = \Fcal[\psi_A](th)$. The following properties of $\psi_A$ and $\Fcal[\psi_A]$ hold
\begin{enumerate}
\item[(I)] The support of $\Fcal[\psi_A]$ is the unit ball $\{y\in\Rbb^D \,: \, \|y\|_{2}\leq 1\}$.
\item[(II)] $\psi_A > 0$ and $\Fcal[\psi_A] \geq 0$.
\item[(III)] There exist constants $c_{A}>0$ and $d_{A}>0$ such that for all $x \in \{y\in\Rbb^D \,: \, \|y\|_{2}\leq c_{A}\}$, $\psi_A(x) \geq d_{A}$.
\item[(IV)] $\psi_A$ and $\psi_{A,h}$ are probability densities on $\Rbb^D$.
\item[(V)] (Lemma in~\cite{BumpFunction}) For all $A>0$, there exists $\beta_A > 0$
such that
\begin{equation}
\label{lim}
    \lim_{\|t\|_{2} \rightarrow \infty} \exp{\{\beta_A \|t\|_{2}^{\frac{A}{A + 1}}\}} \psi_A(t) = 0
\end{equation} 
\item[(VI)] It holds
\begin{equation}
    \label{psi3}
\| \psi_{A,h}\|_2 = \frac{I(A)}{h^{D/2}} \| u_A * u_A \|_2.
\end{equation}
\end{enumerate}
%
%
Fix $A>0$ and define the function 
$\bar{g}$ by
\begin{equation*}
\forall y \in \Rbb^D, \ \bar{g}(y) = \left(\frac{1}{2 \pi}\right)^{D} \int e^{-it^{\top}y} \Fcal[\psi_A](ht) \  \Phi_X(t) dt, 
\end{equation*}
which may be rewritten using usual Fourier calculus, for all
$y \in \Rbb^D$, as
\begin{equation*}
\bar{g}(y) = (\psi_{A,h} * G) (y) = \frac{1}{ h^D} \int_{\Rbb^D} \psi_A\left(\frac{\|y - u\|_2}{h}\right) dG(u).
\end{equation*}
The density $\bar{g}$ is a kernel smoothing of the distribution $G$. The bandwidth parameter $h$ will be chosen appropriately in Theorem~\ref{theorem:rateH} below.
We now construct an estimator of $\bar{g}$ by truncating $\widehat{\Phi}_{n,1/\kappa}$ depending on $\kappa$. Adaptation with respect to $\kappa$ is handled in Section~\ref{subsec:adapt}. 
For some integer $m_{\kappa} > 0$ to be chosen later, let 
\begin{equation*}
\forall y \in \Rbb^D, \  \widehat{g}_{n,\kappa}(y) = \left(\frac{1}{2 \pi}\right)^{D} \int e^{-it^{\top}y} \Fcal[\psi_A] (ht) \ T_{m_{\kappa}} \widehat{\Phi}_{n,1/\kappa}(t) dt,
\end{equation*}
in which
for any multivariate analytic function $\phi$ defined in a neighborhood of $0$ in $\Cbb^D$ 
as $ \phi: x \mapsto \sum_{(i_1, \dots, i_D) \in \Nbb^D} c_i \prod_{a=1}^D x_a^{i_a}$, 
its truncation on $\Cbb_m[X]$ for any integer $m$ is
\begin{equation*}
T_m \phi: x \mapsto \sum_{(i_1, \dots, i_D) \in \Nbb^D \: \ i_1 + \dots + i_D\leq m} c_i \prod_{a=1}^D x_a^{i_a}.
\end{equation*}
Since for all $t \in \Rbb^D$,  $T_{m_{\kappa}} \widehat{\Phi}_{n,1/\kappa}(-t) = T_{m_{\kappa}} \widehat{\Phi}_{n,1/\kappa}(t)$, the function $\widehat{g}_{n,\kappa}$ is  real valued.
Finally, define an estimator of the support of the signal as the upper level set
\begin{equation*}
\widehat{\Mcal}_{\kappa} = \left\{ y \in \Rbb^D \ | \ \widehat{g}_{n,\kappa}(y) > \lambda_{n,\kappa} \right\},
\end{equation*}
for some $\lambda_{n,\kappa}$.
The main theorem of this section gives an upper bound of the maximum risk.

\begin{theorem}
\label{theorem:rateH}
Let $\kappa \in (1/2,1]$, $a>0$ , $d \leq D$, $r_0 > 0$.
For $c_h \geq \exp{(2D+2)}$ and $\ell \in (0, 1)$, define $m_{\kappa}$ and $h$ as 
\begin{equation*}
    m_{\kappa} = \Bigg \lfloor \frac{1}{4 \kappa} \frac{ \log(n)}{\log  \log(n)} \Bigg \rfloor, \  \ \  h = c_h S m_{\kappa}^{-\kappa}
\end{equation*}
and $\lambda_{n,\kappa}$ depending whether $d<D$ or $d=D$ as
\begin{itemize}
    \item if $d<D$,
    \begin{equation*}
        \lambda_{n,\kappa} =\left(\frac{1}{h}\right)^{\ell},
    \end{equation*}

    \item if $d=D$,
    \begin{equation*}
        \lambda_{n,\kappa} =\frac{1}{4} ac_{A}^{D}d_A.
    \end{equation*}
\end{itemize}
Then for any $\kappa_{0}\in (1/2,1]$,  $\nu\in(0,\nu_\text{est}]$, $c(\nu)>0$, $E>0$ $S>0$, there exists $n_0$ and $C > 0$ such that for all $n \geq n_0$,
\begin{equation*}
\sup_{\kappa\in[\kappa_{0},1]}
\underset{\Qbb\in \Qcal^{(D)} (\nu,c({\nu}),E)}
        {\sup_{G \in St_{\Kcal}(a,d,r_0)\cap \Lcal(\kappa,S,{\cal H})}} \ 
    \frac{\log(n)^{\kappa}}{\log( \log(n))^{\kappa + \frac{A+1}{A}}} \Ebb_{(G * \Qbb)^{\otimes n}}[H_{\Kcal}(\Mcal_{G},\widehat{\Mcal}_{\kappa})]
        \leq C .
\end{equation*}
\end{theorem}


\begin{remark}
\begin{itemize}
\item We prove in the next section a nearly matching lower bound. Thus, the minimax rate of convergence of the support in truncated Hausdorff distance depends on $\kappa$, that is on the way the distribution of the signal behaves at infinity. This rate deteriorates when the distribution of the signal has heavier tails. Indeed, since the distribution of the noise is unknown, taking into account distant observation points to build the estimator of the support becomes more difficult.

\item When $d<D$, thanks to the use of the kernel $\psi_{A}$, our estimator does not require the knowledge of $d$, which has to be compared with the estimator in~\cite{GPVW12} where prior knowledge of $d$ is needed.

\item In~\cite{GPVW12}, the upper bound on the rate is of order $1/\sqrt{\log n}$. Here we get a bound of order $1/(\log n)^{\kappa}$ depending on the tail of the distribution of the signal. We do not need to know the distribution of the noise, contrarily to~\cite{GPVW12} where the distribution of the noise is used in the construction of the estimator, as usual in the classical deconvolution literature.

\item It may be seen from the proof of Theorem~\ref{theorem:rateH} that the choice $\lambda_{n,\kappa} = \frac{1}{4} ac_{A}^{D}d_A$ is valid for any $d$. However, this requires the knowledge of $a$. 

\item Note that there are two truncation steps: the first one is implicit in the construction of $\hat{\Phi}_{n,1/\kappa}$ (chosen at the end of~Section~\ref{subsec:chara}) and the second one appears in the definition of $\widehat{g}_{n,\kappa}$. This second truncation is necessary to control the error of $\hat{\Phi}_{n,1/\kappa}$ on $B_{1/h}^D$ (see Lemma~\ref{lemma:gamma}, compared to the error on $B_\nu^D$ in Proposition~\ref{prop:phihat}), and the degree $m_\kappa$ in the second truncation is always smaller than the degree $m$ used in the construction of $\hat{\Phi}_{n,1/\kappa}$.

\item The constant $\ell$ can be chosen arbitrarily between 0 and 1.
\end{itemize}
\end{remark}

The proof of Theorem~\ref{theorem:rateH} is detailed in Section~\ref{proof:thmrateH}. 
As in~\cite{GPVW12}, the idea is to lower bound $\bar{g}$ on the support $\Mcal_{G}$ when the bandwidth parameter $h$ becomes small, and to upper bound it on every points further than a small distance (depending on $h$) from that support, see Lemmas~\ref{lem_minoration_gbar} and~\ref{lemma:sup} below. 

\begin{lemma}
\label{lem_minoration_gbar}
Assume $G \in St_{\Kcal}(a,d,r_0)$, then for any $h \leq r_0/c_A$,
\begin{equation*}
\inf_{y \in \Mcal_{G} \cap \Kcal} \bar{g}(y) \geq a c_{A}^{d}d_{A}\left(\frac{1}{h}\right)^{D-d}.
\end{equation*}
\end{lemma}
The proof of Lemma~\ref{lem_minoration_gbar} is detailed in Section~\ref{sec_proof_lem_minoration_gbar}.

\begin{lemma}
\label{lemma:sup}
For any $C_1 > 0$, there exists $h_0 > 0$ depending only on $C_1$, $D$ and $A$ such that for any $h \leq h_0$,
\begin{equation*}
\sup \left\{ \bar{g}(y) \ | \ y \in \Kcal, \;d(y,\Mcal_{G}) > h \left[\frac{D}{\beta_A}  \log\left(\frac{1}{h}\right)\right]^{\frac{A+1}{A}} \right\} \leq C_1.
\end{equation*}
\end{lemma}
The proof of Lemma~\ref{lemma:sup} is detailed in Section~\ref{sec_proof_lemmasup}.

The last ingredient is to control the difference between the convoluted density and its estimator, defined as $\Gamma_{n,\kappa} = \|\widehat{g}_{n,\kappa} - \bar{g}\|_{\infty} = \sup_{y \in \Rbb^D} |\widehat{g}_{n,\kappa}(y) - \bar{g}(y)|$.
We first relate it to  $\|T_{m_{\kappa}} \widehat{\Phi}_{n,1/\kappa} - \Phi_X\|_{2, 1/h}$.

\begin{lemma}
\label{lemma:gamma}
Let $h>0$ and $m>0$. For any $A>0$,
\begin{equation*}
\Gamma_{n,\kappa} \leq I(A)  \frac{\| u_A * u_A \|_2}{h^{D/2}}
 \|T_{m_{\kappa}} \widehat{\Phi}_{n,1/\kappa} - \Phi_X\|_{2, 1/h} .
\end{equation*}
\end{lemma}
The proof of Lemma~\ref{lemma:gamma} is detailed in Section~\ref{sec_proof_lemmagamma}.
The parameters $m_{\kappa}$ and $h$ are chosen so that $\Gamma_{n,\kappa}$ tends to $0$ with high probability, and the  threshold  $\lambda_{n,\kappa}$ is chosen using Lemmas~\ref{lem_minoration_gbar} and~\ref{lemma:sup}.

\subsection{Adaptation to unknown \texorpdfstring{$\kappa$}{kappa}}
\label{subsec:adapt}

We now propose a data-driven model selection procedure to select $\kappa$ such that the resulting estimator has the right rate of convergence. As usual, the idea is to perform a bias-variance trade off. Although we have an upper bound for the variance term, the bias is not easily accessible. We will use Goldenshluger and Lepski's method, see~\cite{MR2543590}. The variance bound is given as follows:
\begin{equation*}
    \sigma_{n}(\kappa)=c_{\sigma}\frac{(\log\log n)^{\kappa+\frac{A+1}{A}}}{(\log n)^{\kappa}}.
\end{equation*}
Fix some $\kappa_0 >1/2$. The bias proxy is defined as
\begin{equation*}
    B_{n}(\kappa)=0\vee \sup_{\kappa'\in[\kappa_{0},\kappa]}\left(H_{\Kcal}(\widehat{\Mcal}_{\kappa},\widehat{\Mcal}_{\kappa'})-\sigma_{n}(\kappa')\right).
\end{equation*}
The estimator of $\kappa$ is now given by
\begin{equation*}
    \widehat{\kappa}_{n}\in \argmin \left\{B_{n}(\kappa)+\sigma_{n}(\kappa),\;\kappa\in[\kappa_{0},1]
\right\},
\end{equation*}
and the estimator of the support of the signal is $\widehat{\Mcal}_{\widehat{\kappa}_{n}}$. The following theorem states that this estimator is rate adaptive.

\begin{theorem}
\label{theo:adapt}
For any $\kappa_{0}\in (1/2,1]$,  $\nu\in(0,\nu_\text{est}]$, $c(\nu)>0$, $E>0$ $S>0$, $a>0$, $d\leq D$, there exists $c_{\sigma}>0$ such that
\begin{equation*}
    \limsup_{n\rightarrow +\infty}\sup_{\kappa\in[\kappa_{0},1]}
\underset{
\Qbb\in \Qcal^{(D)} (\nu,c({\nu}),E)}
{\sup_{G \in St_{\Kcal}(a,d)\cap \Lcal(\kappa,S,{\cal H})}} 
\frac{\log(n)^{\kappa}}{\log( \log(n))^{\kappa + \frac{A+1}{A}}}\Ebb_{(G * \Qbb)^{\otimes n}}[H_{\Kcal}(\Mcal_{G},\widehat{\Mcal}_{\widehat{\kappa}_{n}})] <+\infty .
\end{equation*}
\end{theorem}
The proof of Theorem~\ref{theo:adapt} is detailed in Section~\ref{proof:theo:adapt}.

\subsection{Lower bound}
\label{subsec:Mlower}

The aim of this subsection is to prove a lower bound for the minimax risk of the estimation of $\Mcal_{G}$ using  the distance $H_{\Kcal}$ as loss function. The proof of Theorem~\ref{thm:lower} is based on Le Cam’s two-points method, see~\cite{MR1462963}, one of the most widespread technique to derive lower bounds. Note that we can not use the lower bound proved in~\cite{GPVW12} since the two distributions they use for the signal $X$ in their two-points proof  have Gaussian tails, for which $\kappa=1/2$. 

\begin{theorem}
\label{thm:lower}
For any $\kappa \in (1/2,1)$, there exists $S_{\kappa}>0$, $a_{\kappa}>0$ and $\Hcal^{\star}_{\kappa}$ a set of complex functions satisfying (Adep) such that the set of the restrictions of its elements to $[-\nu,\nu]^D$ is closed in $L_2([-\nu,\nu]^D)$ for any $\nu>0$, and
such that for all $S\geq S_{\kappa}$, $a\leq a_{\kappa}$, $d \geq 1$, $0 < r_0 < 1$, $E>0$ and $\nu\in(0,\nuest]$ such that $c(\nu)>0$, 
there exists $C>0$ depending only on $a$, $D$, $S$, $E$ and $\nu$, and there exists $n_0$, such that for all $n \geq n_0$,
\begin{equation}
\label{eq:infkappa}
\inf_{\widehat{\Mcal}} \ 
    \underset{\Qbb\in \Qcal^{(D)} (\nu,c({\nu}),E)}
        {\sup_{G \in St_{\Kcal}(a,d,r_0)\cap \Lcal(\kappa,S,{\cal H}^{\star}_{\kappa})}}
    \Ebb_{(G * \Qbb)^{\otimes n}}[H_{\Kcal}(\Mcal_G,\widehat{\Mcal})]
        \geq \frac{C}{\log(n)^{\kappa}},
\end{equation}
and for any $\delta \in (0,1)$, there exists $C>0$ depending only on $a$, $D$, $S$, $E$, $\nu$ and $\delta$, and there exists $n_0$, such that for all $n \geq n_0$,
\begin{equation}
\label{eq:inf1}
\inf_{\widehat{\Mcal}} \ 
    \underset{\Qbb\in \Qcal^{(D)} (\nu,c({\nu}),E)}
        {\sup_{G \in St_{\Kcal}(a,d,r_0)\cap \Lcal(1,S,{\cal H}^{\star}_{1})}}
    \Ebb_{(G * \Qbb)^{\otimes n}}[H_{\Kcal}(\Mcal_G,\widehat{\Mcal})]
        \geq \frac{C}{\log(n)^{1 + \delta}},
\end{equation}
where the infimum in \eqref{eq:infkappa} and \eqref{eq:inf1} is taken over all possible estimators $\widehat{\Mcal}$ of $\Mcal_G$.
\end{theorem}

\begin{remark}
\begin{itemize}
\item The lower bound in Theorem~\ref{thm:lower} almost matches the upper bound for the maximum risk of our estimator in Theorem~\ref{theorem:rateH}. Thus our work identifies the main factor in the minimax rate for the estimation of the support in Hausdorff loss. Notice that \eqref{eq:inf1} is weaker than \eqref{eq:infkappa} with $\kappa=1$, but we were not able to find a $f_{\kappa}$suitable to prove \eqref{eq:infkappa} with $\kappa=1$, see Lemma \ref{lemma:fourierbound} below.

\item In~\cite{GPVW12}, the lower bound does not match the upper bound by a larger power in the rate (almost twice). 

\item The sets of supports we consider are not the same as that considered in~\cite{GPVW12}. In~\cite{GPVW12}, the authors assume that the support is a regular manifold with lower bounded reach. We do not assume regularity, we only assume that the distribution of the signal is $(a,d)$-standard.
\end{itemize}
\end{remark}
The proof of Theorem~\ref{thm:lower} is detailed in Section~\ref{proof:lower}.
As usual for the two-points method, the idea is to find two distributions having support as far as possible in $H_{\Kcal}$-distance, and a noise such that the joint distributions of the observations have total variation distance upper bounded by some $C<1$. 
We shall consider the noise as in~\cite{LCGL2021}, with independent identically distributed coordinates having density $q$ defined as
\begin{equation*}
q: x \in \Rbb \longmapsto c_q \frac{1+\cos(cx)}{(\pi^2-(cx)^2)^2}
\end{equation*}
for some $c>0$, where $c_q$ is such that $q$ is a probability density, and with characteristic function 
\begin{equation*}
\Fcal[q]:t \mapsto c_q\left[\left(1-\left| \frac{t}{c}\right|\right)\cos \left(\pi \frac{t}{c}\right) + \frac{1}{\pi} \sin \left(\pi \left| \frac{t}{c}\right|\right)\right]\mathbf{1}_{-c\leq t \leq c}.
\end{equation*}
Let us now define the two distributions to apply the two-points method.
For any $\kappa \in (1/2,1]$, we first choose a density function $f_{\kappa}$ according to the following Lemma.

\begin{lemma}
\label{lemma:fourierbound}
For any $\kappa \in (1/2,1)$, $p \geq 1$, there exists a continuous density function $f_{\kappa}: \Rbb \rightarrow \Rbb$ in $L_{p}(\Rbb)$, positive everywhere, and positive constants $A,B$ such that for all $u \in \Rbb$,
\begin{equation*}
    |\Fcal[f_{\kappa}](u)| \leq A \exp(-B|u|^{\frac{1}{\kappa}}) \quad \text{and} \quad |\Fcal[f_{\kappa}]'(u)| \leq A \exp(-B|u|^{\frac{1}{\kappa}}).
\end{equation*}
For any $\delta \in (0,1)$, there exists a continuous compactly supported density function $f_1: [-1,1] \rightarrow \Rbb$ positive everywhere such that
\begin{equation*}
    |\Fcal[f_{1}](u)| \leq A \exp(-B|u|^{\delta}) \quad \text{and} \quad |\Fcal[f_{1}]'(u)| \leq A \exp(-B|u|^{\delta}).
\end{equation*}
\end{lemma}

The proof of Lemma~\ref{lemma:fourierbound} is detailed in Section~\ref{proof:fourierbound}.

Then, inspired by~\cite{GPVW12}, for all $\gamma \in (0,1]$, define $\tilde{g}_{\gamma}: \Rbb \rightarrow \Rbb$ and $g_{\gamma}: \Rbb \rightarrow \Rbb^{D-d}$ for all $x \in \Rbb$ as
\begin{equation*}
\tilde{g}_{\gamma}(x) = \cos\left(\frac{x}{\gamma}\right)
\quad \text{and} \quad
g_{\gamma}(x) = \left ( \tilde{g}_{\gamma}(x),0, \dots, 0 \right )^\top.
\end{equation*}
Let $M_0(\gamma) = \{ (u,\gamma g_{\gamma}(u)): u \in \Rbb \}$,  $M_1(\gamma) = \{ (u, -\gamma g_{\gamma}(u)): u \in \Rbb \}$, and let $J \in \Rbb^{D \times D}$ be the matrix
\begin{equation*}
    J = \left(
    \begin{array}{cc|ccc}
      1 & 0 & 0 & \dots &  0 \\
      1 & 1/2 & 0 &\dots & 0 \\
      \hline
      0 & 0 & 1 & \dots & 0 \\
      \vdots & \vdots & \vdots & \ddots & \vdots\\
      0 &  0 & 0 & \dots & 1 \\
      \end{array}
      \right ).
\end{equation*}

For any $\kappa \in (1/2,1]$, let $U(\kappa)$ be the random variable in $\Rbb$ having density $f_{\kappa}$ defined in Lemma~\ref{lemma:fourierbound} and let $S_0(\kappa) = (U(\kappa),\gamma g_{\gamma}(U(\kappa)))$, $S_1(\kappa) = (U(\kappa),-\gamma g_{\gamma}(U(\kappa)))$. For $i \in \{0,1\}$, we shall denote $T_i(\kappa)$  the distribution of $S_i(\kappa)$. Finally we define $X_i(\kappa) = J S_i(\kappa)$, $i=0,1$ and $G_i(\kappa)$ the distribution of $X_i(\kappa)$. When $\kappa < 1$, the supports of $S_i(\kappa)$ and $X_i(\kappa)$ do not depend on $\kappa$. They are illustrated in Figure~\ref{fig_support_lower}. The transformation $J$ is such that the the support of $X_i$ is the graph of a bijective function, which makes it easy to satisfy (Adep) through for instance assumptions (H1), (H2) and Theorem~\ref{thm:id} from Section~\ref{subsec:geocondi}.

\begin{figure}
    \centering
    \includegraphics[scale=.65]{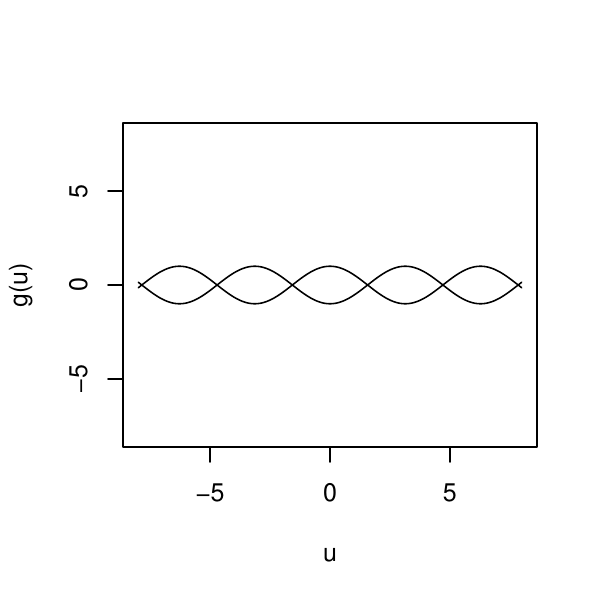}
    \includegraphics[scale=.65]{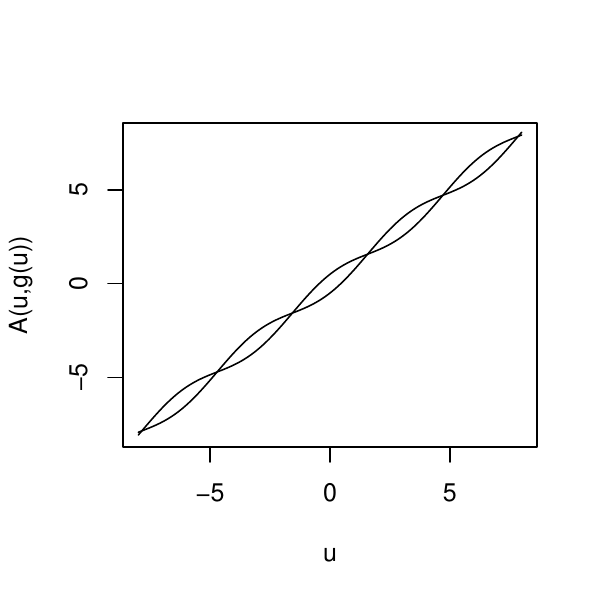}
    \caption{Left: support of $S_0$ and $S_1$ when $\gamma=1$. Right: support of $X_0$ and $X_1$ when $\gamma = 1$.}
    \label{fig_support_lower}
\end{figure}

\begin{lemma}
\label{lemma:A1/kappa}
For any $i \in \{0,1\}$ and $\kappa \in (1/2,1]$, $X_i(\kappa)$ satisfies A($1/\kappa$).
\end{lemma}
The proof of Lemma~\ref{lemma:A1/kappa} is detailed in Section~\ref{proof:lemma:A1/kappa}.
\begin{lemma}
\label{lemma:standard}
There exists $a_{0}>0$ such that for $i \in \{0,1\}$, for any $d \geq 1$, $r_0 < 1$ and $a \leq a_0$, $G_i(\kappa) \in St_{\Kcal}(a,d,r_0)$.
\end{lemma}
The proof of Lemma~\ref{lemma:standard} is detailed in Section \ref{proof:lemma:standard}.
\begin{lemma}
\label{lemma:GiA}
For any $i \in \lbrace 0,1 \rbrace$ and $\kappa \in (1/2,1]$, $\Phi_{X_{i}(\kappa)}$ satisfy (Adep)
\end{lemma}
The proof of Lemma~\ref{lemma:GiA} is detailed in Section~\ref{proof:lemma:GiA}.

We finally set $\Hcal^{\star}_{\kappa} = \lbrace \Phi_{X_{0}(\kappa)}, \Phi_{X_{1}(\kappa)} \rbrace$. As a finite set, it is closed in $L_2([-\nu, \nu]^D)$.

Let us comment on dimensionnality.
The distributions used here are distributions with support of dimension $1$. This is not an issue since the $d$ in the definition of $St_{\Kcal}(a,d,r_0)$ is an upper bound on the dimension of the support. We could also have used supports with dimension $d$ by adding to $X_{i}(\kappa)$ an independent uniform distribution on a Euclidean ball of a linear space of dimension $d$.

The support of $G_{i}(\kappa)$ is $J M_i(\gamma)$, and the following lemma follows easily from the fact that for $i\in\{0,1\}$, $J M_{i}(\gamma) = \gamma J M_{i}(1)$.
\begin{lemma}
\label{lemma:lowerboundH}
For any $\gamma > 0$,
\begin{equation}
    H_{\Kcal}(J M_0(\gamma),J M_1(\gamma)) = \gamma H_{\Kcal}(J M_0(1), J M_1(1)).
\end{equation}
\end{lemma}
To obtain the lower bound, the parameter $\gamma$ will be chosen as large as possible while making sure that the joint distributions of the observations have total variation distance smaller than some $C < 1$.

\section{Estimation of the distribution of the signal}
\label{sec:Wasserstein}
In this section, we assume that the support $\Mcal_{G}$ of $G$ is a compact subset of $\Rbb^D$. For any $\eta>0$ and $A \subset \Rbb^D$, $A_{\eta}$ will denote the $\eta$-offset of $A$, that is the set of all points $x$ in $\Rbb^D$ such that $d(x,A) \leq \eta$.

To estimate $G$, we shall consider the probability density $\bar{g}$ defined in Section~\ref{subsec:Mupper} and define the probability distribution $P_{\psi_{A,h_n}}$ on $\Rbb^D$ such that, for any $\Ocal$ borelian set of $\Rbb^D$,
\begin{equation*}
P_{\psi_{A,h_n}}(\Ocal) = \int_{\Ocal} \bar{g}(y) dy.
\end{equation*}
We then estimate $P_{\psi_{A,h_n}}$ using the estimation of $\bar{g}$ defined in Section~\ref{subsec:Mupper} for $\kappa=1$,  $\widehat{g}_n:=\widehat{g}_{n,1}$.
Since $\widehat{g}_n$ can be non positive, we use $\widehat{g}_n^{+} = \max{ \{ 0,\widehat{g}_n \} }$ and renormalize it to get a probability distribution. We shall also estimate $P_{\psi_{A,h_n}}$ with a probability distribution having support on a (small) 
offset of the estimated support $\widehat{\Mcal}$ restricted to the closed euclidean ball $\bar{B}(0,R_n)$, for some radius $R_n$ that grows to infinity with $n$. Thus we fix some $\eta> 0$ and define $\widehat{P}_{n, \eta}$ such that, for any $\Ocal$ borelian set of $\Rbb^D$,
\begin{equation*}
\widehat{P}_{n,\eta}(\Ocal)
    = \frac{1}{\int_{(\widehat{\Mcal} \cap \bar{B}(0,R_n))_{{\eta}}} \widehat{g}_n^{+}(y) dy} \int_{\Ocal \cap (\widehat{\Mcal} \cap \bar{B}(0,R_n))_{{\eta}}} \widehat{g}_n^{+}(y) dy
    = c_n \int_{\Ocal \cap (\widehat{\Mcal} \cap \bar{B}(0,R_n))_{{\eta}}} \widehat{g}_n^{+}(y) dy.
\end{equation*}

\subsection{Upper bound for the Wasserstein risk}
\label{sec_Wasserstein_upper}

The aim of this subsection is to give  an upper bound of the Wasserstein maximum risk  for the estimation of $G$. For any $p \in [1,+\infty)$ and any two probability measures $\mu$ and $\nu$ on $\Rbb^D$, we write $W_p(\mu,\nu)$ the Wasserstein distance of order $p$ between $\mu$ and $\nu$, that is
\begin{equation*}
W_p(\mu,\nu) = \inf_{\pi \in \Pi(\mu,\nu)} \left ( \int_{\Rbb^D \times \Rbb^D} \|x-y\|_{2}^p d\pi(x,y) \right )^{1/p},
\end{equation*}
where $\Pi(\mu,\nu)$ is the set of probability measures on $\Rbb^D \times \Rbb^D$ that have marginals $\mu$ and $\nu$.
\begin{theorem}
\label{theorem:upperboundwasserstein} 
For all $\nu\in(0,\nu_\text{est}]$, $c(\nu)>0$, $E>0$, $S>0$, $\eta>0$, $a>0$ , $r_0>0$, $d \leq D$, $p \in [1,+\infty)$, define $m_n$, $h_n$ and $\lambda_n$ as in Theorem~\ref{theorem:rateH} for $\kappa=1$.
Assume that $\lim_{n \rightarrow +\infty} R_n = +\infty$ and that there exists $\delta \in (0,\frac{1}{2})$ such that $R_n \leq \exp(n^{1/2 - \delta})$.
Then there exist $n_0$ and $C > 0$ such that for all $n \geq n_0$,
\begin{equation*}
\underset{\Qbb\in \Qcal^{(D)} (\nu,c({\nu}),E)}
        {\sup_{G \in St_{\Kcal}(a,d,r_0)\cap \Lcal(1,S,{\cal H})}}
    \Ebb_{(G * \Qbb)^{\otimes n}}[W_p(G,\widehat{P}_{n,\eta})]
        \leq C \frac{\log \log(n)}{\log(n)}.
\end{equation*}
\end{theorem}
The proof of Theorem~\ref{theorem:upperboundwasserstein} is detailed in Section~\ref{proof:theorem:upperboundwasserstein}.
Note that the magnitude of $\eta$ does not appear to be crucial when looking at the proof, at least in an asymptotic perspective.
Considering an offset of size $\eta$ of $\widehat{\Mcal}$ is meant to control the variance of the estimator $\widehat{P}_{n,\eta}$.

\begin{remark}
\begin{itemize}
\item The lower bound in Theorem~\ref{theorem:lowerboundwasserstein} almost matches the upper bound for the maximum risk of our estimator in Theorem~\ref{theorem:upperboundwasserstein}. Thus our work identifies the main factor in the minimax rate for the estimation of the distribution in Wasserstein loss.

\item Comparison with earlier results in the deconvolution setting~\cite{MR3189324} or~\cite{MR3314482} with known noise is not easy since the classes of signals they consider is much different than the ones we consider.
\end{itemize}
\end{remark}

\subsection{Lower bound for the Wasserstein risk}

The aim of this subsection is to establish a lower bound for the minimax Wasserstein risk of order $p$ for any $p \geq 1$. Again, we can not use previous lower bounds proved in
\cite{MR3189324} or~\cite{MR3314482} since they use in the two-points method signals with distributions having too heavy tails.

\begin{theorem}
\label{theorem:lowerboundwasserstein}
For any $p \geq 1$, there exists $S_{1}>0$, $a_{1}>0$ and $\Hcal^{\star}_{1}$ a set of complex functions satisfying (Adep) such that the set of the restrictions of its elements  to $[-\nu,\nu]^D$ is closed in $L_2([-\nu,\nu]^D)$ for any $\nu>0$,
and such that for all $S\geq S_{1}$, $a\leq a_{1}$, $d \geq 1$, $0 < r_0 < 1$, $E>0$ and $\nu\in(0,\nuest]$ such that $c(\nu)>0$, 
there exists $C>0$ depending only on $a$, $D$, $S$, $E$ and $\nu$, and there exists $n_0$, such that for all $n \geq n_0$,
\begin{equation*}
\inf_{\widehat{P}_n} \ \underset{\Qbb\in \Qcal^{(D)} (\nu,c({\nu}),E)}
        {\sup_{G \in St_{\Kcal}(a,d,r_0) \cap \Lcal(1,S,{\cal H}^{\star}_{1})}}
    \Ebb_{(G * \Qbb)^{\otimes n}}[W_p(G,\widehat{P}_n)]
        \geq C \frac{1}{\log(n)^{1+\delta}},
\end{equation*}
where the infimum is taken on all possible estimate $\widehat{P}_n$ of $G$.
\end{theorem}

As for Theorem~\ref{thm:lower}, we use Le Cam's two-points method with the same two distributions $G_0(1)$ and $G_1(1)$. The proof essentially consists in showing that there exists a constant $C>0$ independent of $\gamma$ such that $W_p(G_0(1), G_1(1)) \geq C H_{\Kcal}(M_0(\gamma), M_1(\gamma))$, that is $W_p(G_0(1), G_1(1)) \geq C \gamma$ for a constant $C>0$. Once such an equality is established, the lower bound follows from taking $\gamma$ as for Theorem~\ref{thm:lower}.

The rest of the proof is detailed in Section~\ref{proof:lowerboundW}.

\section{Sufficient geometrical conditions for (Adep) to hold}
\label{subsec:geocondi}
Throughout the paper, we assume (Adep) to be able to solve the deconvolution problem. In this section, we provide very general conditions on the support of a random variable which are sufficient for (Adep) to hold.
In what follows we shall say that a random variable satisfies (Adep) when its characteristic function does. 
We first provide simple but useful  properties. For any dimension $d > 0$, we denote $GL_{d}(\mathbb{C})$ the general linear group with coefficient in $\mathbb{C}$.

\begin{proposition}
\label{prop:(Adep)prop}
The following holds.
\begin{itemize}
    \item[(i)]   Let $U$ and $V$ independent random variables satisfying A($\rho$). Then $U$ and $V$ satisfy (Adep) if and only if $U+V$ satisfies (Adep).
    \item[(ii)]  Let $U = \begin{pmatrix} U^{(1)}\\U^{(2)}\end{pmatrix}$ be a random variable such that $U^{(1)}\in\Rbb^{d_{1}}$ and $U^{(2)}\in\Rbb^{d_{2}}$. Let  $A \in GL_{d_{1}}(\Cbb)$, $B \in GL_{d_{2}}(\Cbb)$, $m_1 \in \Cbb^{d_{1}}$ and $m_2 \in \Cbb^{d_{2}}$. Define $V = \begin{pmatrix} A & 0 \\ 0 & B \end{pmatrix} \begin{pmatrix} U^{(1)}\\ U^{(2)}\end{pmatrix} + \begin{pmatrix} m_1\\m_2 \end{pmatrix}$. Then $U$ satisfies A($\rho$) if and only if $V$ satisfies A($\rho$). Moreover, $U$ satisfies A($\rho$) and (Adep) if and only if $V$ satisfies A($\rho$) and (Adep).
    
    \item[(iii)] Let $U^{(1)}$ and $U^{(2)}$ be two independent random variables in $\Rbb^{d_1}$ and $\Rbb^{d_2}$ respectively that satisfy A($\rho$) for some $\rho \geq 1$, then $U = \begin{pmatrix} U^{(1)}\\U^{(2)}\end{pmatrix}$ satisfies (Adep) if and only if $U^{(1)}$ and $U^{(2)}$ are Gaussian or Dirac random variables.
\end{itemize}
\end{proposition}
The proof of Proposition~\ref{prop:(Adep)prop} is detailed in section~\ref{proof:(Adep)}.

Point (i) of Proposition~\ref{prop:(Adep)prop} makes it possible to transfer a proof of (Adep) for a support with full dimension $D$ to a support with dimension $d < D$. Indeed, if $U$ is a random variable with support of dimension $d < D$, by introducing an independent random variable $V$ with support of full dimension $D$, proving that $U+V$ (whose support has full dimension) satisfies (Adep) ensures that $U$ satisfies (Adep) as well. For instance, Theorem~\ref{thm:id} below shows that a random variable having support the centered Euclidean ball with radius $\eta >0$ satisfies A(1) and (Adep). Thus geometric conditions such as those proposed in this section 
can be transposed from one dimension to another.

Point (ii) shows that the fact that A($\rho$) and (Adep) hold is not modified by linear transformations of each component of the signal. 

Finally, Point (iii) shows that to verify (Adep), outside of trivial cases, the two signal components can not be independent. Even further, combined with Point (i), this shows that it is not possible to write the signal as the sum of two independent signals where one of them has independent components: such independent sub-signals with independent components must be part of the noise.

In~\cite{InferSphere}, the authors prove that (Adep) holds for random variables supported on a sphere. In such a context, they prove that the radius of the sphere can be estimated at almost parametric rate. Here we give much more general conditions on the support of a random variable that are sufficient for (Adep) to hold. 

We define the following assumptions (H1) and (H2). Here, if $A$ is a subset of $\Rbb^D$, we write $\Diam(A)$ its diameter $\sup\{ \|x-y\|_{2} \ | \ x,y \in A \}$.

\begin{itemize}
    \item[(H1)] There exist $(A_{\Delta})_{\Delta>0}$ and $(B_{\Delta})_{\Delta>0}$ such that
    $A_\Delta \subset \Rbb^{d_2}$, $B_\Delta \subset \Rbb^{d_1}$, $(B_{\Delta})_{\Delta>0}$ is an increasing sequence,  $\Pbb(X^{(2)} \in A_{\Delta}) > 0$, $\lim_{\Delta \rightarrow 0} \Diam(B_\Delta) = 0$ and $\Pbb(X^{(1)} \in B_\Delta \,|\, X^{(2)} \in A_\Delta) = 1$.
    
    \item[(H2)]There exist $(A_{\Delta})_{\Delta>0}$ and $(B_{\Delta})_{\Delta>0}$ such that $A_\Delta \subset \Rbb^{d_1}$, $B_\Delta \subset \Rbb^{d_2}$,
    $(B_{\Delta})_{\Delta>0}$ is an increasing sequence, $\Pbb(X^{(1)} \in A_{\Delta}) > 0$, $\lim_{\Delta \rightarrow 0} \Diam(B_\Delta) = 0$ and $\Pbb(X^{(2)} \in B_\Delta \,|\, X^{(1)} \in A_\Delta) = 1$.
\end{itemize}

We prove that these assumptions are sufficient to ensure identifiability provided that A($\rho$) is satisfied.

\begin{theorem}
\label{thm:id}
    Assume that the distribution of $X$ satisfies A($\rho$), (H1) and (H2). Then $X$ satisfies A($\rho$) and (Adep).
\end{theorem}

The proof of Theorem~\ref{thm:id} is detailed in Section~\ref{proof:thmid}. 

One can interpret the assumptions (H1) and (H2) geometrically as shown in Figure~\ref{fig:H2}. In essence, it means that there exists a slice (along the first $d_1$, resp. last $d_2$, coordinates, with base $A_\Delta$) such that the random variable belongs to this slice with positive probability and such that on this slice, the support of the distribution is contained in an orthogonal slice (along the last $d_2$, resp. first $d_1$, coordinates) of diameter smaller than $\Delta$.

\begin{figure}[!h]
\centering

    \includegraphics[scale=0.5]{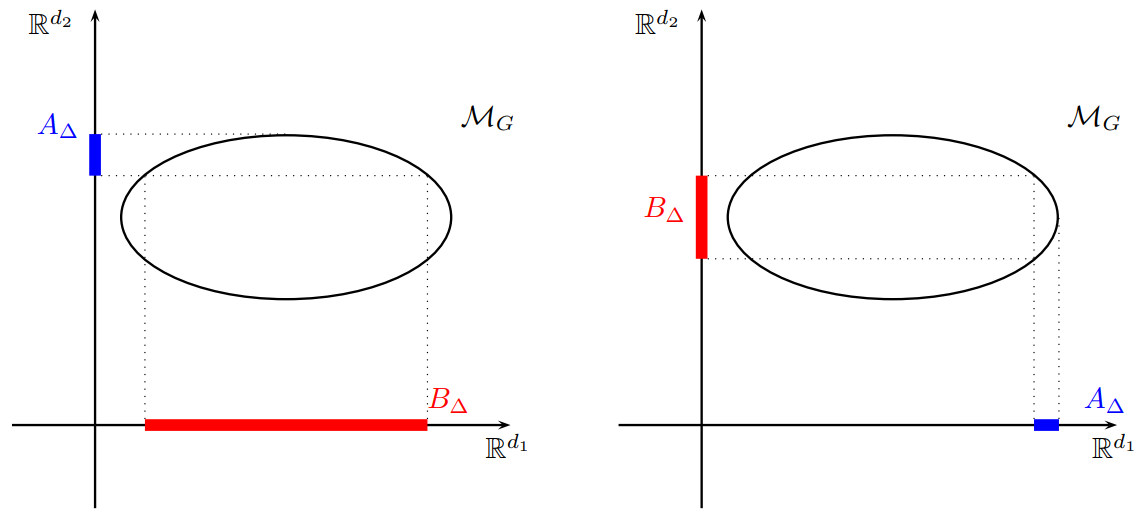}

    \caption{Left: Assumption (H1). Right: Assumption (H2).
    }
    \label{fig:H2}

\end{figure}

A reformulation of (H1) and (H2) based on the support of the signal is as follows.
Let
\begin{multline*}
\Acal_1(\Delta,\varepsilon) = \{ \Mcal \subset \Rbb^{D} \ | \ \text{There exists} \  x=(x_1,x_2) \in \Mcal \\
\text{such that} \ \Diam\Bigg (\pi^{(1:d_1)} \Bigg[\Mcal \cap (\Rbb^{d_1} \times \bar{B}(x_2,\varepsilon))\Bigg]\Bigg) < \Delta \}
\end{multline*}
and
\begin{multline*}
\Acal_2(\Delta,\varepsilon) = \{ \Mcal \subset \Rbb^{D} \ | \ \text{There exists} \  x=(x_1,x_2) \in \Mcal \\
\text{such that} \ \Diam\Bigg(\pi^{(d_1+1:D)} \Bigg[\Mcal \cap (\bar{B}(x_1,\varepsilon) \times \Rbb^{d_2})\Bigg]\Bigg) < \Delta \}.
\end{multline*}
The proof of the following theorem is straightforward.
\begin{theorem}
\label{prop:idH}
Let $\Mcal\in \left(\cap_{\Delta>0}\cup_{\varepsilon>0}\Acal_1(\Delta,\varepsilon) \right)\cap \left(\cap_{\Delta>0}\cup_{\varepsilon>0}\Acal_2(\Delta,\varepsilon)\right)$. Then any random variable with support $\Mcal$ satisfies (H1) and (H2).
\end{theorem}

We now propose sets of compact subsets of $\Rbb^D$ for which Theorem~\ref{prop:idH} holds.
Define the sets $\Bcal_1$ and $\Bcal_2$ as
\begin{equation*}
    \Bcal_1 = \{ \Mcal \subset \Rbb^D \text{ compact} \,|\, \exists x_1 \in \mathbb{R}^{d_{1}}, \text{Card}( (\{x_1\} \times \Rbb^{d_2}) \cap \Mcal ) = 1 \},
\end{equation*}
and
\begin{equation*}
    \Bcal_2 = \{ \Mcal \subset \Rbb^D \text{ compact} \,|\, \exists x_2 \in \mathbb{R}^{d_{2}}, \text{Card}( (\Rbb^{d_1} \times \{x_2\}) \cap \Mcal ) = 1 \}.
\end{equation*}

\begin{theorem}
\label{prop:idgeneral}
Let $\Mcal$ be a subset of $\Rbb^{D}$ such that $\Mcal \in \Bcal_1 \cap \Bcal_2$. If $X$ is a random variable with support $\Mcal$, then $X$ satisfies A($1$), (H1) and (H2).
\end{theorem}
The proof of Theorem~\ref{prop:idgeneral} is detailed in Section~\ref{proof:idgeneral}.

For instance, any closed Euclidean ball, and more generally any strictly convex compact set in $\Rbb^D$, is in $\Bcal_1 \cap \Bcal_2$. To see this, consider the points of the set with maximal first (resp. last) coordinate: they are unique by strict convexity, which ensures that the set is in $\Bcal_1$ (resp. $\Bcal_2$). The same holds for the boundary of any strictly convex compact set.

For the interested reader, we present some considerations on the genericity of (Adep) in Section~\ref{sec:genericity} of the supplementary material. Mainly, as a consequence of Theorems~\ref{prop:idH} and~\ref{thm:id}, (Adep) is generically satisfied in the sense that there exists a $G_\delta$-dense set of subsets of $\Rbb^D$ for the topology induced by the Hausdorff distance such that (Adep) holds for all distributions with support in this set.

\section{Discussion and further work}
\label{sec:discussion}

In this paper, we propose a very general setting in which for signal data corrupted with additive, independent and unknown noise, it is possible to recover the support and the distribution of the signal. We were able to obtain almost minimax rates over very general classes of distributions. We provided general geometrical assumptions of the support for which our results apply.

The fundamental assumptions of this work are twofold. The first one is A($\rho$) about the tail of the distribution of the signal, the second fundamental assumption is (Adep). The identifiability proof uses A($\rho$) for $\rho < 2$. In \cite{Aapo}, the authors extend to $\rho<3$ but assume that the signal has no Gaussian component. In the particular setting  of repeated observations, we were able to relax the assumption to any finite $\rho$, see \cite{Repeated}. 
Understanding whether identifiability holds for larger $\rho$ in other general settings remain as an open question.

An estimation method requires to fix a setting in which the identifiability assumptions hold. In this work, this is done by estimating the characteristic function as an element of $\Upsilon_{\kappa,S}$ and $\Hcal$. The choice of $\Hcal$ has to reflect some prior knowledge of what about the distribution of the signal makes (Adep) hold. This can be done using submodels or geometrical considerations, see the simulation experiments in \cite{InferSphere} for signals supported on a sphere and in \cite{Repeated} for repeated observations in which this precise estimation method is put in action. 

Building stable and numerically efficient algorithms remains an important and open problem. However, we believe our work opens the way to new estimation ideas. Using priors having similar properties as the one described in Section \ref{sec:genericity} could be a starting point to propose algorithms for Bayesian estimation methods.  We plan to investigate such Bayesian methods in further work. 

As usual in deconvolution settings, the minimax rates are very slow. It has been proved in the repeated measurements model that it is possible to adapt to ordinary smooth noises and get polynomial rates of convergence, see~\cite{Repeated}. Getting adaptivity to ordinary smooth noises in a general setting is an important question, with the subsidiary (still more difficult) question to find whether polynomial rates with exponents depending on the intrinsic dimension of the support are attainable.

\section{Acknowledgments}

Jeremie Capitao-Miniconi would like to acknowledge support from the UDOPIA-ANR-20-THIA-0013. \'Elisabeth Gassiat would like to acknowledge the Institut Universitaire de France and the ANR ASCAI: ANR-21-CE23-0035-02.
This work was supported by the French government, through the UCA$^\text{Jedi}$ and 3IA Côte d’Azur Investissements d’Avenir managed by the National Research Agency (ANR-15-IDEX-01 and ANR-19-P3IA-0002).

\appendix

\section{Summary of notations}
\label{sec:notations}
To help the reader, this section gathers the notations introduced in this paper, along with where they are defined.

\begin{itemize}
    \item For any $r > 0$, we write $B_r = (-r,r)$ and for any measurable function $f$ on $B_r^D$, we write $\|f\|_{\infty,r}$ the essential supremum of $f$ over $B_r^D$ and $\|f\|_{2,r} = ( \int_{B_r^D} |f(u)|^2 du)^{1/2}$. Section~\ref{sec_notations_intro}.
    \item $\Phi_{X}$: the characteristic function of $X$, $\Phi_{\varepsilon^{(i)}}$: the characteristic function of $\varepsilon^{(i)}$. Section~\ref{subsec:ident}.
    \item Assumption A($\rho$): There exist $a,b > 0$ such that for all $\lambda \in \Rbb^D$, $\Ebb \left[\exp \left(\lambda^\top X\right)\right] \leq a \exp \left( b \|\lambda\|_{2}^\rho\right)$. Section~\ref{subsec:ident}.
    \item Assumption (Adep): For any $z_{0}\in \Cbb^{d_1}$, 
$z \mapsto \Phi_X (z_{0},z)$
is not the null function and for any $z_{0}\in \Cbb^{d_2}$, 
$z \mapsto \Phi_X (z,z_{0})$
is not the null function. Section~\ref{subsec:ident}.
    \item $\Mcal_G$: support of the distribution $G$. Section~\ref{subsec:ident}.
    \item For all $(t_1,t_2) \in \Rbb^{d_1}\times \Rbb^{d_2}$, $
    \tilde \phi_n(t_1,t_2) = \frac{1}{n}\sum_{\ell=1}^{n} \exp\left\{it_1^\top Y_{\ell}^{(1)} + it_2^\top Y_{\ell}^{(2)}\right\}$. Section~\ref{subsec:chara}.
    \item $M(\phi; \nu | \Phi_X)$ and $M_{n}(\phi)$, Section~\ref{subsec:chara}:
    \begin{multline*}
    M(\phi; \nu | \Phi_X)
        = \int_{B_{\nu}^{d_1}\times B_{\nu}^{d_2}} | \phi(t_1,t_2) \Phi_X(t_1,0) \Phi_X(0,t_2) - \Phi_X(t_1,t_2) \phi(t_1,0) \phi(0,t_2) |^2 \\
            \times |\Phi_{\varepsilon^{(1)}}(t_1)\Phi_{\varepsilon^{(2)}}(t_2) |^2 d t_1 d t_2,
    \end{multline*}
    \begin{equation*}
    M_{n}(\phi) = \int_{B_{\nuest}^{d_1}\times B_{\nuest}^{d_2}} | \phi(t_1,t_2) \tilde\phi_n(t_1,0) \tilde\phi_n(0,t_2) - \tilde \phi_n(t_1,t_2) \phi(t_1,0) \phi(0,t_2) |^2 d t_1 d t_2.
    \end{equation*}
    \item For $S>0$ and $\rho > 0$, $\Upsilon_{\rho,S} = \lbrace 
\phi \text{ analytic } \text{s.t. } \forall z\in \Rbb^{D},\overline{\phi(z)}=\phi(-z), \phi(0) = 1 \text{ and } \forall i \in \Nbb^D \setminus \{0\}, \left| \frac{\partial^{i} \phi(0)}{\prod_{a=1}^d i_a!}\right| \leq \frac{S^{\|i\|_1}}{\|i\|_1^{ \|i\|_1 /\rho}} \rbrace$. Section~\ref{subsec:chara}.
    \item $\Hcal$: a subset of functions $\Cbb^D \rightarrow \Cbb^D$ such that all elements of $\Hcal$ satisfy (Adep) and such that the set of the restrictions to $[-\nuest,\nuest]^D$ of functions in $\Hcal$ is closed in $L^{2}([-\nuest,\nuest]^D)$. Section~\ref{subsec:chara}.
    \item For any integer $m$ and any $\rho > 1$, $\widehat \Phi_{n,m,\rho}$ is a (up to $1/n$) measurable minimizer of the functional $\phi \mapsto M_n(T_m \phi)$ over $\Upsilon_{\rho,S} \cap \Hcal$. Section~\ref{subsec:chara}.
    \item $c_{\nu}=\inf\{|\Phi_{\varepsilon^{(1)}}(t)|,\;t\in [-\nu,\nu]^{d_1}\} \wedge 
    \inf\{|\Phi_{\varepsilon^{(2)}}(t)|,\;t\in [-\nu,\nu]^{d_2}\}$. Section~\ref{subsec:chara}.
    \item For any $\nu>0$, $c(\nu ) >0$, $E>0$,
$\Qcal^{(D)} (\nu,c({\nu}),E)$ is the set of distributions $\mathbb{Q}=\otimes_{j=1}^2 \mathbb{Q}_{j}$ on $\Rbb^D$ such that $c_{\nu}\geq c(\nu)$ and $\int_{\Rbb^D} \|x\|^2 d\mathbb{Q}(x)\leq E$. Section~\ref{subsec:chara}.
    \item $\kappa = 1/\rho$. Section~\ref{subsec:Mupper}.
    \item $d_H(A_1,A_2)$ is the Hausdorff distance between $A_1$ and $A_2$ subsets of $\Rbb^D$. Section~\ref{subsec:Mupper}.
    \item $\Kcal$: a compact subset of $\Rbb^D$ (which can be arbitrarily large).  Section~\ref{subsec:Mupper}.
    \item For all $S_1,S_2$ subsets of $\Rbb^D$, $H_{\Kcal}(S_1,S_2) = d_H(S_1 \cap \Kcal, S_2 \cap \Kcal)$. Section~\ref{subsec:Mupper}.
    \item $\text{St}_{\Kcal}(a,d,r_0)$: set of positive measures $G$ such that for all $x \in \Mcal_G \cap \Kcal$, for all $r \leq r_0$, $G(B(x,r)) \geq a r^d$. The distributions in $St_{\Kcal}(a,d,r_0)$ are called $(a,d)$-standard. Section~\ref{subsec:Mupper}.
    \item $\Lcal(\kappa,S,\Hcal)$:  set of distributions $G$ such that, if $X$ is a random variable with distribution $G$, then $\Phi_X \in \Hcal \cap \Upsilon_{1/\kappa,S}$. Section~\ref{subsec:Mupper}.
    \item Fourier and inverse Fourier transform : when $f$ is an integrable function from $\Rbb^D$ to $\Rbb$, for all $y \in \Rbb^D$, $\Fcal[f](y) = \int e^{it^{\top}y} f(t) dt \ \ \text{and} \ \ \Fcal^{-1}[f](y) = (\frac{1}{2 \pi})^D \int e^{-it^{\top}y} f(t) dt$. End of Section~\ref{overview}.
    \item For any $A>0$, for all $y \in \Rbb$ $u_{A}(y) = \exp\left\{ -\frac{1}{(1-2y)^{A}} -\frac{1}{(1+2y)^{A}}\right\} 1|_{[-\frac{1}{2},\frac{1}{2}]}(y)$. Section~\ref{subsec:Mupper}.
    \item For any $A>0$ and $y \in \Rbb^D$, $\psi_A(y) = I(A) \ \Fcal^{-1}[u_{A} * u_{A}](\|y\|_2)  \ \text{with} \ I(A) = \frac{1}{\int \Fcal^{-1} [ u_{A} * u_{A} ](\|x\|_2) dx}$ and $\psi_{A,h}(x) = h^{-D} \psi_A(x/h)$. Section~\ref{subsec:Mupper}.
    \item The constants $c_{A}>0$ and $d_{A}>0$: for all $x \in \{y\in\Rbb^D \,: \, \|y\|_{2}\leq c_{A}\}$, $\psi_A(x) \geq d_{A}$. Section~\ref{subsec:Mupper}.
    \item For all $y \in \Rbb^D$, $\bar{g}(y) = (\frac{1}{2 \pi})^{D} \int e^{-it^{\top}y} \Fcal[\psi_A](ht) \  \Phi_X(t) dt $. Section~\ref{subsec:Mupper}.
    \item For all $y \in \Rbb^d$, $\widehat{g}_{n,\kappa}(y) = \left(\frac{1}{2 \pi}\right)^{D} \int e^{-it^{\top}y} \Fcal[\psi_A] (ht) \ T_{m_{\kappa}} \widehat{\Phi}_{n,1/\kappa}(t) dt$. Section~\ref{subsec:Mupper}.
    \item $\widehat{\Mcal}_{\kappa} = \left\{ y \in \Rbb^D \ | \ \widehat{g}_{n,\kappa}(y) > \lambda_{n,\kappa} \right\}$. Section~\ref{subsec:Mupper}.
    \item $\Gamma_{n,\kappa} = \|\widehat{g}_{n,\kappa} - \bar{g}\|_{\infty} = \sup_{y \in \Rbb^D} |\widehat{g}_{n,\kappa}(y) - \bar{g}(y)|$. Section~\ref{subsec:Mupper}.
    \item For any $\eta>0$, $A_{\eta}$ denotes the $\eta$-offset of $A$. Section~\ref{sec:Wasserstein}.
    \item For any borelian set $\Ocal$ of $\Rbb^D$, $P_{\psi_{A,h_n}}(\Ocal) = \int_{\Ocal} \bar{g}(y) dy$. Section~\ref{sec:Wasserstein}.
    \item $\widehat{g}_n:=\widehat{g}_{n,1}$ and $\widehat{g}_n^{+} = \max{ \{ 0,\widehat{g}_n \} }$. Section~\ref{sec:Wasserstein}.
    \item For any borelian set $\Ocal$ of $\Rbb^D$, $
    \widehat{P}_{n,\eta}(\Ocal)
    = \frac{1}{\int_{(\widehat{\Mcal} \cap \bar{B}(0,R_n))_{{\eta}}} \widehat{g}_n^{+}(y) dy} \int_{\Ocal \cap (\widehat{\Mcal} \cap \bar{B}(0,R_n))_{{\eta}}} \widehat{g}_n^{+}(y) dy
    = c_n \int_{\Ocal \cap (\widehat{\Mcal} \cap \bar{B}(0,R_n))_{{\eta}}} \widehat{g}_n^{+}(y) dy$.
    \item For any $p \in [1,+\infty)$ and any two probability measures $\mu$ and $\nu$ on $\Rbb^D$, $W_p(\mu,\nu)$ is the Wasserstein distance of order $p$ between $\mu$ and $\nu$. Section~\ref{sec_Wasserstein_upper}.
    \item For $k,l \in \{1, \dots, D\}$ with $k \leq l$, $\pi^{(k:l)}: (x_1, \dots, x_D) \in \Rbb^D  \mapsto (x_k, \dots, x_l) \in \Rbb^{l-k+1}$ and $\pi^{(k)} = \pi^{(k:k)}$. Section~\ref{subsec:geocondi}.
\end{itemize}


\section{Proofs}
\label{appC}

\subsection{Proof of Proposition~\ref{prop:A(rho)}}
\label{proof:A(rho)}

\paragraph*{Case $\rho = 1$}

It is clear that any compactly supported distribution satisfies A(1). Conversely, if $\Ebb[e^{\langle \lambda, X\rangle}] \leq a \exp(b \|\lambda\|_2)$, then for any $\mu > 0$,  we get, for any $b' > b$, if we denote $(e_{j})_{1\leq j \leq D}$ the canonical basis of $\Rbb^D$,
\begin{align*}
\Pbb(\|X\|_2 \geq Db')
    &\leq \sum_{j=1}^{D}\Pbb(|X_{j}| \geq b') \\
    &= \sum_{j=1}^{D}\left\{\Pbb(X_{j} \geq b') + \Pbb(X_{j} \leq - b')\right\} \\
    &= \sum_{j=1}^{D}\left\{\Pbb(\langle \mu e_{j},X \rangle \geq \mu b') + \Pbb(-\langle \mu e_{j},X \rangle \geq  \mu b')\right\} \\
    &\leq \sum_{j=1}^{D}\left\{\frac{\Ebb\left[\exp (\langle \mu e_{j},X \rangle)\right]}{\exp(b' \mu)}
     +\frac{\Ebb\left[\exp (-\langle \mu e_{j},X \rangle)\right]}{\exp(b' \mu)}
     \right\} \quad \text{by Markov inequality} \\
    &\leq 2D \frac{a \exp(b \mu)}{\exp(b' \mu)} \underset{\mu \rightarrow +\infty}{\longrightarrow} 0,
\end{align*}
and hence $\|X\|_2 \leq Db$ almost surely.

\paragraph*{Case $\rho > 1$}

Assume that for any $\lambda \in \Rbb^D$, $\Ebb[e^{\langle \lambda, X \rangle}] \leq a \exp(b \|\lambda\|_2^\rho)$ for some $a,b>0$. Then by using the same directional method as for $\rho=1$, we get that for any $\mu,t \geq 0$,
\begin{align*}
\Pbb(\|X\|_2 \geq t)
    &\leq 2D a \exp(b \mu^\rho - \mu t) \\
    &= 2D a \exp\left(-\left(\frac{\rho} {b}\right)^{\frac{1}{\rho-1}}
    \left(1-\rho^{-\frac{\rho+1}{\rho-1}}\right)t^{\frac{\rho}{\rho-1}}  \right) \quad \text{by taking } \mu = \left( \frac{t}{b\rho} \right)^{\frac{1}{\rho-1}}.
\end{align*}
Observe that since $\rho >1$,
$\left(1-\rho^{-\frac{\rho+1}{\rho-1}}\right)>0$ to get the result.\\

Now, assume that for any $t \geq 0$, $\Pbb(\|X\|_2 \geq t) \leq c \exp(-d t^{\rho / (\rho-1)})$ for some $c,d>0$, then by the Cauchy-Schwarz inequality, for any $\lambda \in \Rbb^D$,
\begin{equation*}
\Ebb[e^{\langle \lambda, X \rangle}]
    \leq \Ebb[e^{\|\lambda\|_2 \|X\|_2}].
\end{equation*}
Then, using that for any nonnegative random variable $Y$, $\Ebb[Y] = \int_{t \geq 0} \Pbb(Y \geq t) dt$,
\begin{align*}
\Ebb[e^{\langle \lambda, X \rangle}]
    &\leq 1 + \int_{t \geq 1} \Pbb(e^{\|\lambda\|_2 \|X\|_2} \geq t) dt \\
    &\leq 1 + \int_{s \geq 0} \Pbb(\|X\|_2 \geq s) \|\lambda\|_2\exp(\|\lambda\|_2s) ds \quad \text{with } t = e^{\|\lambda\|_2 s} \\
    &\leq 1 + c \|\lambda\|_2 \int_{s \geq 0} \exp(-d s^{\frac{\rho}{\rho-1}} + \|\lambda\|_2s) ds\\
    &\leq 1 + c \|\lambda\|_2^{\rho} \int_{s' \geq 0} \exp(\|\lambda\|_2^\rho(-d s'^{\frac{\rho}{\rho-1}} + s')) ds' \quad \text{with } s = s' \|\lambda\|_2^{\rho-1}.
\end{align*}
Note that $-d s'^{\frac{\rho}{\rho-1}} + s' \leq \frac{1}{\rho} (\frac{\rho-1}{\rho d})^{\rho-1}$ for any $s' \geq 0$, and $-d s'^{\frac{\rho}{\rho-1}} + s' \leq -s'$ when $s' \geq (\frac{2}{d})^{\rho-1}$. In particular,
\begin{align*}
\Ebb[e^{\langle \lambda, X \rangle}]
    &\leq 1
        + c \|\lambda\|_2^{\rho} \int_{s'=0}^{(\frac{2}{d})^{\rho-1}} \exp(\|\lambda\|_2^\rho(-d s'^{\frac{\rho}{\rho-1}} + s')) ds' \\
    &\qquad + c \|\lambda\|_2^{\rho} \int_{s' \geq  (\frac{2}{d})^{\rho-1}} \exp(\|\lambda\|_2^\rho(-d s'^{\frac{\rho}{\rho-1}} + s')) ds' \\
    &\leq 1
        + c \|\lambda\|_2^{\rho} \left(\frac{2}{d}\right)^{\rho-1} \exp\left( \frac{\|\lambda\|_2^\rho}{\rho} \left(\frac{\rho-1}{\rho d}\right)^{\rho-1} \right)
        + c \exp\left(-\|\lambda\|_2^\rho \left(\frac{2}{d}\right)^{\rho-1} \right),
\end{align*}
which proves that $A(\rho)$ holds.

\subsection{Proof of Proposition~\ref{prop:(Adep)prop}}
\label{proof:(Adep)}

First, note that if $U$ and $V$ are independent random variables satisfying A($\rho$) then $U+V$ satisfies also A($\rho$) with the same constant $\rho$.

\begin{itemize}
    \item[(i)] If $U$ and $V$ are independent, then for all $(z_1,z_2) \in \Cbb^{d_{1}} \times \Cbb^{d_{2}}$,
    \begin{equation}
    \label{eq:U+V}
        \Phi_{U+V}(z_1,z_2) = \Phi_{U}(z_1,z_2) \Phi_{V}(z_1,z_2).
    \end{equation}
    Assume first that $U$ and $V$ satisfy (Adep). Suppose that there exists $z_0 \in \Cbb^{d_{1}}$ such that for all $z \in \Cbb^{d_{2}}$, $\Phi_{U+V}(z_0,z) = 0$. Then for all $z \in \Cbb^{d_{2}}$,
    \begin{equation*}
        \Phi_{U}(z_0,z) \Phi_{V}(z_0,z) = 0.
    \end{equation*}
    If $Z_U^{(1)}(z_0) = \{ z \in \Cbb^{d_{2}} \, | \, \Phi_X(z_0,z)=0 \}$ and $Z_V^{(1)}(z_0) = \{ z \in \Cbb^{d_{2}} \, | \, \Phi_Y(z_0,z)=0 \}$, $Z_U^{(1)}(z_0) \cup Z_V^{(1)}(z_0) = \Cbb^{d_{2}}$. Since $\Phi_U(z_0, \cdot)$ and $\Phi_V(z_0, \cdot)$ are not the null functions, Corollary 10 of~\cite{MR2568219}, p. 9, implies that $Z_U^{(1)}(z_0) \cup Z_V^{(1)}(z_0)$ has zero $2 d_2$-Lebesgue measure, which contradicts the fact that $Z_U^{(1)}(z_0) \cup Z_V^{(1)}(z_0) = \Cbb^{d_{2}}$. If instead we suppose that there exists $z_0 \in \Cbb^{d_{2}}$ such that for all $z \in \Cbb^{d_{1}}$, $\Phi_{U+V}(z,z_0) = 0$, analogous arguments lead to a contradiction. Thus $U+V$ satisfies (Adep).
   
    Assume now that $U+V$ satisfies (Adep). Then~\eqref{eq:U+V} implies that $\Phi_{U}(z_1,\cdot)$, $\Phi_{V}(z_1,\cdot)$, $\Phi_{U}(\cdot,z_2)$, $\Phi_{V}(\cdot,z_2)$ can not be the null function, so that $U$ and $V$ both satisfy (Adep).

    \item[(ii)]
    First, we remind the definition of the operator norm, let $M \in \Rbb^{D \times D}$, then, $\|M\|_{op} = \sup_{\|x\|_2 \neq 0} \frac{\|Mx\|_2}{\|x\|_2}$.
    Assume that $U$ satisfies $A(\rho)$ with constants $a$ and $b$. Then, for any $\lambda \in \Rbb^{D}$,
    \begin{align*}
        \Ebb\left[\exp{(\lambda^\top V)}\right] &= \Ebb\left[\exp{\left( \lambda^\top\begin{pmatrix} A & 0 \\ 0 & B \end{pmatrix} \begin{pmatrix} U^{(1)}\\ U^{(2)}\end{pmatrix} + \lambda^\top \begin{pmatrix} m_1\\m_2 \end{pmatrix} \right) }\right]\\
        & \leq a\exp{\left(b \left\|\lambda^\top\begin{pmatrix} A & 0 \\ 0 & B \end{pmatrix}
        \right\|_{2}^{\rho}
        +\|\lambda\|_{2} \left\|\begin{pmatrix} m_1\\m_2 \end{pmatrix}\right\|_{2}\right) } \\
        & \leq a\exp{\left(b \left\|\begin{pmatrix} A & 0 \\ 0 & B \end{pmatrix}
        \right\|_{op}^{\rho}\|\lambda\|_{2}^{\rho}
        +\|\lambda\|_{2} \left\|\begin{pmatrix} m_1\\m_2 \end{pmatrix}\right\|_{2}\right) }.
    \end{align*}
    Since $\rho \geq 1$, $\|\lambda\|_{2}\leq \|\lambda\|_{2}^{\rho}$ for $\|\lambda\|_{2}\geq 1$, so that if $U$ satisfies $A(\rho)$ with constants $a$ and $b$, then $V$ satisfies $A(\rho)$ with constants $a \exp \left(\left\|\begin{pmatrix} m_1\\m_2 \end{pmatrix}\right\|_{2}\right) $ and $b \left\|\begin{pmatrix} A & 0 \\ 0 & B \end{pmatrix}
        \right\|_{op}^{\rho}+\left\|\begin{pmatrix} m_1\\m_2 \end{pmatrix}\right\|_{2}$.
    The converse follows from applying the direct proof to $V$ with  
    $-\begin{pmatrix} A^{-1} & 0 \\ 0 & B^{-1} \end{pmatrix} \begin{pmatrix} m_1 \\ m_2 \end{pmatrix}$ and $\begin{pmatrix} A^{-1} & 0 \\ 0 & B^{-1} \end{pmatrix}$.
    
    Now, for all $(z_1,z_2) \in \Cbb^{d_{1}}\times \Cbb^{d_{2}}$,
    \begin{equation*}
    \Phi_{V}(z_1,z_2) = \exp{\left( \lambda^\top  \begin{pmatrix} m_1 \\ m_2 \end{pmatrix} \right)} \Phi_U(A^\top z_1, B^\top z_2)
    \end{equation*}
    and
    \begin{equation*}
    \Phi_{U}(z_1,z_2) = \exp{\left( - \lambda^\top \begin{pmatrix} A^{-1} & 0 \\ 0 & B^{-1} \end{pmatrix} \begin{pmatrix} m_1 \\ m_2 \end{pmatrix} \right)} \Phi_{V}((A^{-1})^\top z_1, (B^{-1})^\top z_2),
    \end{equation*}
    so that $U$ verifies (Adep) if and only if $V$ verifies (Adep).
    
    \item[(iii)]
    Since $U^{(1)}$ and $U^{(2)}$ are independent, for all $z_1 \in \Cbb^{d_1}$ and $z_2 \in \Cbb^{d_2}$, $\Phi_U(z_1,z_2) = \Phi_{U^{(1)}}(z_1) \Phi_{U^{(2)}}(z_2)$. Thus if $U^{(1)}$ and $U^{(2)}$ are deterministic or Gaussian random variables, $U$ satisfies (Adep). Conversely,
    if $U$ satisfies (Adep), then  neither $\Phi_{U^{(1)}}$ nor $\Phi_{U^{(2)}}$ have any zeros. By Hadamard's factorization Theorem (See \cite{AnaComplexBook}, Chapter 5, Theorem 5.1) together with A($\rho$), reasoning variable by variable we obtain that $\Phi_{U^{(1)}} = \exp(P_1)$ and $\Phi_{U^{(2)}} = \exp(P_2)$ for some polynomials $P_1$ and $P_2$ with degree bounded by $\rho$ in each variable. Now, for $j=1,2$, for any $\lambda \in \Rbb^{d_j}$, $t\mapsto \Phi_{U^{(j)}}(t\lambda)$ is the characteristic of the random variable $\langle\lambda,X^{(j)}\rangle$ and writes $\exp(P_{j} (t\lambda))$. But by Marcinkiewicz's theorem 2bis in~\cite{MR1545791}, this implies that $t\mapsto P_{j} (t\lambda)$ is of degree at most two. Since this is true for any $\lambda$, we get that $P_1$ and $P_2$ are polynomials with total degree at most two.
    Thus the polynomials $P_1$ and $P_2$ are of the form $i\langle A, X\rangle - \frac{1}{2} X^\top B X$ for some symmetric matrix $B$ since characteristic functions are equal to 1 at zero and $\Phi_{U^{(j)}}(-z) = \overline{\Phi_{U^{(j)}}(z)}$ for all $z \in \Rbb^{d_{j}}$ for $j = 1,2$. 
    Therefore the distribution of $U_1$ (resp. $U_2$) is a (possibly singular) Gaussian distribution.
\end{itemize}


\subsection{Proof of Theorem~\ref{thm:id}}
\label{proof:thmid}

Consider a random variable $X$ satisfying A($\rho$). Theorem~\ref{thm:id} is a direct consequence of the following Lemma \ref{lemma:id}. Indeed, for any $z_{0}\in \Cbb^{d_{1}}$ and $z\in \Cbb^{d_{2}}$,
\begin{equation*}
\Ebb \left[\exp\left(iz_{0}^\top X^{(1)}+iz^\top X^{(2)}\right)\right]
    = \Ebb \left[
            \Ebb\left[\exp\left(iz_{0}^\top X^{(1)}\right) \, | \, X^{(2)}\right]
            \exp\left(iz^\top X^{(2)}\right)
        \right],
\end{equation*}
and we easily show that ${z\mapsto \Ebb [\exp(iz_{0}^\top X^{(1)}+iz^\top X^{(2)})]}$ is the null function if and only if $\Ebb [\exp(iz_{0}^\top X^{(1)}) \, | \, X^{(2)}]$ is zero $\Pbb_{X^{(2)}}$-a.s.
Define $h$ the measurable function such that 
$$
\Ebb\left[\exp\left(iz_{0}^\top X^{(1)}\right) \, | \, X^{(2)}\right]=h\left(X^{(2)}\right)\;\;\Pbb_{X^{(2)}}-a.s.,
$$
and assume that for all $z\in \Cbb^{d_{2}}$,
$
            \Ebb\left[h\left(X^{(2)}\right)
            \exp\left(iz^\top X^{(2)}\right)
        \right]=0.$
Denote for all $x\in\Rbb^{d_2}$, $h_+ (x)= \max{\lbrace h(x),0 \rbrace}$ and $h_- (x)= \max{ \lbrace -h(x),0 \rbrace}$. We get for all $z\in \Cbb^{d_{2}}$,
$$         \Ebb\left[h_{+}\left(X^{(2)}\right)
            \exp\left(iz^\top X^{(2)}\right)
        \right]=\Ebb\left[h_{-}\left(X^{(2)}\right)
            \exp\left(iz^\top X^{(2)}\right)
        \right].$$
In particular, the probability distributions having density $h_{+}$ and $h_{-}$ with respect to $\Pbb_{X^{(2)}}$ (up to the same normalizing constant) have equal characteristic functions, thus are the same probability distributions, so that $h_{+}\left(X^{(2)}\right)=h_{-}\left(X^{(2)}\right)$, $\Pbb_{X^{(2)}}$-a.s., which implies $h\left(X^{(2)}\right)=0$, $\Pbb_{X^{(2)}}$-a.s.

Likewise, for any ${z_{0}\in\Cbb^{d_{2}}}$, $z\mapsto \Ebb [\exp(iz^\top X^{(1)}+iz_{0}^\top X^{(2)})]$ is the null function if and only if $\Ebb [\exp(iz_{0}^{\top} X^{(2)}) \, | \, X^{(1)}]$ is zero $\Pbb_{X^{(1)}}$-a.s.

\begin{lemma}
\label{lemma:id}
Assume (H1) and (H2). Then, for all $z \in \Cbb^{d_{1}}$, $\Ebb[ \exp{( i z^{\top} X^{(1)} )} \, | \, X^{(2)} ]$ is not $\Pbb_{X^{(2)}}$-a.s. the null random variable and for all $z \in \Cbb^{d_{2}}$, $\Ebb[ \exp{( i z^{\top} X^{(2)} )} \, | \, X^{(1)} ]$ is not $\Pbb_{X^{(1)}}$-a.s. the null random variable 
\end{lemma}

\paragraph*{Proof of Lemma~\ref{lemma:id}}

To begin with, by Proposition~\ref{prop:(Adep)prop}, we may assume without loss of generality that $0 \in B_{\Delta}$ in (H1) and (H2) (up to translation of $X$).

Let $z \in \Cbb^{d_{1}}$ be such that $\Ebb[ \exp{( i z^\top X^{(1)} )} | X^{(2)}] $ is $\Pbb_{X^{(2)}}$-a.s. the null random variable.
Then for any $\Delta > 0$, if we denote $A_{\Delta}$ a set given by (H1), $\Ebb [\exp{( i z^\top X^{(1)} )} | X^{(2)}] 1|_{X^{(2)} \in A_{\Delta}} = 0$ $\Pbb_{X_1^{(2)}}$ a.s., and taking the real part of this equation shows that
\begin{equation}
\label{eq:contrad}
    \Ebb [\cos(\text{Re}(z)^\top X^{(1)})\exp{( - \text{Im}(z)^\top X^{(1)} )}  | X^{(2)} ] 1|_{X^{(2)} \in A_{\Delta}} = 0 \quad \Pbb_{X^{(2)}}  \text{a.s}.
\end{equation}
Using (H1), we can fix $\Delta >0$ small enough such that if $x\in B_\Delta$, $\cos(\text{Re}(z)^\top x) > 0$. But for such $\Delta$, Equation~\eqref{eq:contrad} can not hold since $\Pbb(X^{(1)} \in B_\Delta \,|\, X^{(2)} \in A_\Delta) = 1$. Thus $\Ebb[ \exp{( i z^\top X^{(1)} )} | X^{(2)}] $ is not $\Pbb_{X^{(2)}}$-a.s. the null random variable.

The proof of the other part of Lemma~\ref{lemma:id} is analogous using (H2).

\subsection{Proof of Theorem~\ref{prop:idgeneral}}
\label{proof:idgeneral}

Let $\Mcal$ be a compact subset of $\Rbb^D$.
Let us first prove that the function $u \longmapsto \Diam(\{u\}\times \Rbb^{d_2} \cap \Mcal)$ is upper semi-continuous.

Let $u \in \Rbb^{d_{1}}$.
Since $\Mcal$ is compact, there exist sequences $u_n \rightarrow u$ and
$(x_{n}$, $y_{n})$ in $(\{u_n\}\times \Rbb^{d_2} \cap \Mcal)$  such that
$\|x_{n}-y_{n}\|_{2}=\Diam(\{u_n\}\times \Rbb^{d_2} \cap \Mcal)$ and $\lim_{n\rightarrow +\infty}\|x_{n}-y_{n}\|_{2}=\limsup_{v \rightarrow u} \Diam(\{v\}\times \Rbb^{d_2} \cap \Mcal)$.
Moreover, we may assume that there exists $(x, y)$ in $(\{u\}\times \Rbb^{d_2} \cap \Mcal)$ such that $x_n \rightarrow x$ and $y_n \rightarrow y$.
Taking the limit along those sequences shows that $\Diam(\{u\}\times \Rbb^{d_2} \cap \Mcal)\geq \|x-y\| =\limsup_{v \rightarrow u} \Diam(\{v\}\times \Rbb^{d_2} \cap \Mcal)$, proving the claimed upper-semi continuity.

Now, since $\Mcal$ is compact, there exists $R>0$ such that $\Mcal\subset \bar{B}(0,R)$. If moreover $\Mcal \in \Bcal_1$, there exists $x_{1}\in\Rbb^{d_1}$ such that $\Diam(\{x_{1}\}\times \Rbb^{d_2} \cap \Mcal) = 0$. Using the upper semi-continuity shows that $\Mcal \in \cap_{n \geq 1} \Acal_2(1/n, R)$.
Likewise, if $\Mcal \in \Bcal_2$, there exists $x_{2}\in\Rbb^{d_2}$ such that $\Diam(\Rbb^{d_1}\times \{x_{2}\} \cap \Mcal) = 0$ and $\Mcal \in \cap_{n \geq 1} \Acal_1(1/n, R)$.
 
The end of the proof follows from Theorem~\ref{prop:idH} and the fact that any random variable with compact support satisfies A(1).

\section{Genericity of (H1) and (H2)}
\label{sec:genericity}

\subsection{Genericity results}
\label{sec:genericity:resuts}

The main purpose of this subsection is to show that hypotheses (H1) and (H2) are verified generically.

First, we show that while the set of supports satisfying Theorem~\ref{prop:idH} is dense in the set of closed sets of $\Rbb^D$, its complement is also dense.

\begin{proposition}
\label{prop:dense}
    The set $\Acal = \left(\cap_{\Delta > 0} \cup_{\varepsilon > 0} \Acal_1(\Delta,\varepsilon)\right) \bigcap \left(\cap_{\Delta > 0} \cup_{\varepsilon > 0} \Acal_2(\Delta,\varepsilon)\right)$ and its complement are dense in the set of closed subsets of $\Rbb^D$ endowed with the Hausdorff distance.
\end{proposition}

The proof of Proposition~\ref{prop:dense} is detailed in Section~\ref{sec_proof_prop:dense}.

This proposition shows that any support $\Mcal$ can be altered by a small perturbation to produce both supports that satisfy (H1) and (H2) and supports that satisfy neither. \textit{A fortiori}, the same is true for (Adep), as on one hand (H1), (H2) and A($\rho$) ensure (Adep) by Theorem~\ref{thm:id} and on the other hand a small perturbation of the signal is enough to no longer satisfy (Adep) by Point (i) of Proposition~\ref{prop:(Adep)prop}.

\paragraph*{Topological genericity}
The sets $\cup_{\varepsilon > 0} \Acal_1(\Delta,\varepsilon)$ are open with respect to the Hausdorff distance, and are increasing in $\Delta$. Therefore, $\Acal$ is a $G_\delta$ set, and by the Proposition above, it is also dense. Equivalently, this means that the complement of $\Acal$ is a meagre set, and therefore the set of all supports satisfying (H1) and (H2) is comeagre, so these assumptions hold generically in the comeagre or $G_\delta$-dense sense.

\paragraph*{Measure theoretical genericity}

Similarly to how ``almost everywhere" (with respect to the Lebesgue measure) is another possible notion of genericity in $\Rbb^D$, we construct a random and small perturbation of $\Rbb^D$ such that any compact set is almost surely transformed into a compact set in $\Bcal_1 \cap \Bcal_2$.

More precisely, for any $\varepsilon > 0$, we define a (random) continuous bijection $f : \Rbb^D \longrightarrow \Rbb^D$ such that almost surely, $|f(x) - x| \leq \varepsilon$ for all $x \in \Rbb^D$, and such that if $\Mcal$ is compact, then $f(\Mcal)$ is in $\Bcal_1 \cap \Bcal_2$ almost surely. This random bijection does not depend on which support $\Mcal$ is considered, and can for instance be seen as a modeling of the imperfections of ``realistic" supports, or as a way to introduce a Bayesian prior on the support. In that sense, compact supports are almost surely in $\Bcal_1 \cap \Bcal_2$, and thus compactly supported random variables almost surely satisfy (Adep).

There is no canonical way to define a random perturbation of $\Rbb^D$. Our approach is to tile the space with simplices, then add a small perturbation to each vertex of the tiling, keeping the transformation linear inside each simplex. Visually, this results in a small, random crumpling of the Euclidean space.

\paragraph*{Simplicial tiling of $\Rbb^D$}

Let us recall a few definitions about simplicial complexes. For any $k \in \{0, \dots, D\}$, a $k$-simplex of $\Rbb^D$ is the convex hull of $(k+1)$ affinely independent points of $\Rbb^D$.
A simplicial complex $\Pcal$ is a set of simplices such that every face of a simplex from $\Pcal$ is also in $\Pcal$, and the non-empty intersection of any two simplices $F_1, F_2 \in \Pcal$ is a face of both $F_1$ and $F_2$. $\Pcal$ is a homogeneous simplicial $D$-complex if each simplex of dimension less than $D$ of $\Pcal$ is the face of a $D$-simplex of of $\Pcal$. For any simplex $F$, we write $\Int(F)$ its relative interior.
Finally, a homogeneous simplicial $D$-complex $\Pcal$ is called a simplicial tiling of $A \subset \Rbb^D$ if the relative interior of its simplices form a partition of $A$. Note that the facets of $\Pcal$, that is, its $D$-simplices, do not necessarily form a partition of $A$: two facets can have a non-empty intersection when they share a face.

First, consider a finite simplicial tiling of the hypercube $[0,1]^D$, and extend it to $\Rbb^D$ by mirroring it along the hyperplanes orthogonal to the canonical axes crossing them at integer coordinates. Formally, for any $k = (k_1, \dots, k_D) \in \Zbb^D$, the hypercube $\prod_{i=1}^D [k_i, k_i + 1]$ contains the tiling of $[0,1]^D$, mirrored along axis $i$ if and only if $k_i$ is odd. The faces of the hypercubes defined in this way match, as each pair of hypercubes sharing a face are mirrors of each other with respect to that face.
Thus, the resulting tiling $\Pcal$ is a simplicial tiling of $\Rbb^D$.

Let $(x_n)_{n \in \Nbb}$ be the sequence of vertices of the simplicial tiling $\Pcal$ (i.e. its 0-simplices). We identify each simplex $F \in \Pcal$ with the set of its 0-dimensional faces $\{x_i\}_{i \in I}$, and write $F_I$ in that case. Note that the set $I$ is unique for any given simplex $F$ and characterizes $F$.

\paragraph*{Perturbation of the tiling}

Fix a small $r > 0$. Let $(\varepsilon_n)_{n \in \Nbb}$ be a sequence of i.i.d. uniform variables on $[-r,r]^D$, and define $\Pcal^\varepsilon$ the simplicial complex defined by
\begin{equation*}
    \Pcal^\varepsilon = \{ \{x_i + \varepsilon_i\}_{i \in I} : \{x_i\}_{i \in I} \in \Pcal \}.
\end{equation*}

Note that since the original tiling of $[0,1]^D$ was finite, there exists $r_0 > 0$ such that for any $(\varepsilon_n)_{n \in \Nbb} \in ([-r_0,r_0]^D)^\Nbb$, the vertices of any simplex in $\Pcal$ are still affinely independent after being moved according to $\varepsilon$ and any two simplices $F, F' \in \Pcal$ sharing a face $F''$ (resp. with no intersection) are transformed into two simplices of $\Pcal^\varepsilon$ that share exactly the transformation of $F''$ (resp. with no intersection), so that $\Pcal^\varepsilon$ is indeed a simplicial complex.
Finally, $\Pcal^\varepsilon$ still covers $\Rbb^D$ (as seen when moving each vertex in $[-1,2]^D$ one after the other along a continuous path, showing that no hole is created in the covering of $[0,1]^D$ at any point in time), so for any $r \in (0,r_0]$, $\Pcal^\varepsilon$ is almost surely a simplicial tiling of $\Rbb^D$.

Since the relative interiors of the simplices of $\Pcal$ define a partition of $\Rbb^D$, for each $z \in \Rbb^D$, there exists exactly one face $F_I \in \Pcal$ such that $z \in \Int(F_I)$. Writing $z = \sum_{i \in I} \alpha_i x_i$ (for $\alpha \in (0,1]^{|I|}$ such that $\sum_{i \in I} \alpha_i = 1$), we define the image of $z$ by the perturbation as $f^\varepsilon(z) = \sum_{i \in I} \alpha_i (x_i + \varepsilon_i)$. In other words, each simplex is deformed according to the linear transformation given by the perturbation of its vertices.

The mapping $f^\varepsilon$ is a (random) bijective and continuous transformation of $\Rbb^D$ that is ``small", in the sense that almost surely, $\sup_{z \in \Rbb^D} \|z - f^\varepsilon(z)\| \leq r$.

Note that the transformation $f^\varepsilon$ can be made with arbitrarily small granularity: the same approach works when considering tilings of $[0,\delta]^D$ for any $\delta > 0$ instead of $[0,1]^D$ (up to changing $r$). We may also iterate several random independent transformations $f^{\varepsilon^{(1)}} \circ \dots \circ f^{\varepsilon^{(m)}}$ for $m \geq 1$, and the transformation of $\Mcal$ will still almost surely belong to $\Bcal_1 \cap \Bcal_2$.

\begin{theorem}
\label{th_genericity}
Let $r \in (0,r_0]$ with $r_0$ as above, $\varepsilon = (\varepsilon_n)_{n \in \Nbb}$ be a sequence of i.i.d. uniform r.v. on $[-r,r]^D$, $\delta > 0$, and $f^{\varepsilon}$ be the bijective transformation of $\Rbb^D$ defined above.

Then for any (random) continuous mapping $H : \Rbb^D \rightarrow \Rbb^D$ that is independent of $\varepsilon$, the mapping $F : z \longmapsto \delta f^{\varepsilon}(\frac{H(z)}{\delta})$ satisfies: for any compact set $\Mcal \subset \Rbb^D$, $F(\Mcal) \in \Bcal_1 \cap \Bcal_2$ a.s..
\end{theorem}

The proof of Theorem~\ref{th_genericity} is detailed in Section~\ref{app_proof_genericity}.

This shows that for any compact set $\Mcal \in \Rbb^D$, a small change into the set $F(\Mcal)$ where $F$ is a transformation of $\Rbb^D$ of the type described in the Theorem almost surely results in a set in $\Bcal_1 \cap \Bcal_2$.


\subsection{Proof of Proposition~\ref{prop:dense}}
\label{sec_proof_prop:dense}

First, let us show that the set
\begin{equation*}
\Acal := \left(\cap_{\Delta > 0} \cup_{\varepsilon > 0} \Acal_1(\Delta,\varepsilon)\right) \bigcap \left(\cap_{\Delta > 0} \cup_{\varepsilon > 0} \Acal_2(\Delta,\varepsilon)\right)
\end{equation*}
is dense.

Let $\delta > 0$ and let $\Mcal$ be a closed subset of $\Rbb^D$, we show that there exists a closed $\Mcal'$ in $\cap_{\Delta > 0} (\Acal_1(\Delta,\delta) \cap \Acal_2(\Delta,\delta))$ (and thus in $\Acal$) such that $d_{H}(\Mcal,\Mcal') \leq 8 \delta$.

For $k,l \in \{1, \ldots, D\}$ with $k \leq l$, write $\pi^{(k:l)}$ the projection $\pi^{(k:l)} : (x_1, \ldots, x_D) \in \Rbb^D  \mapsto (x_k, \ldots, x_l) \in \Rbb^{l-k+1}$ and $\pi^{(k)} = \pi^{(k:k)}$.

Let $z = (z_1, z_2) \in \Mcal$ with $z_1 = \pi^{(1:d_1)}(z)$ and $z_2 = \pi^{(d_1+1:D)}(z)$, $\Mcal'$ is defined by cutting the space in half through $z$ orthogonally to the space of the first $d_1$ cooordinates and spreading the two halves apart, connecting them by a single segment to ensure it is in $\Acal_2(\Delta,\delta)$, then cut and connect again orthogonally to the $(d_1+1)$-th axis to be in $\Acal_1(\Delta,\delta)$.

Formally, define $\Mcal'$ as the union of:
\begin{itemize}
    \item $\{ y \,|\, y = (y_1,y_2) \in \Mcal, \pi^{(1)}(y_1) \leq \pi^{(1)}(z_1) \text{ and } \pi^{(1)}(y_2) \leq \pi^{(1)}(z_2) \}$,
    \item $\{ (y_1 + 4\delta (1, 0, \dots, 0), y_2) \,|\, y = (y_1,y_2) \in \Mcal, \pi^{(1)}(y_1) \geq \pi^{(1)}(z_1) \text{ and } \pi^{(1)}(y_2) \leq \pi^{(1)}(z_2) \}$,
    \item $\{ (y_1, y_2 + 4\delta (1, 0, \dots, 0)) \,|\, y = (y_1,y_2) \in \Mcal, \pi^{(1)}(y_1) \leq \pi^{(1)}(z_1) \text{ and } \pi^{(1)}(y_2) \geq \pi^{(1)}(z_2) \}$,
    \item $\{ (y_1 + 4\delta (1, 0, \dots, 0), y_2 + 4\delta (1, 0, \dots, 0)) \,|\, y = (y_1,y_2) \in \Mcal, \pi^{(1)}(y_1) \geq \pi^{(1)}(z_1) \text{ and } \pi^{(1)}(y_2) \geq \pi^{(1)}(z_2) \}$,
    \item the segments between $z$ and $(z_1 + 4\delta (1, 0, \dots, 0), z_2)$ and between $z$ and $(z_1, z_2 + 4\delta (1, 0, \dots, 0))$.
\end{itemize}
An illustration of this construction is given in Figure~\ref{fig_illustration1}.

\begin{figure}[h]
    \centering
    \includegraphics[scale=0.7]{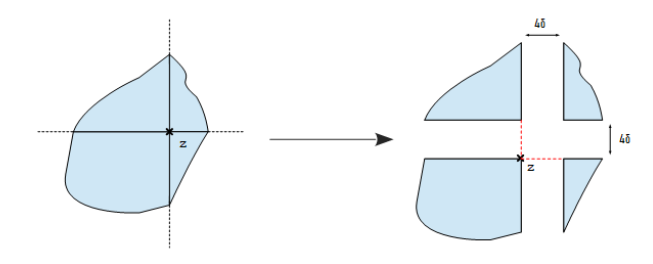}
    \caption{Transforming $\Mcal$ into a set $\Mcal' \in \Bcal_1 \cap \Bcal_2$}
    \label{fig_illustration1}
\end{figure}

By construction, the Hausdorff distance between this set $\Mcal'$ and $\Mcal$ is smaller than $8\delta$ (the points in the first four sets have moved at most $8\delta$ and the segments are at distance at most $8\delta$ of $z$). $\Mcal'$ is also closed, and taking $x = (z_1 + 2\delta (1, 0, \dots, 0), z_2)$ and $x_2 = (z_1, z_2 + 2\delta (1, 0, \dots, 0))$ in the definition of $\Acal_1(\Delta,\delta)$ and $\Acal_2(\Delta,\delta)$ is enough to check that $\Mcal' \in \Acal_1(\Delta,\delta) \cap \Acal_2(\Delta,\delta)$ for any $\Delta > 0$.

To show that the complement of $\Acal$ is dense, let $\Mcal$ be a closed subset of $\Rbb^D$ and $\eta > 0$, and let $\Mcal' = \{x+y \,|\, x \in \Mcal, y \in [-\eta,\eta]^D \}$. Then $d_H(\Mcal,\Mcal') \leq \eta \sqrt{D}$ by construction, and for any $\Delta \leq 2\eta$ and $\varepsilon > 0$, $\Mcal' \notin \Acal_1(\Delta,\varepsilon)$, and thus $\Mcal' \in \Acal^\complement$ where $\Acal^\complement$ is the complement of the set $\Acal$.

Note that if $\Mcal$ is the support of a random variable $X$, then $\Mcal'$ is the support of $X + Y$, where $Y$ is a uniform random variable on $[-\eta,\eta]^D$ that is independent of $X$. In that case, by Proposition~\ref{prop:(Adep)prop} (i), $X+Y$ is a small perturbation of $X$ that does not satisfy (Adep).

\subsection{Proof of Theorem~\ref{th_genericity}}
\label{app_proof_genericity}

Let $\Mcal$ be a compact set of $\Rbb^D$. Since $\varepsilon$ and $H$ are independent, and thus $f^\varepsilon$ and $H(\Mcal)$ are independent, writing $\mu_H$ the distribution of $H(\Mcal)$:
\begin{equation*}
\Pbb(F(\Mcal) \in \Bcal_1 \cap \Bcal_2)
    = \int \Pbb\left(f^\varepsilon(\frac{h}{\delta}) \in \Bcal_1 \cap \Bcal_2\right) d\mu_H(h)
    = 1,
\end{equation*}
provided that for any compact set $\Mcal' \in \Rbb^D$, $f^\varepsilon(\Mcal') \in \Bcal_1 \cap \Bcal_2$ a.s..

Thus, it suffices to show that for any compact set $\Mcal \in \Rbb^D$, almost surely, $f^\varepsilon(\Mcal)$ is in the set $\Bcal_1$ from Theorem~\ref{prop:idgeneral}. The proof for $\Bcal_2$ is identical.

We will show that $\text{Card}(\argmax_{z \in f^\varepsilon(\Mcal)} \pi^{(1)}(z)) = 1$, where $\pi^{(1)}(z)$ is the first coordinate of $z$.
First, since $\Mcal$ is compact and $f^\varepsilon$ is continuous, $f^\varepsilon(\Mcal)$ is compact, therefore the supremum of $\pi^{(1)}$ is reached at least at one point.

\begin{lemma}
\label{lemma_genericity}
The two following properties hold almost surely.
\begin{enumerate}
    \item Let $F'_I, F'_J \in \Pcal^\varepsilon$ be two different simplices, then at least one of the two following points holds:
    \begin{itemize}
        \item $\sup_{x \in f^\varepsilon(\Mcal) \cap \Int(F'_I)} \pi^{(1)}(x) \neq \sup_{x \in f^\varepsilon(\Mcal) \cap \Int(F'_J)} \pi^{(1)}(x)$
        \item $\pi^{(1)}$ does not reach its maximum on $f^\varepsilon(\Mcal) \cap \Int(F'_I)$ or does not reach its maximum on $f^\varepsilon(\Mcal) \cap \Int(F'_J)$.
    \end{itemize}

    \item Let $F'_I \in \Pcal^\varepsilon$, then the supremum of $\pi^{(1)}$ on $f^\varepsilon(\Mcal) \cap \Int(F'_I)$ is reached at at most one point of $\Int(F'_I)$.
\end{enumerate}
\end{lemma}

A consequence of this lemma is that almost surely, the maximizer of $\pi^{(1)}$ on $f^\varepsilon(\Mcal)$ is unique, as all maximizers of $\pi^{(1)}$ on $f^\varepsilon(\Mcal)$ belong to the relative interior of one simplex of $\Pcal^\varepsilon$, which shows that $f^\varepsilon(\Mcal)$ is almost surely in $\Bcal_1$.

\begin{proof}[Proof of Lemma~\ref{lemma_genericity}]
The following functions will be of use in the proof. For any finite $J \subset \Nbb$ such that $F'_J = \{x_i+\varepsilon_i\}_{i \in J} \in \Pcal^\varepsilon$, for any $j \in J$ and $\alpha \in (0,1]$, let
\begin{multline*}
u_{\alpha,J} : e \in \Rbb \longmapsto
    \sup\Bigg\{ \alpha (\pi^{(1)}(x_j) + e) + \sum_{k \in J \setminus \{j\}} \alpha_k \pi^{(1)}(x_k + \varepsilon_k) \text{, where } \\
        z = \alpha x_j + \sum_{k \in J \setminus \{j\}} \alpha_k x_k \in \Mcal, \  \alpha_k \in (0,1] \text{ and } \alpha + \sum_{k} \alpha_k = 1 \Bigg\}.
\end{multline*}
In other words, $u_{\alpha_J}$ is the supremum of $\pi^{(1)}$ on the slice of $f^\varepsilon(\Mcal) \cap \Int(F'_J)$ that gives weight $\alpha$ to the vertex $(x_j + \varepsilon_j)$. To simplify the notations, let $w_k : z \longmapsto \alpha_k$ be the ``weight" functions.
It is straightforward to check that
\begin{enumerate}
    \item the function $u_{\alpha,J}$ is linear with slope $\alpha$,
    \item $\sup_{x \in f^\varepsilon(\Mcal) \cap \Int(F'_J)} \pi^{(1)}(x) = \sup_{\alpha \in (0,1]} u_{\alpha,J}(\pi^{(1)}(\varepsilon_j))$,
    \item the function $h : \pi^{(1)}(\varepsilon_j) \longmapsto \sup_{x \in f^\varepsilon(\Mcal) \cap \Int(F'_J)} \pi^{(1)}(x)$ (all coordinates of all $\varepsilon_k$ other than $\pi^{(1)}(\varepsilon_j)$ being fixed) is convex,
    \item \label{pt_genericity_4} if the supremum of $\pi^{(1)}$ on the closure of $f^\varepsilon(\Mcal) \cap \Int(F'_J)$ is reached at some point $z \in F'_J$ when $\pi^{(1)}(\varepsilon_j) = e$, then $w_j(z)$ is a sub-gradient of $h$ at $e$,
    \item \label{pt_genericity_5} since the number of points where the sub-gradient of a convex function on $\Rbb$ is not unique is at most countable, almost surely (whether all coordinates of all $\varepsilon_k$ other than $\pi^{(1)}(\varepsilon_j)$ are fixed or not), $h$ has a unique sub-gradient at $\pi^{(1)}(\varepsilon_j)$.
\end{enumerate}

Let us now prove the first point of the lemma.
Let $F'_I = \{x_i+\varepsilon_i\}_{i \in I}$ and $F'_J = \{x_i+\varepsilon_i\}_{i \in J}$ be two different simplices of $\Pcal^\varepsilon$, and let $j \in J \setminus I$ (by exchanging the two simplices, we may assume without loss of generality that $J$ is not a subset of $I$).

Consider the following, conditionally to $(\varepsilon_n)_{n \neq j}$ and $\pi^{(2:D)}(\varepsilon_j)$. Assume that $h(\pi^{(1)}(\varepsilon_j)) = \sup_{x \in f^\varepsilon(\Mcal) \cap \Int(F'_I)} \pi^{(1)}(x)$
(otherwise we are in the first case of the first point of the lemma). We may assume without loss of generality (by point~\ref{pt_genericity_5} above) that the sub-gradient of $h$ at $\pi^{(1)}(\varepsilon_j)$ is unique. Two cases are possible:
\begin{itemize}
    \item the sub-gradient of $h$ at $\pi^{(1)}(\varepsilon_j)$ is 0. Then $\pi^{(1)}$ does not reach its maximum on $f^\varepsilon(\Mcal) \cap \Int(F'_J)$, since if $z$ is a maximizer of $\pi^{(1)}$, then $w_j(z) = 0$ by point~\ref{pt_genericity_4},

    \item the sub-gradient of $h$ at $\pi^{(1)}(\varepsilon_j)$ is positive, so there exists a single point $e$ such that $h(e) = \sup_{x \in f^\varepsilon(\Mcal) \cap \Int(F'_I)} \pi^{(1)}(x)$. Since $\pi^{(1)}(\varepsilon_j)$ is uniform on $[-r,r]$ by construction, we almost surely have $\pi^{(1)}(\varepsilon_j) \neq e$, and thus this second case almost surely never happens.
\end{itemize}

For the second point of the lemma, by points~\ref{pt_genericity_4} and~\ref{pt_genericity_5}, if the set of maximizers of $\pi^{(1)}$ on $f^\varepsilon(\Mcal) \cap \Int(F'_J)$ is a non-empty set $\Zcal$, then for any $j \in J$, almost surely, $w_j$ is constant on $\Zcal$. Since every point $z \in F'_J$ is characterized by the vector $(w_j(z))_{j \in J}$, this shows that $\Zcal$ contains a single point, which concludes the proof.
\end{proof}

\section{Proofs}
\label{sec:proofs}

\subsection{Proof of Proposition~\ref{prop:phihat}}
\label{subsec:prop:phihat}

For any $\nu>0$ and $\zeta \in L^2([-\nu,\nu]^D)$ (resp. $L^\infty([-\nu,\nu]^D)$), write $\|\zeta\|_{2,\nu}$ (resp. $\|\zeta\|_{\infty,\nu}$) its $L^2$ (resp. $L^\infty$) norm.

Let $\rho_0 \in [1,2)$. Let us start with some preliminary results.

From~\cite{InferSphere}, Section 7.1, for all $\nu>0$, there exists $b>0$, $\eta > 0$, $c_M > 0$ and $c_Z > 0$ such that, writing $\epsilon(u) = b/\log \log(1/u)$, the following properties hold for any $\rho' \in [1,\rho_0]$.
\begin{itemize}
    \item For all $\phi \in \Upsilon_{\rho',S}$ and for all $\zeta \in L^2([-\nu,\nu]^D)$ such that $\phi+\zeta \in \Upsilon_{\rho',S}$ and $\|\zeta\|_{2,\nu} \leq \eta$,
    \begin{equation}
    \label{eq_minoration_M}
        M(\phi+\zeta;\nu|\phi) \geq c_\nu^4 \|\zeta\|_{2,\nu}^{2 + 2\epsilon(\|\zeta\|_{2,\nu})}.
    \end{equation}

    \item For all $n \geq 1$, writing $Z_n(t,\phi) = \sqrt{n} ( \tilde\phi_n(t) - \phi(t) \Phi_{\varepsilon^{(1)}}(t_1)\Phi_{\varepsilon^{(2)}}(t_2) )$, one has for all $\phi \in \Upsilon_{\rho',S}$ and $h \in L^2([-\nu,\nu]^D)$ such that $\phi+\zeta \in \Upsilon_{\rho',S}$,
    \begin{equation}
    \label{eq_ecart_Mn_M}
    |M_n(\phi+\zeta)-M(\phi+\zeta;\nuest|\phi) - (M_n(\phi)-M(\phi;\nuest|\phi))|
        \leq c_M \frac{\|Z_n(\cdot,\phi)\|_{\infty,\nuest}}{\sqrt{n}} \|\zeta\|_{2,\nuest}^{1 - \epsilon(\|\zeta\|_{2,\nuest})}.
    \end{equation}
    
    \item For all $s \in [1,n]$,
    \begin{equation}
    \label{eq_deviations_Z}
    \Pbb(\|Z_n(\cdot, \Phi_X)\|_{\infty,\nuest} \geq c_Z \sqrt{s}) \leq e^{-s}.
    \end{equation}
\end{itemize}
Moreover, from Lemma H.3 of~\cite{LCGL2021supplementary}, there exists a constant $c_T > 0$ such that for all $\rho' \in [1,\rho_0]$, $m \geq \rho' D$ and $\phi \in \Upsilon_{\rho',S}$,
\begin{equation*}
    \|\phi - T_m\phi\|_{\infty,\nuest} \leq c_T (S\nuest)^m m^{-m/\rho' + D}.
\end{equation*}
Let $\rho' \in [1, \rho_0]$ and assume that $m \geq 2\rho' \frac{\log n}{\log \log n}$, then this equation becomes $\|\phi - T_{m}\phi\|_{\infty,\nuest} = O(n^{-2+o_n(1)})$, where $o_n(1)$ denotes a sequence tending to $0$ when $n$ tends to infinity. In particular, there exists $n_0$ such that for all $n \geq n_0$,
\begin{equation}
\label{eq_troncature}
\sup_{\rho' \in [1,\rho_0]}
    \sup_{\nu \in (0,\nuest]}
    \sup_{m \geq 2\rho' \frac{\log n}{\log \log n}}
    \sup_{\phi \in \Upsilon_{\rho',S}} \|\phi - T_{m} \phi\|_{2,\nu} \leq \frac{1}{n}
\end{equation}
and
\begin{equation}
\label{eq_troncature_dansM}
\sup_{\rho' \in [1,\rho_0]}
    \sup_{m \geq 2\rho' \frac{\log n}{\log \log n}}
    \sup_{\phi \in \Upsilon_{\rho',S}} |M_n(\phi) - M_n(T_{m} \phi)|
    \leq c \|\phi - T_{m} \phi\|_{\infty,\nuest}
    \leq \frac{1}{n}
\end{equation}
for some $c>0$ that depends only on $\nuest$, $\rho_0$ and $S$, using that $\sup_{\phi \in \Upsilon_{\rho_0,S}} \|\phi\|_{\infty,\nuest} < +\infty$.
Finally, following the proof of equation (25) of Section A.3 of~\cite{LCGL2021}, for any $\nu' \geq \nu$, $m \geq 1$ and $\phi \in \Cbb_m[X]$
\begin{equation}
\label{eq_extension_support}
    \|\phi\|_{2,\nu'} \leq m^{D/2} (4\frac{\nu'}{\nu})^{m+D/2} \|\phi\|_{2,\nu}.
\end{equation}
Let us now prove the proposition. Let $\rho \in [1,\rho_0]$ such that $\Phi_X \in \Upsilon_{\rho,S} \cap \Hcal$.
By definition, for any $m \geq 1$ and $\rho' \in [\rho,\rho_0]$, $\widehat{\Phi}_{n,m,\rho'}$ is such that $\widehat{\Phi}_{n,m,\rho'} \in \Upsilon_{\rho',S} \cap \Hcal$ and
\begin{align*}
M_n(T_m \widehat{\Phi}_{n,m,\rho'})
    &\leq \inf_{\phi \in \Upsilon_{\rho',S} \cap \Hcal} M_n(T_m \phi) + \frac{1}{n} \\
    &\leq \inf_{\phi \in \Upsilon_{\rho,S} \cap \Hcal} M_n(T_m \phi) + \frac{1}{n} \\
    &\leq M_n(T_{m} \Phi_X) + \frac{1}{n}
\end{align*}
and thus, by~\eqref{eq_troncature_dansM},
\begin{equation}
\label{eq_maj_Mn_phihat}
\sup_{\rho' \in [\rho, \rho_0]}
\sup_{m \geq 2\rho' \frac{\log n}{\log \log n}}
    M_n(\widehat{\Phi}_{n,m,\rho'})
        \leq M_n(\Phi_X) + \frac{3}{n}.
\end{equation}
Therefore, by~\eqref{eq_ecart_Mn_M}, for any $\nu \in (0,\nuest]$, writing $\zeta_{m,\rho'} = \widehat{\Phi}_{n,m,\rho'} - \Phi_X$,
\begin{align}
\nonumber
M(\widehat{\Phi}_{n,m,\rho'} ; \nu | \Phi_X)
    &\leq M(\widehat{\Phi}_{n,m,\rho'} ; \nuest | \Phi_X) \\
\label{eq_maj_Mn_phihat2}
    &\leq c_M \frac{\|Z_n(\cdot, \Phi_X)\|_{\infty, \nuest}}{\sqrt{n}} \|\zeta_{m,\rho'}\|_{2,\nuest}^{1 - \epsilon(\|\zeta_{m,\rho'}\|_{2,\nuest})}
        + \frac{3}{n}.
\end{align}
Let us show that we may apply~\eqref{eq_minoration_M}. Combining~\eqref{eq_maj_Mn_phihat} with Lemma A.1 of~\cite{LCGL2021} shows that for any $\delta > 0$, there exist $c_\eta > 0$ and $n_0$ (which do not depend on $\rho$) such that for all $n \geq n_0$,
with probability at least $1 - 4e^{-c_\eta n}$,
\begin{equation*}
\sup_{\rho' \in [\rho, \rho_0]}
\sup_{m \geq 2\rho \frac{\log n}{\log \log n}}
    M(\widehat{\Phi}_{n,m,\rho'} ; \nuest | \Phi_X) \leq \delta.
\end{equation*}
In addition, since $\Upsilon_{\rho_0,S} \cap \Hcal$ is compact in $L^2([-\nuest,\nuest]^D)$, $\phi \mapsto M(\phi;\nuest|\Phi_X)$ is continuous on $L^2([-\nuest,\nuest]^D)$, and $M(\phi ; \nuest | \Phi_X) = 0$ implies $\phi = \Phi_X$ for all $\phi \in \Hcal \cap \Upsilon_{\rho_0,S}$ by Theorem~\ref{thm:idG}, there exists $\delta > 0$ such that
\begin{equation*}
\inf_{\phi \in \Upsilon_{\rho_0,S} \cap \Hcal \text{ s.t. } \|\phi - \Phi_X\|_{2,\nuest} \geq \eta}
    M(\phi ; \nuest | \Phi_X) > \delta.
\end{equation*}
Therefore, there exist $c_\eta > 0$ and $n_0$ (which do not depend on $\rho$) such that for all $n \geq n_0$, with probability at least $1 - 4e^{-c_\eta n}$,
\begin{equation}
\label{eq_h_unif_small}
\sup_{\rho' \in [\rho, \rho_0]}
\sup_{m \geq 2\rho \frac{\log n}{\log \log n}}
    \|\zeta_{m,\rho'}\|_{2,\nuest} \leq \eta,
\end{equation}
which is what we need to apply~\eqref{eq_minoration_M}.

Fix now $\nu \in (0,\nuest]$, $c(\nu) > 0$ and $E > 0$ such that $\Qbb \in \Qcal^{(D)}(\nu,c(\nu),E)$. In particular, $c_\nu \geq c(\nu) > 0$. Then, by~\eqref{eq_minoration_M} and~\eqref{eq_maj_Mn_phihat2},
\begin{equation}
\label{eq_majorationh_step1}
\|\zeta_{m,\rho'}\|_{2,\nu}^{2 + 2\epsilon(\|\zeta_{m,\rho'}\|_{2,\nu})}
    \leq \frac{2}{c(\nu)^4} \max\left( c_M \frac{\|Z_n(\cdot, \Phi_X)\|_{\infty, \nuest}}{\sqrt{n}} \|\zeta_{m,\rho'}\|_{2,\nuest}^{1 - \epsilon(\|\zeta_{m,\rho'}\|_{2,\nuest})}, \frac{3}{n} \right).
\end{equation}
By~\eqref{eq_troncature} and~\eqref{eq_extension_support}, assuming $m \in [2\rho' \frac{\log n}{\log \log n}, C \frac{\log n}{\log \log n}]$ in the following series of inequalities (for some fixed $C > 2\rho'$),
\begin{align}
\nonumber
\|\zeta_{m,\rho'}\|_{2,\nuest}
    &\leq 2\max(\|T_{m} \zeta_{m,\rho'}\|_{2,\nuest}, \frac{3}{n}) & \text{by~\eqref{eq_troncature}} \\
\nonumber
    &\leq 2\max\left(\|T_{m} \zeta_{m,\rho'}\|_{2,\nu} m^{\frac{D}{2}} (4\frac{\nuest}{\nu})^{m+\frac{D}{2}}, \frac{3}{n}\right) & \text{by~\eqref{eq_extension_support}} \\
\nonumber
    &\leq 4\max\left(\|\zeta_{m,\rho'}\|_{2,\nu} m^{\frac{D}{2}} (4\frac{\nuest}{\nu})^{m+\frac{D}{2}}, \frac{3 m^{\frac{D}{2}} (4\frac{\nuest}{\nu})^{m+\frac{D}{2}}}{n}\right) & \text{by~\eqref{eq_troncature}} \\
\label{eq_extension_support_2}
    &\leq n^{\epsilon(1/n)} \max\left(\|\zeta_{m,\rho'}\|_{2,\nu}, \frac{1}{n} \right),
\end{align}
up to increasing the constant $b$ in the definition of $u \mapsto \epsilon(u)$, which can be done without loss of generality. Together with~\eqref{eq_majorationh_step1} and~\eqref{eq_deviations_Z}, one gets for all $s \in [1, c_\eta n]$ (assuming $c_\eta \leq 1$ without loss of generality), with probability at least $1 - 4e^{-c_\eta n} - e^{-s} \geq 1 - 5 e^{-s}$ (on the event where $\|Z(\cdot,\Phi_X)\|_{\infty,\nuest} \leq c_Z \sqrt{s}$ and~\eqref{eq_h_unif_small} holds) that for all $\rho' \in [\rho,\rho_0]$ and $m \in [2\rho' \frac{\log n}{\log \log n}, C \frac{\log n}{\log \log n}]$,
\begin{align*}
\|\zeta_{m,\rho'}\|_{2,\nu}^{2 + 2\epsilon(\|\zeta_{m,\rho'}\|_{2,\nu})}
    &\leq c \max\left(\sqrt{\frac{s}{n}} n^{\epsilon(1/n)} \left(\|\zeta_{m,\rho'}\|_{2,\nu} \vee \frac{1}{n}\right)^{1 - \epsilon(\|\zeta_{m,\rho'}\|_{2,\nuest})}, \frac{1}{n} \right)
\end{align*}
for some constant $c > 0$ that does not depend on $\rho$, $\rho'$ or $m$. Since $\epsilon$ is increasing, recalling that $\|\zeta_{m,\rho'}\|_{2,\nuest} \leq \eta$ on the event considered, by~\eqref{eq_extension_support_2},
\begin{align*}
\epsilon(\|\zeta_{m,\rho'}\|_{2,\nuest})
    &\leq 
        \begin{cases}
        \max( \epsilon( \|\zeta_{m,\rho'}\|_{2,\nu} n^{\epsilon(1/n)}), \epsilon(n^{-1 + \epsilon(1/n)}) )  \quad &\text{if $\|\zeta_{m,\rho'}\|_{2,\nu} \leq n^{-2\epsilon(1/n)}$,} \\
        \epsilon(\eta) &\text{always,}
        \end{cases} \\
    &\leq 
        \begin{cases}
        2 \epsilon(1/n) \quad &\text{if $\|\zeta_{m,\rho'}\|_{2,\nu} \leq n^{-2\epsilon(1/n)}$,} \\
        \epsilon(\eta) &\text{always,}
        \end{cases} \\
    &\leq [\epsilon(\eta) \textbf{\text{ or }} 2 \epsilon(1/n)]
\end{align*}
for $n$ large enough (depending on $b$), up to decreasing $\eta$, where for compactness of notations, $[A \textbf{\text{ or }} B]$ means $\min(A,B)$ if $\|\zeta_{m,\rho'}\|_{2,\nu} \leq n^{-2\epsilon(1/n)}$ and $A$ otherwise in the following.
Gathering the two previous equations shows that either
\begin{equation*}
\|\zeta_{m,\rho'}\|_{2,\nu}^{1 + 3 [ \epsilon(\eta) \textbf{\text{ or }} 2 \epsilon(1/n) ]}
    \leq c \sqrt{\frac{s}{n^{1 - 2\epsilon(1/n)}}}
\end{equation*}
or
\begin{equation*}
\|\zeta_{m,\rho'}\|_{2,\nu}^{2 + 2 [ \epsilon(\eta) \textbf{\text{ or }} 2 \epsilon(1/n) ]}
    \leq \frac{c}{n}.
\end{equation*}
Therefore, assuming $3\epsilon(\eta) \leq 1$ without loss of generality, $\|\zeta_{m,\rho'}\|_{2,\nu} \leq n^{-\epsilon(1/n)}$ as soon as $s \leq n^{1 - 10 \epsilon(1/n)} / c^2$ and thus, up to changing the constant $c$, for $n$ large enough and for all $s \in [1, n^{1-10\epsilon(1/n)}/c^2]$, with probability at least $1 - 4e^{-c_\eta n} - e^{-s}$, for all $\rho' \in [\rho,\rho_0]$ and $m \in [2\rho' \frac{\log n}{\log \log n}, C \frac{\log n}{\log \log n}]$,
\begin{align*}
\|\zeta_{m,\rho'}\|_{2,\nu}^2
    \leq c \left( \frac{s}{n^{1 - 2\epsilon(1/n)}} \right)^{1 - 6\epsilon(1/n)}.
\end{align*}
Finally, note that $4e^{-c_\eta n} e^{n^{1-10 \epsilon(1/n)}} \longrightarrow 0$, so that the probability that the last equation holds is larger than $1 - 2e^{-s}$ for $n$ large enough, which concludes the proof for the version with $\widehat\Phi_{n,m,\rho'}$. The version for $T_m \widehat\Phi_{n,m,\rho'}$ follows from this and~\eqref{eq_troncature}.

\subsection{Proof of Lemma~\ref{lem_minoration_gbar}}
\label{sec_proof_lem_minoration_gbar}

Let $y \in \Mcal_{G} \cap \Kcal$.
By property (III) of $\psi_A$,
\begin{align*}
\bar{g}(y)
    &= \frac{1}{h^D} \int \psi_A\left(\frac{\|y - u\|}{h}\right) dG(u) \\
    &\geq \frac{1}{h^D} \int_{\|u-y\|_{2}\leq c_{A}h} \psi_A(\frac{\|y - u\|}{h}) dG(u) \\
    &\geq \frac{1}{h^D} d_{A} G(B(y,c_{A}h)) \\
    &\geq \frac{1}{h^D} d_{A} a (c_{A}h)^d.
\end{align*}

\subsection{Proof of Lemma~\ref{lemma:sup}}
\label{sec_proof_lemmasup}

Recall the definition of $\bar{g}$: for all $y \in \Rbb^D$,
\begin{equation*}
\bar{g}(y) = \frac{1}{h^D} \int \psi_A(\frac{\|y - u\|}{h}) dG(u).
\end{equation*}
Let $C_1 > 0$ and $\epsilon >0$. By Property (V) of $\psi_A$ and~\eqref{lim}, there exists $T > 0$ (depending on $A$ and $C_1$) such that for any $t \geq T$, $\psi_{A}(t) \leq C_1 \exp(- \beta_A t^{A/(A+1)})$.
Take $y \in \Rbb^D$ such that $d(y,\Mcal_{G}) > (\frac{\varepsilon}{\beta_A})^{\frac{A+1}{A}} h \log(\frac{1}{h})^{\frac{A+1}{A}}$, then for all $u \in \Mcal_{G}$, $\frac{\|y-u\|}{h} \geq (\beta_A^{-1} \log(\frac{1}{h^\varepsilon}))^{\frac{A+1}{A}}$, therefore there exists $h_0 > 0$ depending only on $\varepsilon$, $D$, $A$ and $T$ (thus $C_1$ such that $h \leq h_0$ implies $\frac{\|y-u\|}{h} \geq T$ and thus
\begin{align*}
\psi_A(\frac{\|y-u\|}{h})
    &\leq C_1 \exp{\{ - \beta_A (\frac{\|y-u\|}{h})^{A/(A+1)} \}} \\
    &\leq C_1 \exp{\{ - \log(\frac{1}{h^{\varepsilon}}) \}} \\
    &= C_1 h^{\varepsilon},
\end{align*}
and finally $\bar{g}(y) \leq C_1 (\frac{1}{h})^{D-\varepsilon}$ since $G$ is a probability distribution. Lemma~\ref{lemma:sup} follows by taking $\varepsilon=D$.

\subsection{Proof of Lemma~\ref{lemma:gamma}}
\label{sec_proof_lemmagamma}

For $y \in \Rbb^D$,
\begin{equation*}
\widehat{g}_{n,\kappa}(y) - \bar{g}(y)
    = (\frac{1}{2 \pi})^D \int e^{-it^{\top}y} \Fcal[\psi_A](th) ( T_{m_{\kappa}} \widehat{\Phi}_{n,1/\kappa}(t) - \Phi_X(t)) dt.
\end{equation*}
Since $\Fcal[\psi_A](th)$ is 0 for $\|t\|_2 > 1/h$,
\begin{align}
\nonumber
\widehat{g}_{n,\kappa}(y) - \bar{g}(y)
    &= (\frac{1}{2 \pi})^D \int e^{-it^{\top}y} \Fcal[\psi_A](th) ( T_{m_{\kappa}} \widehat{\Phi}_{n,1/\kappa}(t) - \Phi_X(t)) 1|_{\|t\|_2 \leq 1/h} dt \\
\nonumber 
    &= \Fcal^{-1}[\Fcal[\psi_{h}] \lbrace ( T_{m_{\kappa}} \widehat{\Phi}_{n,1/\kappa} - \Phi_X) 1|_{\|t\|_2 \leq 1/h} \rbrace](y) \\
\label{convfourier}
    &= \Fcal^{-1} [ \Fcal[\psi_{A,h}] ] * \Fcal^{-1} [ ( T_{m_{\kappa}} \widehat{\Phi}_{n,1/\kappa} - \Phi_X) 1|_{\|t\|_2 \leq 1/h} ](y).
\end{align}
By Young's convolution inequality,
\begin{equation*}
\| \widehat{g}_{n,\kappa} - \bar{g} \|_{\infty} \leq \| \Fcal^{-1} [ \Fcal[\psi_{A,h}] ]\|_2 \|\Fcal^{-1} [ ( T_{m_{\kappa}} \widehat{\Phi}_{n,1/\kappa} - \Phi_X) 1|_{\|t\|_2 \leq 1/h} ] \|_2.
\end{equation*}
Finally, using Parseval's equality and the fact that $\Fcal^{-1}[\Fcal[\psi_{A,h}]] = \psi_{A,h}$,

\begin{equation*}
\| \widehat{g}_{n,\kappa} - \bar{g} \|_{\infty} \leq \|\psi_{A,h} \|_2
\| T_{m_{\kappa}} \widehat{\Phi}_{n,1/\kappa} - \Phi_X\|_{2, 1/h},
\end{equation*}
and use \eqref{psi3} to conclude the proof.

\subsection{Proof of Theorem~\ref{theorem:rateH}}
\label{proof:thmrateH}

Let $\kappa_{0}\in (1/2,1]$,  $\nu\in(0,\nu_\text{est}]$, $c(\nu)>0$, $E>0$, $S>0$ and $C>0$. Let $\kappa\in[\kappa_{0},1]$, $\Qbb\in \Qcal^{(D)} (\nu,c({\nu}),E)$ and $G \in St_{\Kcal}(a,d,r_0)\cap \Lcal(\kappa,S,{\cal H})$.

Using inequalities analogous to (28)-(29) p.17 of~\cite{LCGL2021}, we get that for all $\kappa' \in [\kappa_0, \kappa]$ and all integer $m$,
\begin{equation}
\label{paramphi}
\| T_{m} \widehat{\Phi}_{n,1/\kappa'} - \Phi_X \|^2_{2, 1/h}
    \leq 4 U(h) + 4 m^D ( 2 + 2 \frac{1}{h \nu})^{2m + D} \Bigg (2  V(\nu) + \| \widehat \Phi_{n,1/\kappa'} - \Phi_X \|^2_{2, \nu} \Bigg ),
\end{equation}
where
\begin{align*}
U(h) &= c h^{-D -2m - 2/\kappa'} S^{2m} m^{-2\kappa' m + 2D} \exp(2 \kappa' (S/h)^{1/\kappa'}) \\
        \text{and} \quad
    V(\nu) &= c (S \nu)^{2 m + 2/\kappa'} m^{-2\kappa' m + 2D}.
\end{align*}
Thus, applying Lemma~\ref{lemma:gamma} and using $h = c_h S m_{\kappa'}^{-\kappa'}$, there exists $C > 0$ such that on the event where \eqref{paramphi} holds:
\begin{multline}
\label{eq:gammasquare}
    \Gamma_{n,\kappa'}^2 \leq C (c_h S)^{-2D-2m_{\kappa'} - 2/\kappa'} m_{\kappa'}^{2D (\kappa'+1) + 2 } S^{2 m_{\kappa'}} \text{exp}(2 \kappa' c_h^{-1/\kappa'} m_{\kappa'}) \\
    + C m_{\kappa'}^{D(1+\kappa')} ( 2 + 2 \frac{m_{\kappa'}^{\kappa'}}{c_h S \nu})^{2m_{\kappa'} + D} \Bigg ((S \nu)^{2 m_{\kappa'} + 2/\kappa' } m_{\kappa'}^{-2 \kappa' m_{\kappa'} +2 D}
        + \| \widehat \Phi_{n,1/\kappa'} - \Phi_X \|^2_{2, \nu}
    \Bigg).
\end{multline}
The first term of the upper bound is upper bounded as follows.
\begin{align}
& (c_h S)^{-2D-2m_{\kappa'} - 2/\kappa'} m_{\kappa'}^{2D (\kappa'+1) + 2 } S^{2 m_{\kappa'}} \text{exp}(2 \kappa' c_h^{-1/\kappa'} m_{\kappa'}) \nonumber \\
&= S^{-2D-2/\kappa'} \exp{ \{ (-2D-2m_{\kappa'} - 2/\kappa') \log(c_h) + (2D(\kappa' +1) +2) \log(m_{\kappa'}) + 2 \kappa' c_{h}^{-1/\kappa'} m_{\kappa'} \}} \nonumber \\
&\leq C \exp{ \{ (-2\log(c_h) +1) m_{\kappa'} +  2(D(\kappa' + 1) + 1) \log(m_{\kappa'})\}} \label{eq1}\\
&\leq C \exp{ \{ (-2\log(c_h) + 3 + 2D(\kappa' + 1) ) m_{\kappa'} \}}\label{eq2} ,
\end{align}
for another constant $C>0$, where inequality~\eqref{eq1} holds because $2 \kappa c_{h}^{1/\kappa} > 1$ and inequality~\eqref{eq2} holds because $\log(m_{\kappa}) \leq m_{\kappa}$.
The second term of the upper bound is upper bounded by
\begin{align*}
& m_{\kappa'}^{D(1+\kappa')} ( 2 + 2 \frac{m_{\kappa'}^{\kappa'}}{c_h S \nu})^{2m_{\kappa'} + D}
        \Bigg ((S \nu)^{2 m_{\kappa'} + 2/\kappa' } m_{\kappa'}^{-2 \kappa' m_{\kappa'} +2 D}
        + \| \widehat \Phi_{n,1/\kappa'} - \Phi_X \|^2_{2, \nu}
        \Bigg) \\
    &\leq C \ m_{\kappa'}^{D(1+2\kappa')} (2 \kappa' m_{\kappa'})^{2 \kappa' m_{\kappa'}} (2\kappa')^{-2\kappa' m_{\kappa'}}  (c_h S \nu)^{ -2 m_{\kappa'} - D } \\
        &\hspace{13em} \times \Bigg((S \nu)^{2 m_{\kappa'} + 2/\kappa' } m_{\kappa'}^{-2 \kappa' m_{\kappa'} + 2 D}
        + \| \widehat \Phi_{n,1/\kappa'} - \Phi_X \|^2_{2, \nu}
        \Bigg) \\
    &\leq C \Bigg( \exp{\Bigg \{ (-2\log(c_h) + (3D+2 \kappa')) m_{\kappa'} \Bigg\}} \\
        &\hspace{5em} + (2\kappa' m_{\kappa'})^{2 \kappa' m_{\kappa'}} \exp{\Bigg \{ (-2 \log(c_h) + D(1 + 2\kappa')) m_{\kappa'} \Bigg \}}
        \| \widehat \Phi_{n,1/\kappa'} - \Phi_X \|^2_{2, \nu}
        \Bigg)
\end{align*}
for another constant $C>0$. Putting all together, we get that for yet another constant $C>0$, 
\begin{align*}
    \Gamma_{n,\kappa'}^2 \leq C \max \Bigg( & \exp{ \Bigg \{ (-2\log(c_h) + 3 + 2D(\kappa' + 1) ) m_{\kappa'} \Bigg \}}, \exp{\Bigg \{ (-2\log(c_h) + (3D+2 \kappa')) m_{\kappa'} \Bigg \}}, \\
    & (2\kappa' m_{\kappa'})^{2 \kappa' m_{\kappa'}} \exp{ \Bigg \{ (-2 \log(c_h) + D(1 + 2\kappa')) m_{\kappa'} \Bigg\}}
    \| \widehat \Phi_{n,1/\kappa'} - \Phi_X \|^2_{2, \nu}
    \Bigg).
\end{align*}
Choosing $c_h \geq \exp{\{ 2D + 2\}}$ and $m_{\kappa'}=\frac{1}{4\kappa'} \frac{ \log n}{\log \log n}$
for some $\gamma \in (0,1)$, it follows that
\begin{equation}
\label{eq_majoration_Gamma_generale}
\Gamma_{n,\kappa'}^{2} \leq C e^{-m_{\kappa'}} \Bigg[ 1 \vee n^{1/2}
    \| \widehat \Phi_{n,1/\kappa'} - \Phi_X \|^2_{2, \nu} 
    \Bigg].
\end{equation}
By Proposition~\ref{prop:phihat}, taking $s = \log n$ and $\delta, \delta''$ such that $(1-\delta)(1-\delta'') > 1/2$, we obtain that with probability at least $1-2/n$, for all $\kappa' \leq \kappa$, $\Gamma_{n,\kappa'}^2 \leq C e^{-m_{\kappa'}} \longrightarrow 0$. Note that we could also take $s = n^{1/2 - \delta'''}$ for any $\delta''' > 0$ and still have $\Gamma_{n,\kappa'}^2 \leq C e^{-m_{\kappa'}}$ with probability at least $1 - 2e^{-s}$, up to changing the constant $C$, by picking $\delta$ and $\delta''$ small enough in Proposition~\ref{prop:phihat}.

Now, by Lemma~\ref{lem_minoration_gbar}, for any $h \leq (r_0 / c_A) \wedge 1$,
\begin{align*}
\inf_{y \in \Mcal_{G} \cap \Kcal} \widehat{g}_{n,\kappa'}(y)
    &\geq \inf_{y \in \Mcal_{G} \cap \Kcal} \bar{g}(y) - \Gamma_{n,\kappa'} \\
    &\geq c_{A}^d d_{A}  a \left(\frac{1}{h}\right)^{D-d} - \Gamma_{n,\kappa'} \\
    &\geq \frac{c_{A}^d d_{A}  a}{2} \left(\frac{1}{h}\right)^{D-d} 
\end{align*}
as soon as $\Gamma_{n,\kappa'}\leq \frac{c_{A}^d d_{A} a}{2}$, and this lower bound is strictly larger than $\lambda_{n,\kappa}$ for any $d$.
This implies that on the event where $\Gamma_{n,\kappa'}\leq \frac{c_{A}^d d_{A}  a}{2}$, $\Mcal_{G} \cap \Kcal \subset \widehat{\Mcal}_{\kappa'} \cap \Kcal$. Next,
\begin{equation*}
\sup_{y \in \Kcal, d(y,\Mcal_{G}) \geq h\left[\frac{D}{\beta_{A}} \log\left(\frac{1}{h}\right)\right]^{\frac{A+1}{A}}} \widehat{g}_{n,\kappa'}(y)
    \leq
    \sup_{y \in \Kcal, d(y,\Mcal_{G}) \geq h\left[\frac{D}{\beta_{A}} \log\left(\frac{1}{h}\right)\right]^{\frac{A+1}{A}}}
    \bar{g}(y) + \Gamma_{n,\kappa'}.
\end{equation*}
Choosing $C_{1} = \frac{c_{A}^d d_{A} a}{16}$ and applying Lemma~\ref{lemma:sup} we get that, on the event where $\Gamma_{n,\kappa'} \leq \frac{c_{A}^d d_{A} a}{16}$,
\begin{equation*}
\sup_{y \in \Kcal, d(y,\Mcal_{G}) \geq h\left[\frac{D}{\beta_{A}} \log\left(\frac{1}{h}\right)\right]^{\frac{A+1}{A}}} \widehat{g}(y)
     \leq 2 C_1 
\end{equation*}
for $n$ large enough, and this upper bound is strictly less than $\lambda_{n,\kappa'}$ for any $d$.
This implies that
\begin{equation*}
\left\{ y : y \in \Kcal, d(y,\Mcal_{G}) >  h\left[\frac{D}{\beta_{A}} \log\left(\frac{1}{h}\right)\right]^{\frac{A+1}{A}} \right\} \cap \widehat{\Mcal}_{\kappa'} = \varnothing.
\end{equation*}
We may now take $h$ as in the statement of the Theorem.
As a result, we have proved that:
for all $\kappa_{0}\in (1/2,1]$, $S>0$, $a>0$ $d\leq D$, $\nu\in (0,\nu_{\text{est}}]$, $c(\nu)>0$ and $E>0$, there exist $c'>0$ and $n_{0}$ such that for all $n\geq n_{0}$, for all $\kappa\in[\kappa_{0},1]$, $G\in St_{\Kcal}(a,d,r_0)\cap \Lcal(\kappa,S,{\cal H})$ and $\Qbb\in \Qcal^{(D)} (\nu,c({\nu}),E)$, with ${(G * \Qbb)^{\otimes n}}$-probability at least $1-\frac{2}{n}$,
\begin{equation}
    \label{probUnif}
 \sup_{\kappa' \in  [\kappa_{0},\kappa]} 
\frac{\log(n)^{\kappa'}}{\log( \log(n))^{\kappa' + \frac{A+1}{A}}}
H_{\Kcal}(\Mcal_{G},\widehat{\Mcal}_{\kappa'}) \leq c' .
\end{equation}
Using the fact that $H_{\Kcal}(\Mcal_{G},\widehat{\Mcal}_{\kappa})$ is uniformly upper bounded, the theorem follows.

\subsection{Proof of Theorem~\ref{theo:adapt}}
\label{proof:theo:adapt}

Fix $\kappa_{0}\in (1/2,1]$, $S>0$, $a>0$ $d\leq D$, $\nu\in (0,\nu_{\text{est}}]$, $c(\nu)>0$ $E>0$. Using the end of the proof of Theorem~\ref{theorem:rateH}, there exist $n_{0}$ and $c'$ such that for all $\kappa\in[\kappa_{0},1]$, all $G\in St_{\Kcal}(a,d,r_0)\cap \Lcal(\kappa,S,{\cal H})$ and all $\Qbb\in \Qcal^{(D)} (\nu,c({\nu}),E)$, with ${(G * \Qbb)^{\otimes n}}$-probability at least $1-\frac{2}{n}$, \eqref{probUnif} holds. Let us now choose $c_{\sigma}=c'$ and consider the event where~\eqref{probUnif} holds.
By the triangular inequality, for any $\kappa \in [\kappa_0,1]$,
\begin{align*}
H_{\Kcal}(\Mcal_{G},\widehat{\Mcal}_{\widehat{\kappa}_{n}})
    &\leq H_{\Kcal}(\Mcal_{G},\widehat{\Mcal}_{\kappa})+H_{\Kcal}(\widehat{\Mcal}_{\kappa},\widehat{\Mcal}_{\widehat{\kappa}_{n}}) \\
    &\leq \sigma_{n}(\kappa)+H_{\Kcal}(\widehat{\Mcal}_{\kappa},\widehat{\Mcal}_{\widehat{\kappa}_{n}}).
\end{align*}
Now, using the definition of $B_{n}(\cdot)$, if $\kappa \leq \widehat{\kappa}_{n}$, then
\begin{equation*}
H_{\Kcal}(\widehat{\Mcal}_{\kappa},\widehat{\Mcal}_{\widehat{\kappa}_{n}})\leq B_{n}(\widehat{\kappa}_{n})+\sigma_{n}(\kappa)
\end{equation*}
while if $\kappa \geq \widehat{\kappa}_n$, then
\begin{equation*}
H_{\Kcal}(\widehat{\Mcal}_{\kappa},\widehat{\Mcal}_{\widehat{\kappa}_{n}})\leq B_{n}(\kappa)+\sigma_{n}(\widehat{\kappa}_{n})
\end{equation*}
so that in all cases,
\begin{align*}
H_{\Kcal}(\widehat{\Mcal}_{\kappa},\widehat{\Mcal}_{\widehat{\kappa}_{n}})
    &\leq B_{n}(\widehat{\kappa}_{n})+\sigma_{n}(\kappa)
        + B_{n}(\kappa)+\sigma_{n}(\widehat{\kappa}_{n}) \\
    &\leq 2 B_{n}(\kappa)+2 \sigma_{n}(\kappa)
\end{align*}
using the definition of $\widehat{\kappa}_{n}$, and therefore
\begin{equation*}
H_{\Kcal}(\Mcal_{G},\widehat{\Mcal}_{\widehat{\kappa}_{n}})
\leq 2 B_{n}(\kappa)+3 \sigma_{n}(\kappa).
\end{equation*}
By the triangular inequality and the definition of $B_{n}(\cdot)$,
\begin{align*}
B_{n}(\kappa)
    &\leq 0 \vee \sup_{\kappa' \in [\kappa_{0},\kappa]}\left\{H_{\Kcal}(\widehat{\Mcal}_{\kappa},{\Mcal}_{G})+
H_{\Kcal}({\Mcal}_{G},\widehat{\Mcal}_{\kappa'})-\sigma_{n}(\kappa')\right\} \\
    &\leq H_{\Kcal}(\widehat{\Mcal}_{\kappa},{\Mcal}_{G})
        + 0 \vee \sup_{\kappa' \in [\kappa_{0},\kappa]}\left\{H_{\Kcal}({\Mcal}_{G},\widehat{\Mcal}_{\kappa'})-\sigma_{n}(\kappa')\right\} \\
    &\leq \sigma_{n}(\kappa).
\end{align*}
Thus, for all $\kappa\in[\kappa_{0},1]$, all $G\in St_{\Kcal}(a,d,r_0)\cap \Lcal(\kappa,S,{\cal H})$ and all $\Qbb \in \Qcal^{(D)} (\nu,c({\nu}),E)$, with ${(G * \Qbb)^{\otimes n}}$-probability  at least $1-\frac{2}{n}$,
\begin{equation*}
H_{\Kcal}(\Mcal_{G},\widehat{\Mcal}_{\widehat{\kappa}_{n}}) \leq 5 \sigma_{n}(\kappa),
\end{equation*}
and using the fact that $H_{\Kcal}(\Mcal_{G},\widehat{\Mcal}_{\widehat{\kappa}}) \leq \sup_{x,x' \in \Kcal} d(x,x')$ on the event of probability at most $2/n$ where this does not hold, Theorem~\ref{theo:adapt} follows.

\subsection{Proof of Lemma~\ref{lemma:fourierbound}}
\label{proof:fourierbound}

\paragraph*{Case $\kappa \neq 1$}

This case is based on~\cite{MR1574180}. In the following, we will note constants that can change with upper case $A$, $B$ and $C$.
In~\cite{MR1574180}, Theorem 2, the author defines for any positive constants $\mu>0$, $q>1$ and $a>0$ a function $\zeta_{q,\mu,a}$, such that for $x \in \Rbb$,
\begin{equation*}
    \zeta_{q,\mu,a}(x) = -i \int_{\mathcal{C}} z^{\mu} \exp(z^q - qax^2z) dz,
\end{equation*}
where $\mathcal{C}$ is a curve in the complex plane so that the maximum of $ |z^{\mu} \exp(z^q - qax^2z)|$ for $z \in \mathcal{C}$ is attained on the positive real line. The author shows that $\zeta_{q,\mu,a}$ and $\zeta_{q,\mu,a}^2$ are integrable functions.

The author uses the saddle-point integration method to show that there exist $A>0$ and $B>0$ which depend on $q$, $\mu$ and $a$ such that

\begin{equation}
\label{eq:bounderfourierzeta}
    |\Fcal[\zeta_{q,\mu,a}](t)| \leq A \exp(-B x^{\frac{2q}{q+1}}).
\end{equation}
Finally, for $\kappa \in (1/2,1)$, fix $\mu>0$, $a>0$, and define 

\begin{equation*}
f_{\kappa} = c_{f_{\kappa}} \text{Re}[\zeta_{\frac{1}{2\kappa-1},\mu,a}]^2 * u_1,
\end{equation*}
where $u_1 : x \in \Rbb \mapsto \exp(-\frac{1}{1-4x^2})1|_{(-1/2,1/2)}(x)$ and $c_{f_{\kappa}}$ is a constant that ensures that $f_{\kappa}$ is a density.

Let us first prove that there exist $A>0$ and $B>0$ positive constants such that $|\Fcal[\text{Re}[\zeta_{q,\mu,a}]^2](t)| \leq A \exp(-B|t|^{1/\kappa})$.
\begin{align}
\nonumber
|\Fcal[\text{Re}[\zeta_{\frac{1}{2\kappa-1},\mu,a}]^2](t)|
    &= |\Fcal[\text{Re}[\zeta_{\frac{1}{2\kappa-1},\mu,a}]] \ast \Fcal[\text{Re}[\zeta_{\frac{1}{2\kappa-1},\mu,a}]] (t)| \\
\nonumber
    &\leq A \int_{\Rbb} \exp(-B|x-y|^{1/\kappa} - B |y|^{1/\kappa}) dy \\
\nonumber
    &=  \int_{|y-x| \geq |x|/2} \exp(-B|x-y|^{1/\kappa} - B |y|^{1/\kappa}) dy \\
 \nonumber
    &\quad + \int_{|y-x| < |x|/2} \exp(-B|x-y|^{1/\kappa} - B |y|^{1/\kappa}) dy \\
\label{eq:Recarre}
    &\leq A \exp(-B|x|^{1/\kappa}).
\end{align}
Finally, for all $t \in \Rbb$, using that $|\Fcal[u_1](t)| \leq 1$,

\begin{align*}
|\Fcal[f_{\kappa}](t)|
    &= |\Fcal[\text{Re}[\zeta_{\frac{1}{2\kappa-1},\mu,a}]^2](t)| \ |\Fcal[u_1](t)| \\
    &\leq c_{f_{\kappa}} A \exp(-B |x|^{\frac{1}{\kappa}}).
\end{align*}
For $x \in \Rbb$, $\Fcal[f_{\kappa}(x)]' = \Fcal[x \mapsto x f_{\kappa}(x)]$ and 
\begin{equation*}
xf_{\kappa}(x) = c_{f_{\kappa}} v * \text{Re}[\zeta_{\frac{1}{2\kappa-1},\mu,a}]^2(x) + c_{f_{\kappa}} u_1 * \tilde{\zeta}(x),
\end{equation*}
where $v : x \in \Rbb \mapsto x u_1(x)$ and $\tilde{\zeta} : x \in \Rbb \mapsto x \text{Re}[\zeta_{\frac{1}{2\kappa-1},\mu,a}]^2(x)$.

\medskip
Following the same proof as Theorem 2 of~\cite{MR1574180}, there exists $A > 0$ and $B > 0$ such that for all $t \in \Rbb$, $|\Fcal[x \mapsto x \text{Re}[\zeta_{\frac{1}{2\kappa-1},\mu,a}](t)]| \leq A \exp(- B |t|^{1/\kappa})$, so that, following the proof of~\eqref{eq:Recarre}, $|\Fcal[\tilde{\zeta}]|(t) \leq A \exp(-B |t|^{1/\kappa})$.
Hence, there exists $A > 0$ and $B > 0$ such that $|\Fcal[f_{\kappa}]'(t)| \leq A \exp(-B |t|^{1/\kappa})$.

Finally, note that $f_{\kappa}$ is continuous as a convolution of an integrable function with a smooth function, and that for all $x \in \Rbb$, $f_{\kappa}(x) > 0$ since $\text{Re}[\zeta_{\frac{1}{2\kappa-1},\mu,a}]$ and $u$ are not the null function almost everywhere.

\paragraph*{Case $\kappa = 1$}

Let $\delta \in (0,1)$ and define $f_1 : x \in \Rbb \mapsto c_{f_{1}} (u_{\frac{1}{1-\delta}} \ast u_{\frac{1}{1-\delta}} )(x)$, where $c_{f_{1}}$ is a constant that ensures that $f_1$ is a probability density.

There exist $A> 0$ and $B>0$ such that $\Fcal[f_1](x) \leq A \exp(-B|x|^{\delta})$, see Lemma in~\cite{BumpFunction}. Moreover, $\Fcal[f_1]'(x) = 2 c_{f_{1}} \Fcal[ u_{\frac{1}{1-\delta}}](x) \Fcal[ u_{\frac{1}{1-\delta}}]'(x) \leq A \|x \mapsto x u_{\frac{1}{1-\delta}}(x) \|_{1} \exp(-B |x|^{\delta})$.

Finally, note that $f_{1}$ is continuous and does not vanish on its support.

\subsection{Proof of Lemma~\ref{lemma:A1/kappa}}
\label{proof:lemma:A1/kappa}

First, by Lemma~\ref{lemma:fourierbound}, for any $\kappa \in (1/2,1]$, $U(\kappa)$ satisfies $A(1/\kappa)$.

Let $i \in \{0,1\}$. For any $\lambda = (\lambda_1, \dots, \lambda_D) \in \Rbb^D$,
\begin{align}
\nonumber
\Ebb[\exp(\lambda^{\top} X_i(\kappa))]
    &= \Ebb\left[\exp\left( \left(\lambda_1+ \frac{1}{2} \lambda_2\right) U(\kappa) + (-1)^i \gamma \lambda_{2} \frac{1}{2} \cos\left(\frac{U(\kappa)}{\gamma}\right)\right)\right] \\
\label{eq:fkappalaplace}
    &\leq e^{ \gamma \frac{1}{2} | \lambda_{2}| }  \Ebb\left[\exp\left( \left(\lambda_1 + \frac{1}{2} \lambda_2\right) U(\kappa)\right)\right].
\end{align}
Since $U(\kappa)$ satisfies $A(1/\kappa)$, there exist positive constants $A>0$ and $B>0$ such that for all $\lambda = (\lambda_1, \ldots, \lambda_D) \in \Rbb^D$, $\Ebb[\exp(\lambda^{\top} X_i(\kappa))]\leq A\exp(B |\lambda|^{\frac{1}{\kappa}})$. Applying this in \eqref{eq:fkappalaplace},
\begin{align*}
\Ebb[\exp(\lambda^{\top} X_i(\kappa))]
    &\leq A\exp\left(\gamma \frac{1}{2} | \lambda_{2}|+B \left|\left(\lambda_1 + \frac{1}{2} \lambda_2\right)\right|^{\frac{1}{\kappa}}\right)\\
    &\leq A' \exp(B' |\lambda|^{\frac{1}{\kappa}})
\end{align*}
for some other constants $A'$ and $B'$ since $1\leq \kappa$, so that
$X_{i}(\kappa)$ satisfies $A(1/\kappa)$.

\subsection{Proof of Lemma~\ref{lemma:standard}}
\label{proof:lemma:standard}

The proof is done in five steps.
\begin{enumerate}
    \item \label{pt_adstandard_1} We show that $\gamma g_{\gamma}$ is $1$-lipschitz.
    \item \label{pt_adstandard_2} For $i \in \{0,1\}$ and $\kappa \in (\frac{1}{2},1]$, we compute the density $p_i$ of $T_i(\kappa)$ with respect to the $1$-dimensional Hausdorff measure $\mu_{H}$ and we show that for any compact set $\Kcal$, there exists $b(\kappa,\Kcal)>0$ such that, for all $u \in M_i(\gamma) \cap \Kcal$, $|p_i(u)| \geq b(\kappa, \Kcal)$.
    \item \label{pt_adstandard_3} We show that for $i \in \{0,1\}$, $\mu_{H}( \cdot \cap M_i(\gamma))$ is in $St_{\Kcal}(2,d,r_0)$.
    \item \label{pt_adstandard_4} We deduce that for $i \in \{0,1\}$ and $d \geq 1$, $T_i$ is in $St_{\Kcal}(2 b(\kappa, \Kcal), d,r_0)$.
    \item \label{pt_adstandard_5} Finally, we show that for $i \in \{0,1\}$, $d \geq 1$ and $a$ small enough, $G_i(\kappa) \in St_{\Kcal}(a, d, r_0)$.
\end{enumerate}

\paragraph*{Proof of~\ref{pt_adstandard_1}}
For all $x\in\Rbb$, $| \gamma \tilde{g}_{\gamma}'(x) | = | \sin(\frac{x}{\gamma}) | \leq 1$, which implies that $\gamma g_{\gamma}$ is $1$-Lipschitz.

\paragraph*{Proof of~\ref{pt_adstandard_2}} 
Let us first compute the density $p_i$ of $T_i(\kappa)$ with respect to $\mu_{H}$.
For $i \in \{0,1\}$, denote $\zeta_i : x \in \Rbb \mapsto (x, (-1)^i \gamma g_{\gamma}(x))$.
Let $\mathcal{B}$ be an open subset of $\Rbb^D$. For any $\kappa \in (\frac{1}{2},1]$,
\begin{equation*}
    T_i(\kappa)(\mathcal{B}) = \Pbb[ \zeta_i(U(\kappa)) \in \mathcal{B}] = \Pbb[U(\kappa) \in \zeta_i^{-1}(\mathcal{B})] = \int_{\zeta_i^{-1}(\mathcal{B})} f_{\kappa}(u) du.
\end{equation*}
Let $\Jac \zeta_i : u \in \Rbb \mapsto \sqrt{1 + \gamma^2 \tilde{g}_{\gamma}(u)^2}$ be the Jacobian of $\zeta_i$. By the Area Formula (see equation (2.47) in~\cite{MR1857292}),
\begin{equation*}
    s_i(\kappa)(\mathcal{B}) = \int_{\zeta_i^{-1}(\mathcal{B})} \frac{ f_{\kappa}(u)}{\Jac \zeta_i(u)} \Jac \zeta_i(u) du = \int_{\mathcal{B} \cap M_i(\gamma)} \frac{f_{\kappa}(\pi^{(1)}(u))}{\Jac \zeta_i(\pi^{(1)}(u))} d\mu_{H}(u).
\end{equation*}
We then have that for all $x \in \Rbb^D$, 

\begin{equation*}
    p_i(x) = \frac{f_{\kappa}(\pi^{(1)}(x))}{\Jac \zeta_i(\pi^{(1)}(x))} 1|_{M_i(\gamma)}(x).
\end{equation*}
Since  $f_{\kappa}$ is continuous and does not vanish on its support, for any compact set $\Kcal$, $M_i(\gamma) \cap \Kcal$ is a compact subset of the support of $f_{\kappa}$. Thus, since  $f_{\kappa}$ is continuous and does not vanish on its support, for any compact set $\Kcal$, there exists $c(\kappa, \Kcal)>0$ such that for all $u \in M_i(\gamma) \cap \Kcal$, $f_{\kappa}(u) \geq c(\kappa, \Kcal)$.
Moreover, for $i \in \{0, 1 \}$, $\Jac \zeta_i(u) \leq \sqrt{2}$. Therefore, for all $x \in M_i(\gamma) \cap \Kcal$, $|p_i(x)| \geq \frac{c(\kappa, \Kcal)}{\sqrt{2}}$.

\paragraph*{Proof of~\ref{pt_adstandard_3}} 
Recall that 
the $1$-dimensional Hausdorff measure $\mu_{H}$ is defined as the limit $\lim_{\eta \rightarrow 0} \mu_{H}^{\eta}$, where for any set $Z$
\begin{equation*}
\mu_{H}^{\eta}(Z) = \inf \left\{\sum_{i \in \Nbb} \Diam(A_i) : X \subset \bigcup_i A_i \text{ and } \forall i, \Diam(A_i) \leq \eta \right\}.
\end{equation*}
For any $z \in M_i(\gamma)$, there exists $x_0 \in \Rbb$ such that  $z=(x_0, (-1)^i \gamma g_{\gamma}(x_0))$ and, for any $r>0$,
\begin{equation*}
B(z,r) \cap M_i(\gamma) \supset \{(x,(-1)^i \gamma g_{\gamma}(x)), x \in B(x_0, r)\} 
\end{equation*}
since $|x-x_0| \leq r$ implies $\|\gamma g_{\gamma}(x)- \gamma g_{\gamma}(x_0)\|_{\infty} \leq r$.

Let $(A_i)_{i\in\Nbb}$ be a covering of $\{(x,(-1)^i \gamma g_{\gamma}(x)), x \in B(x_0, r)\}$, and $B_i = \pi^{(1)}(A_i)$, then $B_i$ is a covering of $B(x_0, r)$. For all $\eta > 0$,
\begin{equation*}
\mu_{H}^{\eta}(\{(x,(-1)^i \gamma g_{\gamma}(x)), x \in B(x_0, r)\})
    \geq \mu_{H}^{\eta}(B(x_0, r)),
\end{equation*}
thus $\mu_{H}(B(z,r) \cap M_i(\gamma)) \geq \mu_{H}(B(x_0, r)) = 2r$. If $r_0 \leq 1$, then for any $r\leq r_0$, 
\begin{equation*}
    \mu_{H}(B(z,r) \cap M_i(\gamma)) \geq 2r^d,
\end{equation*}
which proves~\ref{pt_adstandard_3}.

\paragraph*{Proof of~\ref{pt_adstandard_4}}
Let $x_i \in M_i(\gamma) \cap \Kcal$ and $r_0 < 1$. Then for all $r \leq r_0$, 
\begin{equation*}
T_i(B(x_i,r) \cap M_i(\gamma))
    = \int_{B(x_i,r) \cap M_i(\gamma)} p_i(u) d\mu_{H}(u)
    \geq b(\kappa, \Kcal) \mu_{H}(B(x_i,r) \cap M_i(\gamma))
    \geq 2 b(\kappa, \Kcal) r^d.
\end{equation*}

\paragraph*{Proof of~\ref{pt_adstandard_5}}
For $i \in \{0,1\}$, let $x_i \in J M_i(\gamma)\cap \Kcal$, $r_0<1$, and take $\tilde{\Kcal}$ such that $J^{-1} \Kcal \subset \tilde{\Kcal}$.
For all $r \leq r_0$,
\begin{align*}
G_i(\kappa)(B(x_i,r)) = \Pbb[J S_i(\kappa) \in B(x_i,r)]
    &\geq \Pbb\left[S_i(\kappa) \in B\left(J^{-1} x_i, \frac{r}{\|J\|_\text{op}}\right)\right] \\
    &\geq \frac{2 b(\kappa, J^{-1} \Kcal)}{\|J\|_\text{op}} r \\
    &\geq \frac{2 b(\kappa, \tilde{\Kcal})}{\|J\|_\text{op}} r,
\end{align*}
so that for some $a_0$ and all $a\leq a_0$, $G_i(\kappa)(B(x_i,r)) \geq a r^d$.

\subsection{Proof of Lemma~\ref{lemma:GiA}}
\label{proof:lemma:GiA}

Let us write $m_{i,\gamma}(x) = ( x + (-1)^i \frac{\gamma}{2} \cos(\frac{x}{\gamma}), 0 , \ldots, 0)$, so that $X_i(\kappa) = (U(\kappa), m_{i,\gamma}(U(\kappa))$.
For $i \in \{0,1\}$, let $w_{i,\kappa, \gamma}$ be the density of the first coordinate of $m_{i,\gamma}(U(\kappa))$, then
\begin{equation*}
    M_{\kappa} = \sup_{x \in \Rbb, \gamma \in [0,1], i\in \{0,1\}} \{ w_{i,\kappa, \gamma}(x) \vee f_{\kappa}(x) \} 
\end{equation*}
is an upper bound of the density of $X_{i}(\kappa)^{(1)}$ and of the first coordinate of $X_{i}(\kappa)^{(2)}$ with respect to the Lebesgue measure. Let us show that $M_{\kappa}$ is finite.
First, note that $m_{i,\gamma}$ is one-to-one from $\Rbb$ to $\Rbb \times \{0\}^{D-2}$ and $m^{-1}_{i,\gamma}$ is Lipschitz with Lipschitz constant upper bounded by $1/2$. One can easily check that for all $x \in \Rbb$,
\begin{equation*}
    w_{i,\kappa,\gamma}(x) = f_{\kappa}((m_{i,\gamma})_{1}(x,0, \ldots, 0)) \frac{1}{(m_{i,\gamma})_{1}'(m^{-1}_{i,\gamma}(x,0,\ldots,0))},
\end{equation*}
where $(m_{i,\gamma})_1(x)$ is the first coordinate of $m_{i,\gamma}(x)$. Since $(m_{i,\gamma})_{1}'$ is lower bounded by $1/2$, $M_{\kappa} \leq \sup_{x \in \Rbb}  f_{\kappa}(x)$, which is finite.

For any $\Delta>0$, define the sets:
\begin{eqnarray*}
    A_{\Delta}^{(1)} = [-\Delta, \Delta] \ &\text{and}& \ \ B_{\Delta}^{(2)} = \bar{B}\left(0,(\frac{1}{2}+2) \Delta\right) \cap (\Rbb \times \{0\}^{D-2}),\\
    A_{\Delta}^{(2)} = [-\Delta, \Delta] \times \{0\}^{D-2} \ &\text{and}& \ \ B_{\Delta}^{(1)} = \bar{B}(0,\Delta) \cap \Rbb.
\end{eqnarray*}
Define $c_{\Delta, \kappa} = \Pbb[U(\kappa) \in A^{(1)}_{\Delta}] \wedge \inf_{\gamma \in [0,1], i \in \{0,1\}} \Pbb[m_{i,\gamma}(U(\kappa)) \in A^{(2)}_{\Delta}]$, and let us prove that $c_{\Delta,\kappa} > 0$.

First, $\Pbb[U(\kappa) \in A_{\Delta}^{(1)}] > 0$ since the density of $U(\kappa)$ is positive everywhere on its support. Then, for $i\in\{0,1\}$,
\begin{align*}
\Pbb[m_{i,\gamma}(U(\kappa)) \in A_{\Delta}^{(2)}]
    &= \Pbb\left(U(\kappa) + (-1)^i \frac{1}{2} \gamma \cos\left(\frac{U(\kappa)}{\gamma}\right) \in [-\Delta, \Delta]\right) \\
    & \geq \Pbb\left(U(\kappa) \in [-\Delta/2,\Delta/2], (-1)^i \frac{1}{2} \gamma \cos\left(\frac{U(\kappa)}{\gamma}\right) \in [-\Delta/2, \Delta/2]\right) \\
    & \geq \Pbb\left(U(\kappa) \in [-\Delta/2,\Delta/2], \cos\left(\frac{U(\kappa)}{\gamma}\right) \in [-\Delta/\gamma,  \Delta/\gamma]\right) \\
    &\geq \Pbb\left(U(\kappa) \in \left[-\frac{\Delta}{2},\frac{\Delta}{2}\right] \cap \left[\arccos(\Delta), \pi - \arccos(\Delta)\right]\right),
\end{align*}
which is positive.


For any $\Delta > 0$ define $B^{(1)}_{\Delta, i, \gamma} = m_{i,\gamma}^{-1}(A^{(2)}_{\Delta})$. Then 
\begin{align*}
\Diam(B^{(1)}_{\Delta, i, \gamma})
    &= \sup_{x, y \in B^{(1)}_{\Delta, i, \gamma}} |x -  y| \\
    &= \sup_{x,y \in A^{(2)}_{\Delta}} |m^{-1}_{i,\gamma}(x) - m^{-1}_{i,\gamma}(y)| \\
    & \leq \frac{1}{2} \sup_{x,y \in A^{(2)}_{\Delta}} \|x-y\|
        \leq \Delta.
\end{align*}
Thus, $B^{(1)}_{\Delta, i, \gamma} \subset B^{(1)}_{\Delta}$, and 
\begin{equation*}
\Pbb[(X_i(\kappa))^{(1)} \in B^{(1)}_{\Delta,i} | (X_i(\kappa))^{(2)} \in A^{(2)}_{\Delta}]  \geq \Pbb[(X_i(\kappa))^{(1)} \in B^{(1)}_{\Delta,i,\gamma} | (X_i(\kappa))^{(2)} \in A^{(2)}_{\Delta}] = 1.
\end{equation*}
Similarly, define $B^{(2)}_{\Delta, i, \gamma} = m_{i,\gamma}(A^{(1)}_{\Delta}) =  \{ ( x + (-1)^i \gamma \frac{1}{2} \cos(\frac{x}{\gamma}),0, \ldots, 0) \ , \ x \in [-\Delta, \Delta] \}$, then
\begin{align*}
\Diam(B^{(2)}_{\Delta, i, \gamma})
    &= \sup_{x, y \in A^{(1)}_{\Delta}} \left|x + (-1)^i \gamma \frac{1}{2} \cos\left(\frac{x}{\gamma}\right) -  y - (-1)^i \gamma \frac{1}{2} \cos\left(\frac{y}{\gamma}\right)\right| \\
    &\leq 2 \Delta + \gamma \frac{1}{2} \left| \cos\left(\frac{x}{\gamma}\right) - cos\left(\frac{y}{\gamma}\right)\right| \\
    &\leq \left(\frac{1}{2} + 2\right) \Delta.
\end{align*}
Thus, $B^{(2)}_{\Delta, i, \gamma} \subset B^{(2)}_{\Delta,i}$, and 
\begin{equation*}
\Pbb[(X_i(\kappa))^{(2)} \in B^{(2)}_{\Delta,i} | (X_i(\kappa))^{(1)} \in A^{(1)}_{\Delta}]
    \geq \Pbb[(X_i(\kappa))^{(2)} \in B^{(2)}_{\Delta,i,\gamma} | (X_i(\kappa))^{(1)} \in A^{(1)}_{\Delta}]
    = 1.
\end{equation*}
Thus, the distribution of $X_i(\kappa)$ satisfies (H1) and (H2). Lemma \ref{lemma:GiA} follows from Lemma \ref{lemma:A1/kappa} and Theorem \ref{thm:id}.

\subsection{Proof of Theorem~\ref{thm:lower}}
\label{proof:lower}

In the following, we will write $A$, $B$, $C$ (with upper case letters) positive constants that can change from line to line. As in~\cite{LCGL2021} and~\cite{GPVW12}, we use the upper bound:
\begin{equation*}
\|(G_0(\kappa) \ast Q)^{\otimes n} - (G_1(\kappa) \ast Q)^{\otimes n}\|_{TV} \leq 1-\left(1-\|(G_0(\kappa) \ast Q)-(G_1(\kappa) \ast Q)\|_{TV} \right)^{n},
\end{equation*}
where $\|\cdot - \cdot\|_{TV}$ denotes the total variation distance.
Using Le Cam's two-points method, the minimax rate will be lower bounded by $H(J M_0(\gamma),J M_1(\gamma))$, that is $\gamma$, (see Lemma~\ref{lemma:lowerboundH}) provided that there exists a constant $C>0$ such that $\|(G_0(\kappa) \ast Q)^{\otimes n} - (G_1(\kappa) \ast Q)^{\otimes n}\|_{TV} \leq C < 1$, so that we only need to find $C>0$ such that
\begin{equation*}
\int_{\Rbb^D} \left| d(G_{0}(\kappa)\ast Q)(x)-d(G_1(\kappa) \ast Q)(x)\right| \leq \frac{C}{n}.
\end{equation*}
Since $Q$ has a density $q$ over $\Rbb^D$, $G_{0}(\kappa) \ast Q$ and $G_1(\kappa) \ast Q$ also have a density over $\Rbb^D$.
We first prove that for $i \in \{0,1\}$,
\begin{equation}
\label{eq:1}
\int_{\Rbb^D} \prod_{j=1}^{D}x_{j}^{2} \left| \frac{d(G_{i}(\kappa) \ast Q)}{dx}(x)\right|^{2} dx  < +\infty.
\end{equation}
Indeed,
\begin{equation*}
\int_{\Rbb^D} \prod_{j=1}^{D}x_{j}^{2} \left| \frac{d(G_{i}(\kappa) \ast Q)}{dx}(x)\right|^{2} dx
    \leq \left \| \frac{d(G_{i}(\kappa) \ast Q)}{dx} \right \|_{\infty} \int_{\Rbb^D} \prod_{j=1}^{D}x_{j}^{2} d(G_{i}(\kappa) \ast Q)(x).
\end{equation*}
First, $\left\| \frac{d(G_{i}(\kappa) \ast Q)}{dx} \right\|_{\infty} \leq \| q \|^{D}_{\infty} < \infty$. Moreover, for $k \in \{1, \ldots, D\}$, writing $X_i(\kappa)^{[k]}$ and $\varepsilon^{[k]}$ for the $k$-th coordinate of $X_i(\kappa)$ and $\varepsilon$,
\begin{align}
\int_{\Rbb^D} \prod_{j=1}^{D}x_{j}^{2} \left| d(G_{i}(\kappa) \ast Q)(x)\right|
    &= \Ebb[\prod_{k=1}^{D} (X_i(\kappa)^{[k]}+\varepsilon^{[k])})^2]
\nonumber \\
\label{eq:Efinite}
    &= \Ebb [ (X_i(\kappa)^{[1]}+\varepsilon^{[1]})^2 (X_i(\kappa)^{[2]}+\varepsilon^{[2]})^2] \prod_{k=3}^{D} \Ebb[(\varepsilon^{[k]})^2].
\end{align}
We have that $(X_i(\kappa)^{[2]}+\varepsilon^{[2]})^2 \leq a^2 (X_i(\kappa)^{[1]})^2 + 2 \gamma X_i(\kappa)^{[1]} + 2 X_i(\kappa)^{[1]} \varepsilon^{[2]} + (1 + \gamma) (\varepsilon^{[2]})^2 + \gamma^2$, using~\eqref{eq:Efinite} and the fact that $\varepsilon^{[2]}$ is independent of all other variables and that, for $k \in \{ 1, 2 \}$, $X_i^{[1]}$ is independent of $\varepsilon^{[k]}$, we finally get that $\int \prod_{j=1}^{D}x_{j}^{2} \left| d(G_{i}(\kappa) \ast Q)(x)\right|$ is upper bounded by product and sum of expectation of $((\varepsilon^{[j]})^2)_{j \in \{1, \ldots, D\}}$, $(X_i(\kappa)^{[1]})^2$, $(X_i(\kappa)^{[1]})^3$ and $(X_i(\kappa)^{[1]})^4$ which are all finite thanks to Lemma~\ref{lemma:fourierbound}.

By the Cauchy-Schwarz inequality,
\begin{multline}
\label{eq:2}
\int_{\Rbb^D} \left| d(G_{0}(\kappa) \ast Q)(x)-d(G_1(\kappa) \ast Q)(x)\right| \\
    \leq \pi^{D/2}\left(\int \prod_{j=1}^{D}(1+x_{j}^{2}) \left| \frac{d((G_{0}(\kappa)-G_{1}(\kappa)) \ast Q)}{dx}(x)\right|^{2} dx \right)^{1/2}.
\end{multline}
By Parseval's identity, for all $\eta \in \{0,1\}^D$,
\begin{align*}
\int_{\Rbb^{D}}\prod_{j=1}^D x_j^{2\eta_j} \left| \frac{d((G_{0}(\kappa)-G_{1}(\kappa)) \ast Q)}{dx}(x)\right|^{2} dx
	&= \int_{\Rbb^{D}} \left| \left( \prod_{j=1}^D \partial^{\eta_j}_{t_j} \right) (\Fcal[G_0(\kappa)]-\Fcal[G_1(\kappa)])(t) \Fcal[Q](t)\right|^2 dt \\
	&= \int_{[-c,c]^{D}} \left| \left( \prod_{j=1}^D \partial^{\eta_j}_{t_j} \right) (\Fcal[G_0(\kappa)]-\Fcal[G_1(\kappa)])(t) \Fcal[Q](t)\right|^2 dt,
\end{align*}
since $\Fcal[Q]$ and for $\eta \in \{0,1\}^D$, $\partial^{\eta}\Fcal[Q]$ are supported on $[-c,c]^D$. Moreover, they are bounded functions, so that there exists a constant $C$ (depending only on $D$) such that
\begin{align*}
\int_{\Rbb^D} \left| d(G_{0}(\kappa) \ast Q)(x)-d(G_1(\kappa) \ast Q)(x)\right|
    &\leq C \sum_{\eta\in\{0,1\}^{D}}\int_{[-c,c]^{D}}\left|\left(\prod_{j=1}^{D} \partial^{\eta_j}_{t_j} \right) (\Fcal[G_0(\kappa)]-\Fcal[G_1(\kappa)])(t)\right|^2 dt \\
    &= \sum_{\eta\in\{0,1\}^{D}}\int_{[-c,c]^{D}}\left|\left(\prod_{j=1}^{D} \partial^{\eta_j}_{t_j} \right) (t \mapsto \Fcal[S_0]-\Fcal[S_1])(A_a^\top t)\right|^2 dt.
\end{align*}
Using the change of variable $u = A^\top t$, and noticing that $\{ A^\top t \ ; \ t \in [-c,c]^D \} \subset [-2c,2c]^D$, there exists a constant $C>0$ depending on $D$ and $a$ such that
\begin{equation*}
    \int_{\Rbb^D} \left| d(G_{0}(\kappa) \ast Q)(x)-d(G_1(\kappa) \ast Q)(x)\right|  \leq C \sum_{\eta\in\{0,1\}^{D}}\int_{[-2c, 2c]^D}\left|\left(\prod_{j=1}^{D} \partial^{\eta_j}_{t_j} \right) (\Fcal[S_0]-\Fcal[S_1])(u)\right|^2 du.
\end{equation*}
For all $t=(t_{1},\ldots,t_{D})\in\Rbb^D$, for $i \in \{0,1\}$,  $\Fcal[T_i](t)=\Fcal[\tilde{T}_i](t_{1},t_{2})$, where $\tilde{T}_{i}$ is the distribution of the $2$ first coordinates of $S_i(\kappa)$ under $T_{i}$. There exists a constant $C>0$ such that
\begin{multline}
\label{eq:lowerbound1}
\int_{\Rbb^D} \left| d(G_{0}(\kappa) \ast Q)(x)-d(G_1(\kappa) \ast Q)(x)\right| \\
    \leq C \sum_{\eta\in\{0,1\}^{2}}\int_{[-2c, 2c]^{2}}\left|\left(\prod_{j=1}^{2} \partial^{\eta_j}_{t_j} \right) (\Fcal[\tilde{T}_0]-\Fcal[\tilde{T}_1])(t)\right|^2 dt.
\end{multline}
Following the same approach as~\cite{GPVW12}, we get that for all $t = (t_1, t_{2}) \in \Rbb^{2}$,  
\begin{align*}
(\Fcal[\tilde{T}_0]-\Fcal[\tilde{T}_1])(t)
    &= \int_{\Rbb} \{ e^{i t_{1} u + i \gamma t_{2} \tilde{g}_{\gamma}(u)} - e^{i t_{1} u - i \gamma t_{2} \tilde{g}_{\gamma}(u)} \} f_{\kappa}(u) du \\
    &= 2i \int_{\Rbb} e^{i t_{1} u} \sin(t_{2} \gamma \tilde{g}_{\gamma}(u)) f_{\kappa}(u) du \\
    &= 2i \int_{\Rbb} e^{i t_{1} u} \sum_{k=0}^{\infty} \frac{ (-1)^k t_{2}^{2k+1} \gamma^{2k+1}}{(2k+1)!} \tilde{g}_{\gamma}^{2k+1}(u) f_{\kappa}(u) du.
\end{align*}
Since $\sum_{k=0}^{\infty} \int_{\Rbb} \frac{ |t_{2}|^{2k+1} \gamma^{2k+1}}{(2k+1)!} |\tilde{g}_{\gamma}^{2k+1}(u)| f_{\kappa}(u) du $ is finite, we can switch integral and sum thanks to Fubini Theorem, so that
\begin{align*}
(\Fcal[\tilde{T}_0]-\Fcal[\tilde{T}_1])(t)
    &= 2i \sum_{k=0}^{\infty} \frac{(-1)^{k}t_{2}^{2k+1}\gamma^{2k+1}}{(2k+1)!} \int_{\Rbb} e^{i t_1 u } \tilde{g}_{\gamma}^{2k+1}(u) f_{\kappa}(u) du \\
    &= 2i \sum_{k=0}^{\infty} \frac{(-1)^{k}t_{2}^{2k+1}\gamma^{2k+1}}{(2k+1)!} m_{k}(t_{1}),
\end{align*}
with for all $u \in \Rbb$,
\begin{equation}
\label{eq:mconvol}
    m_k(u) = \Fcal[\tilde{g}^{2k+1} f_{\kappa}](u) = (\underbrace{\Fcal{[\tilde{g}]} * \Fcal{[\tilde{g}]} * \ldots * \Fcal{[\tilde{g}]}}_{ 2k+1 \ \text{times}} * \Fcal{[f_{\kappa}]})(u).
\end{equation}
Since
\begin{equation*}
    \Fcal{[x \mapsto \cos(\frac{x}{\gamma})]} = \frac{1}{2} \delta_{-\frac{1}{\gamma}} + \frac{1}{2} \delta_{\frac{1}{\gamma}},
\end{equation*}
for all $u \in \Rbb$,
\begin{align*}
(\underbrace{\Fcal{[\tilde{g}]} * \Fcal{[\tilde{g}]} * \ldots * \Fcal{[\tilde{g}]}}_{ 2k+1 \ \text{times}})(u)
    &= \underbrace{\Fcal{[\cos(\frac{\cdot}{ \gamma})]} * \ldots * \Fcal{[\cos(\frac{\cdot}{ \gamma})]}}_{ 2k+1 \ \text{times}} \\
    &= \left(\frac{1}{2}\right)^{2k+1} \sum_{j=1}^{2k+1} \binom{2k+1}{j} \delta_{a_{j}},
\end{align*}
where $a_j = (2j-2k-1)/\gamma$. By~\eqref{eq:mconvol}, 


\begin{equation*}
m_{k}(u)=\left(\frac{1}{2}\right)^{2k+1}\sum_{j=0}^{2k+1} \binom{2k+1}{j} \Fcal[f_{\kappa}](u-a_j) .
\end{equation*}
Therefore,

\begin{equation*}
\sup_{|t|\leq c} | m_{k}(t) | \leq \sup_{|t|\leq c, 0\leq j \leq 2k+1} \left|  \Fcal[f_{\kappa}]\left(t-\frac{2j-2k-1}{ \gamma } \right)\right|
\end{equation*}
and
\begin{equation*}
\sup_{|t|\leq c} | m_{k}'(t) | \leq \sup_{|t|\leq c, 0\leq j \leq 2k+1} \left|  \Fcal[f_{\kappa}]'\left(t-\frac{2j-2k-1}{ \gamma } \right)\right|.
\end{equation*}

\paragraph*{Assume first that $\kappa \in (1/2,1)$}

For $\gamma$ that satisfies $\gamma \leq \frac{1}{2c}$, by Lemma~\ref{lemma:fourierbound}, there exist two constants $A$, $B$ independent of $\gamma$ and $k$ such that
\begin{equation*}
     \sup_{|t|\leq c, 0\leq j \leq 2k+1} \left|  \Fcal[f_{\kappa}]\left(t-\frac{2j-2k-1}{ \gamma } \right)\right| \leq A \exp(- B \gamma^{-\frac{1}{\kappa}})
\end{equation*}
and
\begin{equation*}
    \sup_{|t|\leq c, 0\leq j \leq 2k+1} \left|  \Fcal[f_{\kappa}]'\left(t-\frac{2j-2k-1}{ \gamma } \right)\right| \leq A \exp(- B \gamma^{-\frac{1}{\kappa}}).
\end{equation*}
Thus,
\begin{equation}
\label{eq:supm}
    \sup_{|t|\leq c} | m_{k}(t) | \leq A  \exp(- B \gamma^{-\frac{1}{\kappa}}) ,
\end{equation}
and
\begin{equation}
\label{eq:supm'}
    \sup_{|t|\leq c} | m_{k}'(t) | \leq A \exp(- B \gamma^{-\frac{1}{\kappa}}).
\end{equation}
For all $\eta \in \{0,1\}^{2}$, and $t \in [-c,c]^2$,
\begin{multline*}
\left ( \prod_{j=1}^{2} \partial_{t_{j}}^{\eta_{j}} \right ) (\Fcal[\tilde{T}_0]-\Fcal[\tilde{T}_1])(t)
    = \prod_{j=1}^{2} \partial_{t_{j}}^{\eta_{j}} \left [ 2i \sum_{k=0}^{\infty} \frac{(-1)^{k}t_{2}^{2k+1}\gamma^{2k+1}}{(2k+1)!} m_{k}(t_{1}) \right] \\
    = 2i \eta_{2} \sum_{k=0}^{\infty} \frac{(-1)^{k}t_{2}^{2k}\gamma^{2k+1}}{(2k)!} \partial_{t_{1}}^{\eta_{1}} m_{k}(t_{1}) + 2i (1 - \eta_{2}) \sum_{k=0}^{\infty} \frac{(-1)^{k}t_{2}^{2k+1}\gamma^{2k+1}}{(2k+1)!} \partial_{t_{1}}^{\eta_{1}} m_{k}(t_{1}),
\end{multline*}
so that
\begin{align*}
&\left | \left ( \prod_{j=1}^{2} \partial_{t_{j}}^{\eta_{j}} \right ) (\Fcal[\tilde{T}_0]-\Fcal[\tilde{T}_1])(t) \right | \\
    &\leq 2 \sum_{k=0}^{\infty} \frac{|t_{2}|^{2k}\gamma^{2k+1}}{(2k)!} | \partial_{t_{1}}^{\eta_{1}} m_{k}(t_{1})| + 2 \sum_{k=0}^{\infty} \frac{|t_{2}|^{2k+1}\gamma^{2k+1}}{(2k+1)!} | \partial_{t_{1}}^{\eta_{1}} m_{k}(t_{1})|.
\end{align*}
By~\eqref{eq:supm} and~\eqref{eq:supm'}, there exists a constant $C>0$ which depends only on $D$ and $A$ such that
\begin{equation*}
   \left ( \prod_{j=1}^{2} \partial_{t_{j}}^{\eta_{j}} \right ) (\Fcal[\tilde{S}_0]-\Fcal[\tilde{S}_1])(t) \leq C \exp(-B \gamma^{-\frac{1}{\kappa}}) \sup_{|t_{2}|\leq c} \Bigg( \gamma  \cosh(| t_{2}| \gamma ) + \sinh(|t_{2}| \gamma )  \Bigg).
\end{equation*}
For $\gamma$ small enough, there exists a constant $C_1>0$ which depends only on $D$, $A>0$ and $C>0$ such that
\begin{equation*}
\left | \left ( \prod_{j=1}^{2} \partial_{t_{j}}^{\eta_{j}} \right ) (\Fcal[\tilde{T}_0]-\Fcal[\tilde{T}_1])(t) \right | \leq  C \exp(-B \gamma^{-\frac{1}{\kappa}}).
\end{equation*}
Finally, using \eqref{eq:lowerbound1}, there exist constants $C>0$ and $B>0$ which depend only on $D$ such that
\begin{equation*}
\int_{\Rbb^D} \left| d(G_{0}(\kappa) \ast Q)(x)-d(G_1(\kappa) \ast Q)(x)\right|  \leq C \exp(-B \gamma^{-\frac{1}{\kappa}}).
\end{equation*}
Taking $\gamma = c_{\gamma}(\log n)^{-\kappa}$ with $c_{\gamma} \leq B_1^{\kappa}$ shows that there exists $C>0$ such that
\begin{equation*}
\int_{\Rbb^D} \left| d(G_{0}(\kappa) \ast Q)(x)-d(G_1(\kappa) \ast Q)(x)\right|  \leq \frac{C}{n}.
\end{equation*}

\paragraph*{Let us now consider the case $\kappa = 1$}

For $\gamma$ that satisfies $\gamma \leq \frac{1}{2c}$, by Lemma~\ref{lemma:fourierbound}, for all $\delta \in (0,1)$, there exist two constants $A>0$, $B>0$ independent of $\gamma$ and $k$ such that
\begin{equation*}
     \sup_{|t|\leq c, 0\leq j \leq 2k+1} \left|  \Fcal[f_{1}]\left(t-\frac{2j-2k-1}{ \gamma } \right)\right| \leq A \exp(- B \gamma^{-\delta} ),
\end{equation*}
and
\begin{equation*}
    \sup_{|t|\leq c, 0\leq j \leq 2k+1} \left|  \Fcal[f_{1}]'\left(t-\frac{2j-2k-1}{ \gamma } \right)\right| \leq A \exp(- B \gamma^{-\delta} ).
\end{equation*}
Thus, there exist constants $A>0$ and $B>0$ independent of $\gamma$ and $k$ such that
\begin{equation}
\label{eq:supmkappa1}
    \sup_{|t|\leq c} | m_{k}(t) | \leq A \exp(- B \gamma^{-\delta}),
\end{equation}
and
\begin{equation}
\label{eq:supm'kappa1}
    \sup_{|t|\leq c} | m_{k}'(t) | \leq A \exp(- B \gamma^{-\delta}).
\end{equation}
Doing the same computation as in the case $\kappa \in (1/2,1)$ shows that for all $\eta \in \{0,1\}^{2}$ and $t \in [-c,c]^2$,
\begin{align*}
&\left | \left ( \prod_{j=1}^{2} \partial_{t_{j}}^{\eta_{j}} \right ) (\Fcal[\tilde{T}_0]-\Fcal[\tilde{T}_1])(t) \right | \\
    &\leq 2 \sum_{k=0}^{\infty} \frac{|t_{2}|^{2k}\gamma^{2k+1}}{(2k)!} | \partial_{t_{1}}^{\eta_{1}} m_{k}(t_{1})| + 2 \sum_{k=0}^{\infty} \frac{|t_{2}|^{2k+1}\gamma^{2k+1}}{(2k+1)!} | \partial_{t_{1}}^{\eta_{1}} m_{k}(t_{1})|.
\end{align*}
By~\eqref{eq:supmkappa1} and~\eqref{eq:supm'kappa1}, there exist constants $C>0$ and $B>0$ which depend only on $D$ such that
\begin{equation*}
   \left ( \prod_{j=1}^{2} \partial_{t_{j}}^{\eta_{j}} \right ) (\Fcal[\tilde{T}_0]-\Fcal[\tilde{T}_1])(t) \leq C \exp(-B \gamma^{-\delta}) \sup_{|t_{2}|\leq c} \Bigg( \gamma \cosh(| t_{2}| \gamma ) + \sinh(|t_{2}| \gamma)  \Bigg).
\end{equation*}
For $\gamma$ small enough, there exists a constant $C>0$ which depends only on $D$ such that
\begin{equation*}
\left | \left ( \prod_{j=1}^{2} \partial_{t_{j}}^{\eta_{j}} \right ) (\Fcal[\tilde{T}_0]-\Fcal[\tilde{T}_1])(t) \right | \leq  C \exp(-B \gamma^{-\delta}).
\end{equation*}
Finally, using~\eqref{eq:lowerbound1}, there exists a constant $C>0$ which depends only on $D$ such that
\begin{equation*}
\int_{\Rbb^D} \left| d(G_{0}(\kappa) \ast Q)(x)-d(G_1(\kappa) \ast Q)(x)\right|  \leq C \exp(-B \gamma^{-\delta}).
\end{equation*}
Taking $\gamma = c_{\gamma}(\log n)^{-\frac{1}{\delta}}$ with $c_{\gamma} \leq B^{\frac{1}{\delta+1}}$ shows that there exists $C>0$ such that
\begin{equation*}
\int_{\Rbb^D} \left| d(G_{0}(\kappa) \ast Q)(x)-d(G_1(\kappa) \ast Q)(x)\right|  \leq \frac{C}{n}.
\end{equation*}

\subsection{Proof of Theorem~\ref{theorem:upperboundwasserstein}}
\label{proof:theorem:upperboundwasserstein}

We shall need two technical lemmas. The following one is easily proved following the arguments at the end of the proof of Theorem~\ref{theorem:rateH}.

\begin{lemma}
\label{Lemma:encadrementMcal}
Let $G$ be a probability measure with compact support $\Mcal_{G}$. Assume $G \in St_{\Mcal_{G}}(a,d,r_0)$ for some constants $a>0$, $d>0$ and $r_0 >0$.
Recall that $\Gamma_n := \Gamma_{n,1} = \|\hat{g}_n - \bar{g}\|_\infty$. Then

\begin{itemize}
    \item[(1)] For any $C_1 > 0$ and $c > 0$, there exists $h_0 > 0$ such that if $h_n \leq h_0$, on the event where
    \begin{equation*}
    C_1 + \Gamma_n < \lambda_n < \ a c_{A}^{d}d_{A} (\frac{1}{h_n})^{D-d} - \Gamma_n,
    \end{equation*}
    it holds
    \begin{equation*}
        \Mcal_{G} \subset \widehat{\Mcal} \subset (\Mcal_{G})_{c}.
    \end{equation*}

    \item[(2)] For $m_n$, $h_n$, $\lambda_n$ chosen as in Theorem~\ref{theorem:rateH},
        for all $C_1 \in (0,a c_{A}^{d}d_{A})$ and $\delta' > 0$, there exist $C>0$ and $n_0 \geq 0$ such that for all $n \geq n_0$, with probability at least $1 - 2 \exp(-n^{1/2 - \delta'})$,
        \begin{equation*}
            \Gamma_n^2 \leq C e^{-m_n}
                \qquad \text{and} \qquad
                C_1 + \Gamma_n < \lambda_n < \ a c_{A}^{d}d_{A} (\frac{1}{h_n})^{D-d} - \Gamma_n.
        \end{equation*}
        and in particular, for all $c > 0$ and $\delta' > 0$, there exists $n_0' \geq 0$ such that for all $n \geq n_0'$, with probability at least $1 - 2 \exp(-n^{1/2 - \delta'})$,
        \begin{equation*}
            \Mcal_{G} \subset \widehat{\Mcal} \subset (\Mcal_{G})_{c}.
        \end{equation*}
        In particular, since $R_n \longrightarrow +\infty$ and $\Mcal_{G}$ is compact, up to increasing $n_0$, on this event,
        \begin{equation*}
            \Mcal_{G} \subset \widehat{\Mcal} \cap \bar{B}(0,R_n) \subset (\Mcal_{G})_{c}.
        \end{equation*}
\end{itemize}
\end{lemma}

In the rest of the proof of the Theorem, we lighten the notation $\widehat{\Mcal} \cap \bar{B}(0,R_n)$ into $\widehat{\Mcal}$ (equivalently, we redefine the estimator $\hat{\Mcal}$ as the intersection of the estimator of Section~\ref{subsec:Mupper} with the closed euclidean ball of radius $R_n$).

\begin{lemma}
\label{lemma:bargfaraway}
Let $G$ be a probability measure with compact support $\Mcal_{G}$. Assume $G \in St_{\Mcal_{G}}(a,d,r_0)$ for some constants $a>0$, $d>0$ and $r_0 >0$.
Then for any $\alpha > 0$, $c > 0$ and $p \in [1,+\infty)$, there exists $C(\alpha, c) > 0$ such that, on the event where
\begin{equation*}
    \Mcal_{G} \subset \widehat{\Mcal} \subset (\Mcal_{G})_{c},
\end{equation*}
it holds
\begin{equation*}
\| \bar{g} \|_{L_1{(\Rbb^D \setminus (\widehat{\Mcal})_{c})}}
	\leq C(\alpha,c) h_n^{\alpha}
	    \quad \text{and} \quad
	\int_{\Rbb^D \setminus (\widehat{\Mcal})_{c}} \|x\|^p |\bar{g}(x)| dx \leq  C(\alpha,c) h_n^\alpha.
\end{equation*}
\end{lemma}

\begin{proof}
By definition,
\begin{equation*}
\| \bar{g} \|_{L_1(\Rbb^D \setminus (\widehat{\Mcal})_{c})}
    = \frac{1}{h_n^D} \int_{x \in \Mcal_{G}} \int_{y \in \Rbb^D \setminus (\widehat{\Mcal})_{c}} \psi_A\left( \frac{\|y - x\|_2}{h_n}\right) dy dG(x).
\end{equation*}
By~\eqref{lim}, for any $A>0$, there exists $C>0$ such that for any $x \in \Mcal_{G}$ and $y \in \Rbb^D \setminus (\widehat{\Mcal})_{c}$, 
\begin{align*}
\psi_{A}\left( \frac{\|y - x\|_2}{h_n}\right)
    \leq C \exp\left(- \beta_{A} \frac{\|y-x\|_2^{A/(A+1)}}{h_n^{A/(A+1)}}\right)
    &\leq C \exp\left(- \beta_{A} \frac{ d(y,\Mcal_{G})^{A/(A+1)}}{h_n^{A/(A+1)}}\right).
\end{align*}
Since $\Mcal_{G} \subset \widehat{\Mcal}$, for all $y \in \Rbb^D \setminus (\widehat{\Mcal})_{c}$, $d(y,\Mcal_{G}) \geq c$, so for any $\alpha > 0$, there exists a constant $\tilde{C}>0$ such that
\begin{equation*}
C \exp\left(- \beta_{A} \frac{ d(y,\Mcal_{G})^{A/(A+1)}}{h_n^{A/(A+1)}}\right)
	\leq \tilde{C} \frac{h_n^{D + \alpha}}{d(y,\Mcal_{G})^{D + \alpha}}.
\end{equation*}
Moreover, since $\Mcal_{G}$ is compact, $\Diam(\Mcal_{G})$ is finite, so that on the event where $\Mcal_G \subset \widehat{\Mcal}$,
\begin{align*}
\int_{\Rbb^D \setminus (\widehat{\Mcal})_{c}} \frac{1}{d(y,\Mcal_{G})^{D+\alpha}} dy 
	&\leq \int_{\Rbb^D \setminus (\Mcal_{G})_{c}} \frac{1}{d(y,\Mcal_{G})^{D+\alpha}} dy
	< \infty. \\
	& \leq \int_{\Rbb^D} \left(\frac{1}{c \vee (\|y\| - \Diam(\Mcal_{G})/2)}\right)^{D+\alpha} dy < \infty.
\end{align*}
Therefore, for all $c > 0$ and $\alpha > 0$, there exists $C$ depending on $A$, $D$, $c$, $\alpha$ and $\Diam(\Mcal_G)$ such that
\begin{equation*}
\| \bar{g} \|_{L_1(\Rbb^D \setminus (\widehat{\Mcal})_{c})}
    \leq C h_n^{\alpha}.
\end{equation*}
The proof that the same holds for $\int_{\Rbb^D \setminus (\widehat{\Mcal})_{c}} \|x\|^p \bar{g}(x) dx $ is similar.
\end{proof}

Let $G \in St_{\Mcal_{G}}(a,d,r_0)$ be such that if $X \sim G$, then $\Phi_X \in \Hcal \cap \Upsilon_{1,S}$. Fix $p \in [1, +\infty)$. We use a bias-variance decomposition of $W_p(G,\widehat{P}_{n, \eta})$ through the triangle inequality
\begin{equation*}
    W_p(G, \widehat{P}_{n,\eta}) \leq W_p(G,P_{\psi_{A,h}}) + W_p(P_{\psi_{A,h}},\widehat{P}_{n,\eta}).
\end{equation*}
The proof is done is several steps :

\begin{enumerate}
    \item[\textbf{(1)}] We first show that there exists $C>0$ depending only on $A$ and $D$ such that the bias satisfies
\begin{equation*}
    W_p(G, P_{\psi_{A,h_n}}) \leq C h_n.
\end{equation*}
    \item[\textbf{(2)}] We prove that for any $\alpha \geq 1$, on the event where
\begin{equation*}
    \Mcal_{G} \subset \widehat{\Mcal} \subset (\Mcal_{G})_{c},
\end{equation*}
there exists $C' > 0$ such that 
\begin{equation*}
\label{lemma:varw}
W_p(P_{\psi_{A,h_n}},\widehat{P}_{n,\eta}) 
	\leq C' (h_n^{\alpha} + \Gamma_n).
\end{equation*}
    \item[\textbf{(3)}] We show that the choice of the parameters $m_n$, $h_n$ and $\lambda_n$ gives the result.
\end{enumerate}

\paragraph*{Proof of (1)}

Let $Y_\psi$ be a random variable with density $\psi_{A,h_n}$ and independent of $X$, so that the distribution of $X+Y_{\psi}$ is $P_{\psi_{A,h_n}}$. By definition of $W_p$,
\begin{equation*}
W_p^p(G, P_{\psi_{A,h_n}})
	\leq \Ebb(\| X + Y_{\psi} - X \|^p_2)
	= \Ebb(\|Y_{\psi}\|^p_2) = h_{n}^p \int_{\Rbb^D} \|u\|^p \psi_{A,1}(u) du,
\end{equation*}
and the integral is finite by point (V) of the properties of $\psi_{A}$.

\paragraph*{Proof of (2)}

If $\nu$ and $\mu$ are probability measures on $\Rbb^D$ having respective densities $f$ and $g$ with respect to the Lebesgue measure, letting $\omega$ be the measure with density $\min(f,g)$ with respect to the Lebesgue measure, $a \in \Rbb^D$, and $\delta_a$ the Dirac measure in $a$, by convexity of $x \mapsto x^p$,
\begin{align*}
W_p^p(\mu,\nu)
    &\leq 2^{p-1} \left( W_p^p(\mu, \omega + (1-\omega(\Rbb^D)) \delta_a) + W_p^p(\nu, \omega + (1-\omega(\Rbb^D)) \delta_a) \right) \\
    &\leq 2^{p-1} \int_{\Rbb^D} \| x-a \|^p (f(x) - g(x) - 2\min(f(x),g(x))) dx \\
    &= 2^{p-1} \int_{\Rbb^D} \| x-a \|^p |f(x) - g(x)| dx,
\end{align*}
so that
\begin{equation}
\label{ineq:Villani}
    W_p^p(\mu,\nu) \leq 2^{p-1} \min_{a \in \Rbb^D} \int_{\Rbb^D} \| x-a \|^p |f(x) - g(x)| dx.
\end{equation}
This entails
\begin{align}
W_p^p(P_{\psi_{A,h_n}},\widehat{P}_{n,\eta}) 
    &\leq 2^{p-1} \min_{a \in \Rbb^D} \int_{\Rbb^D} \| x-a \|^p |\bar{g}(x) - c_n \widehat{g}^{+}_n(x) 1|_{(\widehat{\Mcal})_{\eta}})(x)| dx \nonumber \\
    \label{eq:varweq}
    &\leq 2^{p-1} \int_{(\widehat{\Mcal})_{\eta}} \|x\|^p  |\bar{g}(x) - c_n \widehat{g}^{+}_n(x)| dx
        + 2^{p-1} \int_{\Rbb^D \setminus (\widehat{\Mcal})_{\eta}} \|x\|^p \bar{g}(x) dx.
\end{align}
For $S$ compact subset of $\Rbb^D$, write $M_{S} = \sup_{x \in S} \|x\|^p$ and $\text{Vol}(S)$ for the Lebesgue measure of $S$, then
\begin{align*}
\int_{(\widehat{\Mcal})_{\eta}} \|x\|^p |\bar{g}(x) - c_n \widehat{g}^+_n(x)| dx 
    &\leq M_{(\widehat{\Mcal})_{\eta}} \int_{(\widehat{\Mcal})_{\eta}} |\widehat{g}^{+}_n(x) - \bar{g}(x)|dx
        + M_{(\widehat{\Mcal})_{\eta}} \frac{|c_n-1|}{c_n}\\
    & \leq M_{(\widehat{\Mcal})_{\eta}} \text{Vol}((\widehat{\Mcal})_{\eta}) \Gamma_n
        + M_{(\widehat{\Mcal})_{\eta}} \frac{|c_n-1|}{c_n}.
\end{align*}
We also have 
\begin{align*}
\frac{|c_n -1|}{c_n} = \left | \frac{1}{c_n} - 1 \right |
    &= \Bigg| \int_{(\widehat{\Mcal})_{\eta}} (\widehat{g}^{+}_n(y) - \bar{g}(y)) dy - \int_{\Rbb^D \setminus (\widehat{\Mcal})_{\eta}} \bar{g}(y) dy \Bigg| \\
    &\leq \| \widehat{g}_n - \bar{g} \|_{L_1((\widehat{\Mcal})_{\eta})} + \|\bar{g}\|_{L_1(\Rbb^D \setminus (\widehat{\Mcal})_{\eta})}.
\end{align*}
Using Hölder's inequality,
\begin{equation*}
\| \widehat{g}_n - \bar{g} \|_{L_1((\widehat{\Mcal})_{\eta})} \leq \text{Vol}((\widehat{\Mcal})_{\eta}) \Gamma_n.
\end{equation*}
By Lemma~\ref{lemma:bargfaraway}, for any $\alpha > 0$, there exists $C$ such that
\begin{equation*}
\int_{(\widehat{\Mcal})_{\eta}} \|x\|^p |\bar{g}(x) - c_n \widehat{g}^+_n(x)| dx 
	\leq 2 M_{(\widehat{\Mcal})_{\eta}} \text{Vol}((\widehat{\Mcal})_{\eta}) \Gamma_n 
		+ M_{(\widehat{\Mcal})_{\eta}} C h_n^{\alpha}.
\end{equation*}
For any $c>0$, when $\widehat{\Mcal} \subset (\Mcal_{G})_{c}$, one has $(\widehat{\Mcal})_{\eta} \subset (\Mcal_{G})_{\eta + c}$.
This inclusion entails $M_{(\widehat{\Mcal})_{\eta}} \leq M_{(\Mcal_{G})_{\eta + c}}$ and $\text{Vol}((\widehat{\Mcal})_{\eta}) \leq \text{Vol}((\Mcal_{G})_{\eta + c})$. Therefore, for any $c>0$,
\begin{equation}
\label{eq:proofvarw1}
\int_{(\widehat{\Mcal})_{\eta}} \|x\|^p |\bar{g}(x) - c_n \widehat{g}^+_n(x)| dx
	\leq 2 M_{(\Mcal_{G})_{\eta + c}} \text{Vol}((\Mcal_{G})_{\eta + c}) \Gamma_n
		+ M_{(\Mcal_{G})_{\eta + c}} C h_n^{\alpha}.
\end{equation}
Again by Lemma~\ref{lemma:bargfaraway}, on the event where $\Mcal_{G} \subset \widehat{\Mcal} \subset (\Mcal_{G})_{\eta}$, 
\begin{equation}
\label{eq:proofvarw2}
\int_{\Rbb^D \setminus (\widehat{\Mcal})_{\eta}} \|x\|^p |\bar{g}(x)| dx \leq C' h_n^\alpha.
\end{equation}
Finally, using \eqref{eq:varweq}, \eqref{eq:proofvarw1} and \eqref{eq:proofvarw2}, for any $\alpha \geq 1$, there exists $C>0$ such that
\begin{equation*}
W_p^p(P_{\psi_{A,h_n}},\widehat{P}_{n,\eta})
	\leq C (h_n^{\alpha} + \Gamma_n).
\end{equation*}

\paragraph*{Proof of (3)}

Using \textbf{(1)} and \textbf{(2)}, for sequences $h_n$, $m_n$ and $\lambda_n$ satisfying the assumptions of Theorem~\ref{theorem:upperboundwasserstein}, on the event where $\Mcal_{G} \subset \widehat{\Mcal} \subset (\Mcal_{G})_{\eta}$, for any $\alpha \geq p$, there exists $C > 0$ such that
\begin{equation*}
W_p(G, \widehat{P}_{n,\eta})
	\leq C (h_n + (h_n^{\alpha} + \Gamma_n)^{1/p})
	\leq 2C (h_n + \Gamma_n^{1/p}).
\end{equation*}
We may assume $h_n \leq 1$ for all $n$ without loss of generality.
As stated in Lemma~\ref{Lemma:encadrementMcal}, for any $\delta' > 0$, there exist $C'$ and $n_0$ 
such that for all $n\geq n_0$, with probability at least $1 - 2\exp(-n^{1/2 - \delta'})$, $\Gamma_n^{1/p} \leq C' e^{-m_n/(2p)}$ and $\Mcal_{G} \subset \widehat{\Mcal} \subset (\Mcal_{G})_{\eta}$, and therefore
\begin{align*}
W_p(G, \widehat{P}_{n,\eta^{\star}})
    &\leq C m_n^{-1}
\end{align*}
on this event, up to changing the constant $C$.

On the event of probability at most $2\exp(-n^{1/2-\delta'})$ where this does not hold, since the support of $\widehat{P}_{n,\eta^{\star}}$ is a subset of $\bar{B}(0,R_n)$, $W_p(G, \widehat{P}_{n,\eta^{\star}}) \leq 2R_n$.

Therefore, taking $\delta' < \delta$ where $\delta$ is as defined in the statement of the Theorem, there exists $C>0$ 
such that for $n\geq n_0$,
\begin{equation*}
    \Ebb_{(G * \Qbb)^{\otimes n}}[W_p(G,\widehat{P}_{n,\eta^{\star}})] \leq C m_n^{-1},
\end{equation*}
which concludes the proof.

\subsection{Proof of Theorem~\ref{theorem:lowerboundwasserstein}}
\label{proof:lowerboundW}

Let $\widehat{P}_n$ be an estimator of $G$. According to~\cite{MR1462963},
\begin{equation*}
\underset{\Qbb \in \Qcal^{(D)} (\nu,c({\nu}),E)}
    {\sup_{G \in St_{\Kcal}(a,d,r_0)\cap \Lcal(1,S,{\mathcal{H}_1^{\star}})}}
    \Ebb_{(G * \Qbb)^{\otimes n}}[W_p(G,\widehat{P}_n)] \geq  \frac{1}{2} W_p(G_0(1), G_1(1))  (1 - \| G_0(1) \ast Q) - G_1(1) \ast Q \|_1)^{n}.
\end{equation*}
Using the same two distributions $G_0(1)$, $G1(1)$ and the same set $\mathcal{H}_1^{\star}$ as in Theorem~\ref{thm:lower}. We have shown in Theorem~\ref{thm:lower} that there exists a constant $C>0$ such that
\begin{equation*}
    \| G_0(1) \ast Q - G_1(1) \ast Q \|_{TV} \leq \frac{C}{n},
\end{equation*}
taking $\gamma$ of the form $c \log(n)^{-1-\delta}$ for any $\delta > 0$ and $c$ small enough, which implies that the minimax risk is lower bounded by $W_p(G_0(1), G_1(1))$.
We show that there exist constants $c>0$ and $n_0>0$ such that for $n \geq n_0$
\begin{equation*}
    W_p(G_0(1),G_1(1)) \geq c \gamma.
\end{equation*}
Let $\mathcal{U}_\gamma$ be the set of $u \in \Rbb$ such that $|\cos(\frac{u}{\gamma})| \geq 1/2$, that is $\mathcal{U}_\gamma = \bigcup_{k \in \Zbb} [k\pi \gamma - \frac{\pi \gamma}{3}, k\pi \gamma + \frac{\pi \gamma}{3}]$. For each $k \in \Zbb$, let $I_{k,\gamma} := [k\pi \gamma - \frac{\pi \gamma}{2}, k\pi \gamma + \frac{\pi \gamma}{2}]$.
Let us also define, for any two sets $A$ and $B$ of $\Rbb^d$, $d(A,B) = \inf_{x \in A, y \in B} \|x-y\|_2$.
We first show that
\begin{equation*}
d(M_0(\gamma) \cap (\mathcal{U}_\gamma \times \Rbb^{D-1}), M_1(\gamma)) \geq \gamma (\frac{1}{4 \sqrt{2}} \wedge \frac{\pi}{6}).
\end{equation*}
Let $x \in M_0(\gamma) \cap (\mathcal{U}_\gamma \times \Rbb^{D-1})$ and $y \in M_1(\gamma)$. There exists $k \in \Zbb$ such that $x \in M_0(\gamma) \cap ((\mathcal{U}_\gamma \cap I_{k,\gamma}) \times \Rbb^{D-1})$. If $y \in I_{k,\gamma} \times \Rbb^{D-1}$ (that is, if the first coordinate of $x$ and $y$ are in the same interval $I_{k,\gamma}$), then
\begin{equation*}
    \|x-y\|_2 \geq d(M_0(\gamma) \cap ((\mathcal{U}_\gamma \cap I_{k,\gamma}) \times \Rbb^{D-1}),M_1(\gamma) \cap (I_{k,\gamma} \times \Rbb^{D-1})).
\end{equation*}
All points of $M_0(\gamma)$ are of the form $(u, u + \frac{1}{2}\gamma \cos(\frac{u}{\gamma}), 0, \dots 0)^\top$ and the distance between $(u, u + \frac{1}{2}\gamma \cos(\frac{u}{\gamma}), 0, \dots, 0)^\top$ and the diagonal defined by $\mathcal{D} := \{(u, u,0,\dots,0)^\top : u \in \Rbb\}$ is $\frac{1}{4\sqrt{2}} \gamma |\cos(\frac{u}{\gamma})|$.
Since the sets $M_0(\gamma) \cap ((\mathcal{U}_\gamma \cap I_{k,\gamma}) \times \Rbb^{D-1})$ and $M_1(\gamma) \cap (I_{k,\gamma} \times \Rbb^{D-1})$ are on opposite sides of the diagonal $\mathcal{D}$,
\begin{align*}
d(M_0(\gamma) \cap ((\mathcal{U}_\gamma \cap I_{k,\gamma}) \times \Rbb^{D-1}),M_1(\gamma) \cap (I_{k,\gamma} \times \Rbb^{D-1}))
    &\geq d(M_0(\gamma) \cap ((\mathcal{U}_\gamma \cap I_{k,\gamma}) \times \Rbb^{D-1}), \mathcal{D})\\
    &= \frac{1}{4 \sqrt{2}} \gamma, 
\end{align*}
so that $\|x - y\|_2 \geq \frac{1}{4 \sqrt{2}} \gamma$.
If now $y \notin I_{k,\gamma} \times \Rbb^{D-1}$,
\begin{align*}
d(M_0(\gamma) \cap ((\mathcal{U}_\gamma \cap I_{k,\gamma}) \times \Rbb^{D-1}), M_1(\gamma) \cap ((\Rbb \setminus I_{k,\gamma}) \times \Rbb^{D-1}))
    &\geq d(\mathcal{U}_\gamma \cap I_{k,\gamma}, \Rbb \setminus I_{k,\gamma}) \\
    &= \frac{\pi \gamma}{6},
\end{align*} 
so that $\|x-y\|_2 \geq \frac{\pi \gamma}{6}$, and thus $d(M_0(\gamma) \cap (\mathcal{U}_\gamma \times \Rbb^{D-1}), M_1(\gamma)) \geq \gamma (\frac{1}{4 \sqrt{2}} \wedge \frac{\pi}{6})$.

\bigskip
Now, let us show that $W_p(G_0(1),G_1(1)) \geq \gamma (\frac{1}{8 \sqrt{2}} \wedge \frac{\pi}{12})$.
Let $\pi$ be a transport plan between $G_0(1)$ and $G_1(1)$, then
\begin{align*}
\int_{M_0(\gamma) \times M_1(\gamma)} \|x-y\|_2^p d\pi(x,y)
    &\geq \int_{M_0(\gamma) \cap (\mathcal{U}_\gamma \times \Rbb^{D-1}) \times M_1(\gamma)} \|x-y\|_2^p d\pi(x,y) \\
    &\geq d(M_0(\gamma) \cap (\mathcal{U}_\gamma \times \Rbb^{D-1}), M_1(\gamma))^p \  \pi(M_0(\gamma) \cap (\mathcal{U}_\gamma \times \Rbb^{D-1}) \times M_1(\gamma)) \\
    &= d(M_0(\gamma) \cap (\mathcal{U}_\gamma \times \Rbb^{D-1}), M_1(\gamma))^p \  G_0(1)(M_0(\gamma) \cap (\mathcal{U}_\gamma \times \Rbb^{D-1})) \\
    &= d(M_0(\gamma) \cap (\mathcal{U}_\gamma \times \Rbb^{D-1}), M_1(\gamma))^p \  \Pbb[U(1) \in \mathcal{U}_{\gamma}]
\end{align*}
since $G_1(1)$ has support $M_1(\gamma)$ and by definition of $G_0(1)$.
Therefore, by taking the infimum on all transport plans between $G_0(1)$ and $G_1(1)$,
\begin{equation*}
W_p(G_0(1),G_1(1)) \geq \gamma (\frac{1}{4 \sqrt{2}} \wedge \frac{\pi}{6}) \Pbb[U(1) \in \mathcal{U}_{\gamma}]^{1/p}.
\end{equation*}
$U(1)$ admits a density $f_1$ with respect to Lebesgue measure that is supported on $[-1,1]$ and continuous. Let us write $\omega$ one of its modulus of continuity. We have
\begin{align*}
\Pbb[U(1) \in \mathcal{U}_{\gamma}]
    &= \int_{[-1,1]} f_1(x) 1|_{\mathcal{U}_\gamma}(x) dx \\
    &= \sum_{k \in [-1/(\pi\gamma), 1/(\pi\gamma)]} \int_{[k\pi \gamma - \frac{\pi \gamma}{3}, k\pi \gamma + \frac{\pi \gamma}{3}]} f_1(x) dx \\
    &\leq \sum_{k \in [-1/(\pi\gamma), 1/(\pi\gamma)]} \left( \int_{[k\pi \gamma - \frac{\pi \gamma}{3}, k\pi \gamma + \frac{\pi \gamma}{3}]} f_1(k\pi \gamma) dx + \frac{2\pi\gamma}{3} \omega(\pi\gamma/3) \right) \\
    &\leq \sum_{k \in [-1/(\pi\gamma), 1/(\pi\gamma)]} \frac{2}{3} \int_{[k\pi \gamma - \frac{\pi \gamma}{2}, k\pi \gamma + \frac{\pi \gamma}{2}]} f_1(k\pi \gamma) dx + \frac{3}{\pi \gamma} \frac{2\pi\gamma}{3} \omega(\pi\gamma/3) \\
    &\leq \sum_{k \in [-1/(\pi\gamma), 1/(\pi\gamma)]} \frac{2}{3} \int_{[k\pi \gamma - \frac{\pi \gamma}{2}, k\pi \gamma + \frac{\pi \gamma}{2}]} f_1(x) dx + \frac{3}{\pi \gamma} \left(\frac{2\pi\gamma}{3} \omega(\pi\gamma/3) + \pi \gamma \omega(\pi\gamma/2)\right) \\
    &\leq \frac{2}{3} \int_{[-1,1]} f_1(x) dx + 3 \left( \frac{2}{3} \omega(\pi\gamma/3) + \omega(\pi\gamma/2) \right) \\
    &\underset{\gamma \rightarrow 0}{\longrightarrow} \frac{2}{3} \int_{[-1,1]} f_1(x) dx = \frac{2}{3}.
\end{align*}
Therefore, there exists $n_0$ such that for all $n \geq n_0$, $W_p(G_0(1),G_1(1)) \geq \gamma (\frac{1}{8 \sqrt{2}} \wedge \frac{\pi}{12})$.



\printbibliography 

@inproceedings{Aapo,
  title        = {Disentangling Identifiable Features from Noisy Data with Structured Nonlinear ICA},
  author       = {Hälvä, H. and Le Corff, S. and Lehéricy, L. and So, J. and  Zhu, Y. and Gassiat, \'{E}. and  Hyvarinen, A.},
  year         = {2021},
  booktitle    = {Advances in Neural Information Processing Systems (NeurIPS)},
  }

@misc{Repeated,
  author = {Capitao-Miniconi, Jérémie and Gassiat, \'{E}lisabeth and Lehéricy, Luc},
  title = {Repeated measurements deconvolution corrupted with unknown noise},
  howpublished = {preprint},
  year = {2024},
}

@article {Federer59,
    AUTHOR = {Federer, Herbert},
     TITLE = {Curvature measures},
   JOURNAL = {Trans. Amer. Math. Soc.},
  FJOURNAL = {Transactions of the American Mathematical Society},
    VOLUME = {93},
      YEAR = {1959},
     PAGES = {418--491},
      ISSN = {0002-9947},
   MRCLASS = {53.00},
  MRNUMBER = {110078},
MRREVIEWER = {C. B. Allendoerfer},
       DOI = {10.2307/1993504},
       URL = {https://doi-org.ezproxy.universite-paris-saclay.fr/10.2307/1993504},
}

@article {LCGL2021,
      AUTHOR = {Gassiat, \'{E}lisabeth and Le Corff, Sylvain and Leh\'{e}ricy, Luc},
     TITLE = {Deconvolution with unknown noise distribution is possible for
              multivariate signals},
   JOURNAL = {Ann. Statist.},
  FJOURNAL = {The Annals of Statistics},
    VOLUME = {50},
      YEAR = {2022},
    NUMBER = {1},
     PAGES = {303--323},
}

@article {LCGL2021supplementary,
      AUTHOR = {Gassiat, \'{E}lisabeth and Le Corff, Sylvain and Leh\'{e}ricy, Luc},
     TITLE = {Supplementary to ``Deconvolution with unknown noise distribution is possible for
              multivariate signals''},
                 JOURNAL = {Ann. Statist.},
  FJOURNAL = {The Annals of Statistics},
        YEAR = {2022},
    URL = {https://doi.org/10.1214/21-AOS2106SUPP},
}

@article {GPVW12,
    AUTHOR = {Genovese, Christopher R. and Perone-Pacifico, Marco and
              Verdinelli, Isabella and Wasserman, Larry},
     TITLE = {Manifold estimation and singular deconvolution under
              {H}ausdorff loss},
   JOURNAL = {Ann. Statist.},
  FJOURNAL = {The Annals of Statistics},
    VOLUME = {40},
      YEAR = {2012},
    NUMBER = {2},
     PAGES = {941--963},
      ISSN = {0090-5364},
   MRCLASS = {62G05 (62G20 62H11 62H12)},
  MRNUMBER = {2985939},
MRREVIEWER = {Arash Ali Amini},
       DOI = {10.1214/12-AOS994},
       URL = {https://doi-org.revues.math.u-psud.fr/10.1214/12-AOS994},
     }

@article {MR1574180,
    AUTHOR = {Morgan, G. W.},
     TITLE = {A {N}ote on {F}ourier {T}ransforms},
   JOURNAL = {J. London Math. Soc.},
  FJOURNAL = {The Journal of the London Mathematical Society},
    VOLUME = {9},
      YEAR = {1934},
    NUMBER = {3},
     PAGES = {187--192},
      ISSN = {0024-6107},
   MRCLASS = {DML},
  MRNUMBER = {1574180},
       DOI = {10.1112/jlms/s1-9.3.187},
       URL = {https://doi-org.ezproxy.universite-paris-saclay.fr/10.1112/jlms/s1-9.3.187},
}

@incollection {MR1462963,
    AUTHOR = {Yu, Bin},
     TITLE = {Assouad, {F}ano, and {L}e {C}am},
 BOOKTITLE = {Festschrift for {L}ucien {L}e {C}am},
     PAGES = {423--435},
 PUBLISHER = {Springer, New York},
      YEAR = {1997},
   MRCLASS = {62G10 (62G05 62G20)},
  MRNUMBER = {1462963},
MRREVIEWER = {M. P. Moklyachuk},
}

@book {MR3837127,
    AUTHOR = {Boissonnat, Jean-Daniel and Chazal, Fr\'{e}d\'{e}ric and Yvinec,
              Mariette},
     TITLE = {Geometric and topological inference},
    SERIES = {Cambridge Texts in Applied Mathematics},
 PUBLISHER = {Cambridge University Press, Cambridge},
      YEAR = {2018},
     PAGES = {xii+233},
      ISBN = {978-1-108-41089-2; 978-1-108-41939-0},
   MRCLASS = {62-01 (52-01 55-01 55N35 57R05 60D05 62H11)},
  MRNUMBER = {3837127},
MRREVIEWER = {Henry Hugh Adams},
       DOI = {10.1017/9781108297806},
       URL = {https://doi-org.ezproxy.universite-paris-saclay.fr/10.1017/9781108297806},
}

@article {MR3189324,
    AUTHOR = {Dedecker, J\'{e}r\^{o}me and Michel, Bertrand},
     TITLE = {Minimax rates of convergence for {W}asserstein deconvolution
              with supersmooth errors in any dimension},
   JOURNAL = {J. Multivariate Anal.},
  FJOURNAL = {Journal of Multivariate Analysis},
    VOLUME = {122},
      YEAR = {2013},
     PAGES = {278--291},
      ISSN = {0047-259X},
   MRCLASS = {62G05 (62C20)},
  MRNUMBER = {3189324},
MRREVIEWER = {Zailong Wang},
       DOI = {10.1016/j.jmva.2013.08.009},
       URL = {https://doi-org.ezproxy.universite-paris-saclay.fr/10.1016/j.jmva.2013.08.009},
}

@article {MR4255215,
    AUTHOR = {Brunel, Victor-Emmanuel and Klusowski, Jason M. and Yang,
              Dana},
     TITLE = {Estimation of convex supports from noisy measurements},
   JOURNAL = {Bernoulli},
  FJOURNAL = {Bernoulli. Official Journal of the Bernoulli Society for
              Mathematical Statistics and Probability},
    VOLUME = {27},
      YEAR = {2021},
    NUMBER = {2},
     PAGES = {772--793},
      ISSN = {1350-7265},
   MRCLASS = {62G05 (62G20)},
  MRNUMBER = {4255215},
MRREVIEWER = {Ulrich Stadtm\"{u}ller},
       DOI = {10.3150/20-bej1229},
       URL = {https://doi-org.ezproxy.universite-paris-saclay.fr/10.3150/20-bej1229},
}

@misc{AFischer,
  title={Estimation via length-constrained generalized empirical principal curves under small noise},
  author={Delattre, Sylvain and Fischer, Aur{\'e}lie},
  year={2019},
  eprint={1911.06728},
  archivePrefix={arXiv},
  primaryClass={math.ST}
}

@article {MR3909931,
    AUTHOR = {Aamari, Eddie and Levrard, Cl\'{e}ment},
     TITLE = {Nonasymptotic rates for manifold, tangent space and curvature
              estimation},
   JOURNAL = {Ann. Statist.},
  FJOURNAL = {The Annals of Statistics},
    VOLUME = {47},
      YEAR = {2019},
    NUMBER = {1},
     PAGES = {177--204},
      ISSN = {0090-5364},
   MRCLASS = {62G05 (62C20)},
  MRNUMBER = {3909931},
MRREVIEWER = {Michael Stolz},
       DOI = {10.1214/18-AOS1685},
       URL = {https://doi-org.ezproxy.universite-paris-saclay.fr/10.1214/18-AOS1685},
}

@article {MR4421180,
    AUTHOR = {Divol, Vincent},
     TITLE = {Measure estimation on manifolds: an optimal transport
              approach},
   JOURNAL = {Probab. Theory Related Fields},
  FJOURNAL = {Probability Theory and Related Fields},
    VOLUME = {183},
      YEAR = {2022},
    NUMBER = {1-2},
     PAGES = {581--647},
      ISSN = {0178-8051},
   MRCLASS = {62G05 (49Q22 62C20)},
  MRNUMBER = {4421180},
       DOI = {10.1007/s00440-022-01118-z},
       URL = {https://doi-org.ezproxy.universite-paris-saclay.fr/10.1007/s00440-022-01118-z},
}

@article{AaronBoundary,
  title={Minimax boundary estimation and estimation with boundary},
  author={Aamari, Eddie and Aaron, Catherine and Levrard, Cl{\'e}ment},
  journal={Bernoulli},
  volume={29},
  number={4},
  pages={3334--3368},
  year={2023},
  publisher={Bernoulli Society for Mathematical Statistics and Probability}
}

@article {MR2298884,
    AUTHOR = {Meister, Alexander},
     TITLE = {Estimating the support of multivariate densities under
              measurement error},
   JOURNAL = {J. Multivariate Anal.},
  FJOURNAL = {Journal of Multivariate Analysis},
    VOLUME = {97},
      YEAR = {2006},
    NUMBER = {8},
     PAGES = {1702--1717},
      ISSN = {0047-259X},
   MRCLASS = {62H12 (62G07)},
  MRNUMBER = {2298884},
MRREVIEWER = {Jussi S. Klemel\"{a}},
       DOI = {10.1016/j.jmva.2005.04.004},
       URL = {https://doi-org.ezproxy.universite-paris-saclay.fr/10.1016/j.jmva.2005.04.004},
}

@article {MR579432,
    AUTHOR = {Devroye, Luc and Wise, Gary L.},
     TITLE = {Detection of abnormal behavior via nonparametric estimation of
              the support},
   JOURNAL = {SIAM J. Appl. Math.},
  FJOURNAL = {SIAM Journal on Applied Mathematics},
    VOLUME = {38},
      YEAR = {1980},
    NUMBER = {3},
     PAGES = {480--488},
      ISSN = {0036-1399},
   MRCLASS = {62G05 (62G10)},
  MRNUMBER = {579432},
MRREVIEWER = {Paul Deheuvels},
       DOI = {10.1137/0138038},
       URL = {https://doi-org.ezproxy.universite-paris-saclay.fr/10.1137/0138038},
}

@article {MR1604449,
    AUTHOR = {Cuevas, Antonio and Fraiman, Ricardo},
     TITLE = {A plug-in approach to support estimation},
   JOURNAL = {Ann. Statist.},
  FJOURNAL = {The Annals of Statistics},
    VOLUME = {25},
      YEAR = {1997},
    NUMBER = {6},
     PAGES = {2300--2312},
      ISSN = {0090-5364},
   MRCLASS = {62G05},
  MRNUMBER = {1604449},
       DOI = {10.1214/aos/1030741073},
       URL = {https://doi-org.ezproxy.universite-paris-saclay.fr/10.1214/aos/1030741073},
}

@article {InferSphere,
    AUTHOR = {Capitao-Miniconi, J\'{e}r\'{e}mie and Gassiat, \'{E}lisabeth},
     TITLE = {Deconvolution of spherical data corrupted with unknown noise},
   JOURNAL = {Electron. J. Stat.},
  FJOURNAL = {Electronic Journal of Statistics},
  volume={17},
  number={1},
  pages={607--649},
  year={2023},
  publisher={The Institute of Mathematical Statistics and the Bernoulli Society}
}

@article {BumpFunction,
    AUTHOR = {Tlas, T.},
     TITLE = {Bump functions with monotone {F}ourier transforms satisfying
              decay bounds},
   JOURNAL = {J. Approx. Theory},
  FJOURNAL = {Journal of Approximation Theory},
    VOLUME = {278},
      YEAR = {2022},
     PAGES = {Paper No. 105742, 6},
      ISSN = {0021-9045},
   MRCLASS = {42A38},
  MRNUMBER = {4398414},
       DOI = {10.1016/j.jat.2022.105742},
       URL = {https://doi-org.ezproxy.universite-paris-saclay.fr/10.1016/j.jat.2022.105742},
}

@book {MR2568219,
    AUTHOR = {Gunning, Robert C. and Rossi, Hugo},
     TITLE = {Analytic functions of several complex variables},
      NOTE = {Reprint of the 1965 original},
 PUBLISHER = {AMS Chelsea Publishing, Providence, RI},
      YEAR = {2009},
     PAGES = {xiv+318},
      ISBN = {978-0-8218-2165-7},
   MRCLASS = {32-01},
  MRNUMBER = {2568219},
       DOI = {10.1090/chel/368},
}

@book {MR1857292,
    AUTHOR = {Ambrosio, Luigi and Fusco, Nicola and Pallara, Diego},
     TITLE = {Functions of bounded variation and free discontinuity
              problems},
    SERIES = {Oxford Mathematical Monographs},
 PUBLISHER = {The Clarendon Press, Oxford University Press, New York},
      YEAR = {2000},
     PAGES = {xviii+434},
      ISBN = {0-19-850245-1},
   MRCLASS = {49-02 (49J45 49K10 49Qxx)},
  MRNUMBER = {1857292},
MRREVIEWER = {J. E. Brothers},
}

@article {MR2543590,
    AUTHOR = {Goldenshluger, Alexander and Lepski, Oleg},
     TITLE = {Universal pointwise selection rule in multivariate function
              estimation},
   JOURNAL = {Bernoulli},
  FJOURNAL = {Bernoulli. Official Journal of the Bernoulli Society for
              Mathematical Statistics and Probability},
    VOLUME = {14},
      YEAR = {2008},
    NUMBER = {4},
     PAGES = {1150--1190},
      ISSN = {1350-7265},
   MRCLASS = {62G05},
  MRNUMBER = {2543590},
MRREVIEWER = {Michael Falk},
       DOI = {10.3150/08-BEJ144},
       URL = {https://doi-org.ezproxy.universite-paris-saclay.fr/10.3150/08-BEJ144},
}

@misc{AizenbudSober21,
  title={Non-Parametric Estimation of Manifolds from Noisy Data},
  author={Aizenbud, Yariv and Sober, Barak},
  year={2021},
  eprint={2105.04754},
  archivePrefix={arXiv},
  primaryClass={math.ST}
}

@article {MR4420765,
    AUTHOR = {Puchkin, Nikita and Spokoiny, Vladimir},
     TITLE = {Structure-adaptive manifold estimation},
   JOURNAL = {J. Mach. Learn. Res.},
  FJOURNAL = {Journal of Machine Learning Research (JMLR)},
    VOLUME = {23},
      YEAR = {2022},
     PAGES = {Paper No. [40], 62},
      ISSN = {1532-4435},
   MRCLASS = {62R30 (62M45)},
  MRNUMBER = {4420765},
}

@article {MR3314482,
    AUTHOR = {Dedecker, J\'{e}r\^{o}me and Fischer, Aur\'{e}lie and Michel, Bertrand},
     TITLE = {Improved rates for {W}asserstein deconvolution with ordinary
              smooth error in dimension one},
   JOURNAL = {Electron. J. Stat.},
  FJOURNAL = {Electronic Journal of Statistics},
    VOLUME = {9},
      YEAR = {2015},
    NUMBER = {1},
     PAGES = {234--265},
   MRCLASS = {62G05 (62C20)},
  MRNUMBER = {3314482},
       DOI = {10.1214/15-EJS997},
       URL = {https://doi-org.ezproxy.universite-paris-saclay.fr/10.1214/15-EJS997},
}

@article {MR1545791,
    AUTHOR = {Marcinkiewicz, J.},
     TITLE = {Sur une propri\'{e}t\'{e} de la loi de {G}au\ss },
   JOURNAL = {Math. Z.},
  FJOURNAL = {Mathematische Zeitschrift},
    VOLUME = {44},
      YEAR = {1939},
    NUMBER = {1},
     PAGES = {612--618},
      ISSN = {0025-5874},
   MRCLASS = {DML},
  MRNUMBER = {1545791},
       DOI = {10.1007/BF01210677},
       URL = {https://doi-org.ezproxy.universite-paris-saclay.fr/10.1007/BF01210677},
}

@article {MR4002890,
    AUTHOR = {Garc\'{i}a Trillos, Nicol\'{a}s and Sanz-Alonso, Daniel and Yang,
              Ruiyi},
     TITLE = {Local regularization of noisy point clouds: improved global
              geometric estimates and data analysis},
   JOURNAL = {J. Mach. Learn. Res.},
  FJOURNAL = {Journal of Machine Learning Research (JMLR)},
    VOLUME = {20},
      YEAR = {2019},
     PAGES = {Paper No. 136, 37},
      ISSN = {1532-4435},
   MRCLASS = {62-07 (62H12 62H30)},
  MRNUMBER = {4002890},
}

@article {MR4441130,
    AUTHOR = {Niles-Weed, Jonathan and Berthet, Quentin},
     TITLE = {Minimax estimation of smooth densities in {W}asserstein
              distance},
   JOURNAL = {Ann. Statist.},
  FJOURNAL = {The Annals of Statistics},
    VOLUME = {50},
      YEAR = {2022},
    NUMBER = {3},
     PAGES = {1519--1540},
      ISSN = {0090-5364},
   MRCLASS = {62F99 (62H99)},
  MRNUMBER = {4441130},
       DOI = {10.1214/21-aos2161},
       URL = {https://doi-org.ezproxy.universite-paris-saclay.fr/10.1214/21-aos2161},
}

@article {MR2859954,
    AUTHOR = {Chazal, Fr\'{e}d\'{e}ric and Cohen-Steiner, David and M\'{e}rigot,
              Quentin},
     TITLE = {Geometric inference for probability measures},
   JOURNAL = {Found. Comput. Math.},
  FJOURNAL = {Foundations of Computational Mathematics. The Journal of the
              Society for the Foundations of Computational Mathematics},
    VOLUME = {11},
      YEAR = {2011},
    NUMBER = {6},
     PAGES = {733--751},
      ISSN = {1615-3375},
   MRCLASS = {62G05 (28A33 55P99 62H11 68U05)},
  MRNUMBER = {2859954},
MRREVIEWER = {Antonio Cuevas},
       DOI = {10.1007/s10208-011-9098-0},
       URL = {https://doi-org.ezproxy.universite-paris-saclay.fr/10.1007/s10208-011-9098-0},
}

@misc{berenfeld2022estimating,
      title={Estimating a density near an unknown manifold: a Bayesian nonparametric approach}, 
      author={Cl\'{e}ment Berenfeld and Paul Rosa and Judith Rousseau},
      year={2022},
      eprint={2205.15717},
      archivePrefix={arXiv},
      primaryClass={math.ST}
}

@book {AnaComplexBook,
    AUTHOR = {Stein, Elias M. and Shakarchi, Rami},
     TITLE = {Complex analysis},
    SERIES = {Princeton Lectures in Analysis},
    VOLUME = {2},
 PUBLISHER = {Princeton University Press, Princeton, NJ},
      YEAR = {2003},
     PAGES = {xviii+379},
      ISBN = {0-691-11385-8},
   MRCLASS = {30-01},
  MRNUMBER = {1976398},
MRREVIEWER = {Heinrich\ Begehr},
}

\end{document}